\documentclass{article}
\usepackage[a4paper]{geometry}
\usepackage{amsmath}
\usepackage{amssymb}
\usepackage{amsthm}
\usepackage{mathrsfs}
\usepackage{mathtools}

\usepackage{tikz}
\usetikzlibrary{calc} 
\usepackage{hyperref}
\usepackage{cleveref}
\usepackage{tikz-cd}
\usepackage{mathtools}

\usepackage{placeins}

\usepackage{float}

\usepackage{comment}

\usepackage{enumerate}

\newtheorem{thm}{Theorem}[section]
\crefname{thm}{Theorem}{Theorems}
\newtheorem{cor}[thm]{Corollary}
\newtheorem{prop}[thm]{Proposition}
\crefname{prop}{Proposition}{Propositions}
\newtheorem{lem}[thm]{Lemma}
\crefname{lem}{Lemma}{Lemmas}
\newtheorem{clm}[thm]{Claim}
\newtheorem{conj}[thm]{Conjecture}

\newtheorem{defn}[thm]{Definition}

\crefname{defn}{Definition}{Definitions}

\newtheorem{rmk}[thm]{Remark}

\newtheorem{obs}[thm]{Observation}

\newtheorem*{ack*}{Acknowledgements}

\numberwithin{equation}{section}

\newcommand{\PP}{\mathbb{P}}
\newcommand{\EE}{\mathbb{E}}

\newcommand{\co}{\operatorname{co}}

\usepackage{titling}
\title{Sharp quantitative stability of the Brunn-Minkowski inequality}
\author{Alessio Figalli, Peter van Hintum, Marius Tiba}

\begin{document}

\maketitle

\begin{abstract}
    The Brunn-Minkowski inequality states that for bounded measurable sets $A$ and $B$ in $\mathbb{R}^n$, we have $|A+B|^{1/n} \geq |A|^{1/n}+|B|^{1/n}$. Also, equality holds if and only if $A$ and $B$ are convex and homothetic sets in $\mathbb{R}^d$. The stability of this statement is a well-known problem that has attracted much attention in recent years. This paper gives a conclusive answer by proving the sharp stability result for the Brunn-Minkowski inequality on arbitrary sets.
\end{abstract}

\setcounter{tocdepth}{3}
\tableofcontents

\section{Introduction}

Given measurable sets $X,Y\subset \mathbb{R}^n$ with positive measure, the Brunn-Minkowski inequality says that
$$|X+Y|^{\frac{1}{n}} \ge |X|^{\frac{1}{n}}+|Y|^{\frac{1}{n}}.$$ 
Alternatively, for equal sized measurable sets $A,B\subset \mathbb{R}^n$ and a parameter $t\in(0,1)$, this is equivalent to
$$|tA+(1-t)B|\geq |A|,$$
with equality for equal convex sets $A$ and $B$ (less a measure zero set). Here, $A+B=\{a+b\mid a\in A,\text{ and }b\in B\}$ is the \emph{Minkowski sum}, $tA:=\{ta: a\in A\}$, and $|\cdot|$ refers to the outer Lebesgue measure. 

The Brunn-Minkowski inequality is part of a vast body of geometric inequalities, such as the isoperimetric inequality, the Pr\'ekopa-Leindler inequality, and the Borell-Brascamb-Lieb inequality. The famous isoperimetric inequality, which states that for a given volume the body minimizing its perimeter is the ball, follows from Brunn-Minkowski by taking $A$ a ball and letting $t$ tend to zero. The Pr\'ekopa-Leindler inequality asserts that for $t\in(0,1)$ and functions $f,g,h\colon \mathbb{R}^n\to\mathbb{R}_{\geq 0}$ with the property that $h(tx+(1-t)y)\geq f^{t}(x)g^{1-t}(y)$ for all $x,y\in\mathbb{R}^n$ and $\int f=\int g$, we have $\int h\geq \int f$ with equality if and only if $f(x)=ag(x-x_0)$ is a log-concave function for some $a\in\mathbb{R}_{>0}$ and $x_0\in\mathbb{R}^n$. The Pr\'ekopa-Leindler inequality implies Brunn-Minkowski by taking $f$ and $g$ to be the indicator functions of $A$ and $B$. The Pr\'ekopa-Leindler inequality in turn is subsumed by the Borell-Brascamb-Lieb inequality. Studying these inequalities and their stabilities has sparked a fruitful field of research in recent years.


The stability of the Brunn-Minkowski inequality says that if we are \textit{close} to equality, then the sets are \textit{close} to being equal and convex (up to translates), and the aim is to quantify the two notions of closeness (see e.g. \cite{FigICM14}). The major folklore conjecture concerning the stability of the Brunn-Minkowski inequality is that if we are within a factor $1+\delta$ from equality, then the distance from $A$ and $B$ to a common convex set is $O_n(t^{-1/2}\delta^{1/2})$. 

\begin{conj}\label{squareroot}
For $n \in \mathbb{N}$, $n \geq 2$ and $t \in (0,1/2)$ there exist $c_{n}, d_{n,t}>0$ such that the following holds. Let $A$ and $B$ be measurable sets of equal size with $|tA+(1-t)B|\leq (1+\delta)|A|$ and $\delta<d_{n,t}$. Then there exists a convex set $K$ such that, up to translation, $K \supset A,B$ and
$$|K\setminus A|=|K\setminus B|\le c_{n}t^{-1/2}\delta^{1/2}|A|.$$
\end{conj}

Another important conjecture regarding the stability of the Brunn-Minkowski inequality is that  the distance from $A$ and $B$ to their individual convex hulls is linear $O_{n,t}(\delta)$. Figalli and Jerison \cite{figalli2021quantitative} formulated this conjecture in the case of equal sets, and van Hintum, Spink and Tiba \cite[Section 12]{planarBM} considered this conjecture for arbitrary sets in the plane.  

\begin{conj}\label{linear}
For $n \in \mathbb{N}$ and $t \in (0,1/2)$ there exist $c_{n}, d_{n,t}>0$ such that the following holds. Let $A$ and $B$ be measurable sets of equal size with $|tA+(1-t)B|\leq (1+\delta)|A|$ and $\delta<d_{n,t}$. Then 
$$|\co(A)\setminus A|+|\co(B)\setminus B|\le c_{n,t}\delta|A|.$$
\end{conj}

These conjectures have received a lot of attention becoming central problems in convex geometry (see e.g. \cite{Figalli09,figalli2010mass,christ2012planar,christ2012near,eldan2014dimensionality,figalli2015quantitative,figalli2015stability,figalli2017quantitative,Barchiesi,carlen2017stability,figalli2021quantitative,van2021sharp,van2023locality,SharpDelta,van2020sharp,planarBM}).

The first contribution to the study of stability was made by Freiman \cite{freiman1959addition} in dimension $n=1$.  Freiman's celebrated $3k-4$ Theorem \cite{freiman1959addition,lev1995addition,stanchescu1996addition} from additive combinatorics,
implies the following strong version of \Cref{linear}:
{If $t \in (0,1/2]$ and $A,B \subset \mathbb{R}$ are measurable sets with equal volume such that $|tA+(1-t)B| \leq (1+\delta)|A|$ with $\delta <t$, then $|\co(A)\setminus A| \leq t^{-1}\delta |A|$ and $|\co(B)\setminus B| \leq (1-t)^{-1}\delta|B|$.}
Simple examples show that this result is optimal. 

Stability in higher dimensions is considerably more difficult; in \cite{christ2012planar,christ2012near} Christ showed a qualitative result: If $n\in \mathbb{N}$, $t\in (0,1/2]$, and $A,B \subset \mathbb{R}^n$ are measurable sets with equal volume such that $|tA+(1-t)B| \leq (1+\delta)|A|$ with $\delta$ sufficiently small in terms of $t$ and $n$, then there exists a convex set $K$ such that, up to translation, $K \supset A,B$ and $|K\setminus A|=|K\setminus B| =o_{n,t, \delta}(1)|A|$, where $o_{n,t, \delta}(1) \rightarrow 0$ as $\delta \rightarrow 0$ for fixed $n$ and $t$. In a cornerstone result, Figalli and Jerison \cite{figalli2017quantitative} obtained the first quantitative bounds: $|K\setminus A|=|K\setminus B| \leq \delta^{(t/|\log(t)|)^{\exp(O(n))}}|A|$. A similar result for the Pr\'ekopa-Leindler inequality was recently established by B\"or\"ocky, Figalli, and Ramos \cite{boroczky2022quantitative}.

The only instance of \Cref{squareroot} for arbitrary sets was established in two dimensions by van Hintum, Spink, and Tiba \cite{planarBM}. In an independent direction, van Hintum and Keevash \cite{SharpDelta} (see also \Cref{SharpDelta}) determined the optimal value $d_{n,t}=t^n$ for all $n\in\mathbb{N}$ and $t\in(0,1/2]$, with the same bound on the distance to a common convex set as in the result of Figalli and Jerison.

Even partial answers to \Cref{squareroot} for restricted classes of sets $A$ and $B$ have received much attention. These papers have focused on controlling the weaker ``asymmetry'' measure $\inf_x|A\triangle (B+x)|$, which a priori does not control $|K\setminus A|$. In \cite{Figalli09, Figalli10amass}, Figalli, Maggi, and Pratelli established that given $n\in \mathbb{N}$, $t\in (0,1/2]$, and $A,B \subset \mathbb{R}^n$ convex sets with equal volume, if $|tA+(1-t)B| \leq (1+\delta)|A|$ with $\delta$ sufficiently small in terms of $t$ and $n$, then, up to translation, $|A\triangle B| \leq O_d(t^{-1/2}\delta^{1/2})|A|$. Figalli, Maggi, and Mooney \cite{Euclidean} showed the analogous result when $A$ is a ball and $B$ is arbitrary. Note that this is closely related to the stability of the isoperimetric inequality. Barchiesi and Julin \cite{Barchiesi} extended the previous results to $A$ convex and $B$ arbitrary. Despite all these results supporting \Cref{squareroot}, a conclusive answer remained wide open and outside the scope of the available techniques.

The particular case of equal sets $A=B$ in \Cref{linear} has been thoroughly investigated. Indeed, after establishing in \cite{figalli2015quantitative} some quantitative bounds for \Cref{linear} in all dimensions, Figalli and Jerison \cite{figalli2021quantitative} resolved \Cref{linear} for $A=B$ in dimensions $n=1,2,3$, and subsequently Van Hintum, Spink, and Tiba  \cite{van2021sharp} resolved \Cref{linear} for $A=B$ in all dimensions. Moreover, they determined the optimal dependency on $t$. Furthermore, van Hintum, Spink, and Tiba  \cite[Theorem 1.1]{van2020sharp} established the optimal dependency on $d$ in dimensions $d \leq 4$ when $A=B$ is the hypograph of a function over a convex domain. Another closely related result by van Hintum and Keevash \cite{van2023locality} is that if $A \subset \mathbb{R}^n$ with $|\frac{A+A}{2}|\leq (1+\delta)|A|$ with $\delta<1$, then there exists a set $A'\subset A$ with $|A'|\geq (1-\delta)|A|$ and $|\co(A')|=O_{n,1-\delta}(|A'|)$. 

For distinct sets $A$ and $B$, \Cref{linear} has proved much more difficult. Van Hintum, Spink, and Tiba  in \cite[Theorem 1.5]{van2020sharp}, resolved \Cref{linear}, when $A$ and $B$ are hypographs of functions over the same convex domain. The only instance of \Cref{linear} for arbitrary sets was established by van Hintum, Spink, and Tiba \cite[Section 12]{planarBM} in two dimensions. Despite these determined efforts, for arbitrary sets in higher dimensions \Cref{linear} remained open.

Our main results resolve the conjectured quadratic stability to a common convex hull and the conjectured linear stability to the individual convex hulls  in the Brunn-Minkowski inequality, concluding a long line of research on these problems.

\begin{thm}
\label{main_thm_1}
For all $n \in \mathbb{N}$ and $ t\in (0,1/2]$, there are computable constants $c_n^{\ref{main_thm_1}},d^{\ref{main_thm_1}}_{n,t}>0$ such that the following holds. Assume $\delta \in [0, d^{\ref{main_thm_1}}_{n,t})$ and let $A,B\subset \mathbb{R}^n$ be measurable sets with equal volume satisfying
$$ |tA+(1-t)B| =(1+\delta)|A|.$$
Then, up to translation\footnote{That is, there exist $x,y\in\mathbb{R}^n$ so that $x+A,y+B\subset K$ and $|K\setminus (x+A)|+|K\setminus (y+B)|\leq t^{-c^{\ref{main_thm_1}}n^8}\delta^{\frac{1}{2}}|A|$.}, there is a convex set $K\supset A \cup B$ such that
$$|K\setminus A|+|K\setminus B| \le c_n^{\ref{main_thm_1}}t^{-1/2}\delta^{1/2}|A|.$$
\end{thm}

\begin{thm}\label{LinearThmGeneral}
For $n\in\mathbb{N}$ and $0<t\leq \frac12$, there are constants $c^{\ref{LinearThmGeneral}},d^{\ref{LinearThmGeneral}}_{n,t}>0$ such that the following hold. Assume $\delta\in[0,d^{\ref{LinearThmGeneral}}_{n,t})$, and assume $A,B\subset\mathbb{R}^n$ are measurable sets of equal volume so that $|tA+(1-t)B|\leq (1+\delta)|A|$. Then
$$|\co(A)\setminus A|+|\co(B)\setminus B|\leq t^{-c^{\ref{LinearThmGeneral}}n^{8}}\delta |A|.$$ 
\end{thm}

The proofs of \Cref{main_thm_1} and \Cref{LinearThmGeneral} are very involved and will be obtained by combining a series of intermediate results, all of which have their own interest. Moreover, \Cref{main_thm_1} uses \Cref{LinearThmGeneral}.

We prove \Cref{main_thm_1} by first showing a sharp control of the symmetric difference between $A$ and $B$. 

\begin{thm}
\label{main_thm_6}
For all $n \in \mathbb{N}$ and $t \in (0,1/2]$, there are computable constants $c_n^{\ref{main_thm_6}},d^{\ref{main_thm_6}}_{n,t}>0$ such that the following holds. Assume $\delta \in [0,d^{\ref{main_thm_6}}_{n,t})$ and assume $A,B\subset \mathbb{R}^n$ are measurable sets with equal volume so that
$ |tA+(1-t)B| =(1+\delta)|A|.$ Then, up to translation,
$$|A\triangle B|\leq c_n^{\ref{main_thm_6}}t^{-1/2}\delta^{1/2}|A|.$$
\end{thm}

The exponents of $\delta$ and $t$ are optimal as shown by the example $A=[0,1+\sqrt{\delta/t}]\times [0,1]^{d-1}$ and $B=[0,1]\times [0,1+\sqrt{\delta/t}]\times [0,1]^{d-2}$. In this case, we find $tA+(1-t)B=[0,1+t\sqrt{\delta/t}]\times [0,1+(1-t)\sqrt{\delta/t}]\times [0,1]^{d-2}$, so that $|tA+(1-t)B|\leq(1+\delta)|A|$, while $|A\triangle B|\geq2\sqrt{\delta/t}|A|$.

An important step in proving \Cref{LinearThmGeneral} is to establish the case where $\co(A)$ and $\co(B)$ have a bounded number of vertices.

\begin{thm}
\label{main_thm_2}
For all $n,v \in \mathbb{N}$ and $t \in (0,1)$, there are computable constants $k^{\ref{main_thm_2}}_n(v),c^{\ref{main_thm_2}},d^{\ref{main_thm_2}}_{n,t}>0$ such that the following holds. Assume $\delta \in [0,d^{\ref{main_thm_2}}_{n,t}]$, and assume $A,B\subset \mathbb{R}^n$ are measurable sets of equal volume, so that $\co(A)$ has at most $v$ vertices and $|tA+(1-t)B| =(1+\delta)|A|.$ Then
$$|\co(A)\setminus A|\le k^{\ref{main_thm_2}}_n(v)\min\{t,1-t\}^{-c^{\ref{main_thm_2}}n^8}\delta|A|.$$
\end{thm}

As mentioned above, the volume of the symmetric difference $|A\triangle B|$ is a commonly used parameter to measure stability in geometric inequalities. Here, instead of controlling $|A\triangle B|$, we want to improve this notion of closeness by finding a common convex set $K$ that contains both $A$ and $B$.
The key step to achieve this is contained in the next general theorem about convex sets, which is of independent interest.


\begin{thm}\label{thm_int_convex}
There exists a constant $c_n$ such that given convex sets $X,Y\subset\mathbb{R}^n$, we have
$$\frac{|\co(X\cup Y)|}{\min\{|X|,|Y|\}}-1\leq c_n\frac{|X\triangle Y|}{|X\cap Y|}.$$
\end{thm}

\begin{rmk}
In \Cref{main_thm_1}, \Cref{LinearThmGeneral} and \Cref{main_thm_6}, we can assume that, for fixed $n \in \mathbb{N}$, the function $d_{n,t} \colon (0,1/2] \to \mathbb{R}_+$ is increasing. Similarly, in \Cref{main_thm_2}, we can assume that for fixed $n\in\mathbb{N}$, the function $d_{n,t}$ increases in $(0,1/2]$ and decreases in $[1/2,1)$. This follows from the proofs.
\end{rmk}

\subsection{An alternative approach to \Cref{main_thm_6}.}

While working on this project, we proved the following result. 

\begin{thm}
\label{main_thm_5}
For all $n \in \mathbb{N}$ and $t \in (0,1/2]$, there are computable constants $ c^{\ref{main_thm_5}}_n, d^{\ref{main_thm_5}}_{n,t}, \Gamma^{\ref{main_thm_5}}_{n,t}>0$ such that the following holds. Assume $\delta\in[0, d^{\ref{main_thm_5}}_{n,t}]$,   $\gamma\in[0, \Gamma^{\ref{main_thm_5}}_{n,t}]$,  and assume that $A,B\subset\mathbb{R}^n$,  are measurable sets with equal volume so that 
$$ \left|tA+(1-t)B\right| \leq (1+\delta)|A|\qquad \text{ and }\qquad|\co(A)\setminus A|+|\co(B)\setminus B|\leq \gamma|A|.$$
Then, up to translation, $$|A \triangle B| \leq c^{\ref{main_thm_5}}_{n}\sqrt{\frac{\delta+\gamma}{t}}|A|.$$
\end{thm}

This theorem, combined with \Cref{LinearThmGeneral}, has the following important corollary. 

\begin{cor}\label{badtcor}
For all $n \in \mathbb{N}$ and $t \in (0,1/2]$, there are computable constants $ c_{n,t}^{\ref{badtcor}}, d^{\ref{badtcor}}_{n,t}>0$ such that the following holds. Assume $\delta\in[0, d^{\ref{badtcor}}_{n,t}]$ and assume that $A,B\subset\mathbb{R}^n$, are measurable sets with equal volume so that $\left|tA+(1-t)B\right| \leq (1+\delta)|A|.$ Then, up to translation, $$|A \triangle B| \leq c^{\ref{badtcor}}_{n,t}{\delta^{1/2}}|A|.$$   
\end{cor}

Although this corollary provides a sharp stability bound in terms of $\delta$, it is weaker than \Cref{main_thm_6} in that the dependency on $t$ is suboptimal. Actually, even combining \Cref{main_thm_5} with the optimal result contained in \Cref{main_conj_2} would not obtain the optimal $t$-dependence provided by \Cref{main_thm_6}.
For this reason, in this paper, we develop a completely different approach to prove \Cref{main_thm_6} that bypasses the use of \Cref{main_thm_5}.

Still, we believe that the proof of \Cref{main_thm_5} brings a lot of value in studying the stability of the Brunn-Minkowski inequality, as it uses a mass transport approach in a new original way, and we defer its proof to a forthcoming paper \cite{OTBMStab}.

\subsection{Notation and conventions.}

Before starting our proofs, it is convenient to briefly explain the notation that we will use throughout the paper.
With $c>0$, we shall denote a universal constant independent of the dimension, while $c_n>0$ (and analogous notations) denote dimensional constants. Saying that the quantity $a$ is controlled by $O_n(b)$ means that $|a|\leq c_nb$, while notation $a=\Omega_n(b)$ means that $a\geq c_n|b|$. When a constant also depends on $t$, we write $c_{n,t}$.
To distinguish the constants that appear in the different statements, $c^{\ell.m}$ means that the constant $c$ is the one appearing in Theorem/Proposition/Lemma $\ell.m$. 

Throughout the paper, we fix $n \in \mathbb{N}$ with $n\geq 3$ and either $t \in (0,1/2]$ or $t \in (0,1)$; unless otherwise specified, we assume the former. We use $|\cdot|$ to denote the outer Lebesgue measure in $\mathbb{R}^n$.

Given $s \in \mathbb{R}$ and sets $X$ and $Y$ in $\mathbb{R}^n$, we define $sX=\{sx \colon x \in X\}$ and $X+Y=\{x+y \colon x\in X, y\in Y\}$. A set $X$ in $\mathbb{R}^n$ is convex if for all $t \in [0,1]$ we have $t X+ (1-t)X \subset X$.
The convex hull $\co(X)$ of a set $X$ in $\mathbb{R}^n$ is the intersection of all convex sets containing $X$. In particular, $\co(X)$ is a convex set. Two sets $X$ and $Y$ of $\mathbb{R}^n$ are homothetic if there exist a point $z$ in $\mathbb{R}^n$ and a scalar $s> 0$ such that $X=sY+z$.

Given a bounded convex set $X$ in $\mathbb{R}^n$, we define $\overline{X}$ as the closure of $X$, which is also a convex set. The vertices of $X$, denoted by $V(X)$, represent the set $V(X)=\{x \in \overline{X}  \colon \co(\overline{X} \setminus \{x\})\neq \co(\overline{X})\}$. It follows that $\overline{X}=\co(V(X))$. 

Measureable sets $X_1, \dots, X_k$ in $\mathbb{R}^n$ are said to form an essential partition of $\mathbb{R}^n$ if $|\cap_i X_i^c|=0$ and for $j_1 \neq j_2$, we have $|X_{j_1}\cap X_{j_2}|=0$. 
By a basis $e_1, \dots, e_n$ in $\mathbb{R}^n$, we mean an orthogonal set of vectors with unit length.
In light of \Cref{boundedsandwichreduction}, we can assume that the sets $A$ and $B$ (as well as all parts into which we subdivide $A$ and $B$) are compact.

\subsection{Overview of the proofs of the main results}

We first prove the linear stability \Cref{LinearThmGeneral} which is a crucial tool at several steps in the proof of the quadratic stability \Cref{main_thm_1}. The proof of \Cref{LinearThmGeneral} breaks up into two parts: first we show a linear result for sets $A$ so that $\co(A)$ has few vertices (\Cref{main_thm_2}), then we use this to prove the result for an arbitrary number of vertices.

\smallskip

First consider \Cref{main_thm_2}, i.e., the linear stability to the convex hull for sets with few vertices. The first step  is to reduce to the case where $\co(A)$ is a simplex (see \Cref{LinearThm}). Assume that $A$ and $B$ are a finite union of points and boxes, $V(\co(A))=\{x_0, \dots, x_n\}$,  and assume that $|A|=|B|$ and $|tA+(1-t)B| \leq (1+\delta) |A|$. Our move now is to pick a point $x$ in $A$ and construct cones $C_0, \dots, C_n$ where $C_i$ has a vertex at $x$ and is generated by rays $xx_0, xx_1 \dots, xx_{i-1}, xx_{i+1}, \dots, xx_{d-1}, xx_n$. The cones $C_0, \dots, C_n$ partition $A$ into subsets $A_0, \dots, A_n$ with the property that $\co(A_i)$ is again a simplex. We find a translation $y$ such that the cones $y+C_0, \dots, y+C_n$ partition $B$ into subsets $B_0, \dots, B_n$ where $|A_i|=|B_i|$ (see \Cref{lem_first_step}). Repeating this move in each part, we create a partition of $A$ and $B$ into sets $A_1, \dots, A_m$ and $B_1, \dots, B_m$ with the property that $|A_i|=|B_i|$, the sets $\{tA_i+(1-t)B_i\}_{1\leq i\leq m}$ are all disjoint, and (crucially) $\co(A)=\sqcup_i\co(A_i)$.  Our aim then becomes to show \Cref{main_thm_2} for essentially all parts $A_i$ and $B_i$, that is, $|\co(A_i) \setminus A_i| \leq O_{n,t}(\delta_i)|A_i|$.
Then, we can combine all the pieces to get $$|\co(A)\setminus A|=\sum_i|\co(A_i) \setminus A_i| \leq \sum_i O_{n,t}(\delta_i)|A_i|=O_{n,t}(\delta)|A|.$$
In this process, we stop further subdividing a part $A_i$ when either $A_i=\co(A_i)$ or $|A_i| \leq 0.01|\co(A_i)|$. In \Cref{densitycontrolledpartition}, we show that we can pick the points in each part sufficiently \emph{centrally} in such a way as to guarantee the following two facts. First, provided that some part $A_i$ satisfies $|A_i| \geq 0.01 |\co(A_i)|$, then all the $n+1$ parts $A_{i,0},\dots, A_{i,n}$ in which we subdivide $A_i$ satisfy $|A_{i,j}| \geq \Omega_{n}(1)|\co(A_{i,j})|$. Second, for any parameter $\varepsilon>0$, there exists $k \in \mathbb{N}$ such that essentially all parts $A_i$ constructed in generation $k$ are small, in the sense that the total volumes of all large parts $A_i$ with $\text{diam}(A_i) \geq \varepsilon  $ in generation $\ell$ are at most $\varepsilon$. 

To conclude, we choose $\varepsilon>0$ smaller than $\delta^2|A|$, and sufficiently small so that the boundary of $A$ (recall that $A$ is a finite union of boxes and points) thickened by $\varepsilon$ has size at most $\delta^2|A|$. We run the process up to generation $\ell$ and note that all parts fall into four categories. First, there are the parts $A_i$ where we stopped further subdividing because $|A_i| \leq 0.01 |\co(A_i)|$; for such $A_i$ we also get $|A_{i}| \geq \Omega_{n}(1)|\co(A_{i})|$. Second, there are the parts $A_i$ where we stopped further subdividing because $A_i=\co(A_i)$. Third are the small parts $A_i$ in generation $\ell$ with $\text{diam}(A_i)\leq \varepsilon$ and $0.01|\co(A_i)|<|A_i|<|\co(A_i)|$. As these are neither empty nor full, $\co(A_i)$ must intersect the boundary of $A$ so that the combined volume of these $\co(A_i)$'s is at most $\delta^2|A|$.  Finally, there are the big parts in generation $\ell$ with $\text{diam}(A_i)> \varepsilon$, the total volume of which is at most  $\delta^2|A|$.

Now, neglecting the parts in the third and fourth categories, as they contribute very little, it is easy to check \Cref{main_thm_2} for the parts in the first and second categories. Indeed, in the first category we can use a qualitative stability result (see \Cref{SharpDelta}), and in the second category there is nothing to prove as the set is already convex.
\smallskip

Having proved \Cref{main_thm_2}, we turn our attention to generalizing it to an arbitrary number of vertices of $\co(A)$.

The first step (\Cref{DensityDiamLemma}) is to identify a collection of disjoint convex regions $X_i\subset \co(A)$ of small diameter $\epsilon$ that contain a positive proportion of  the missing volume of $\co(A)\setminus A$, i.e.,  $\sum|\co(A\cap X_i)\setminus (A\cap X_i)|=\Omega_n(|\co(A)\setminus A|)$. To find the regions $X_i$ of small diameter, we induct akin to the proof for few vertices. First, we find as follows a triangulation of $\co(A)$ so that all simplices are pretty full. We triangulate $\partial \co(A)$ and consider the $n$-simplices formed by a $(n-1)$-simplex in $\partial \co(A)$ together with the origin $o$. Within each of these simplices, we find a central point in $A$ which partitions the simplex into smaller simplices all of which preserve the convex hull. We iterate until we find a simplex with low density (between 98\% and 99\%) or small diameter. For simplices with low density, finding a convex subregion with smaller diameter containing a positive proportion of the missing region is a lot simpler (\Cref{LowDensityCaseInductionStep}).

A standard reduction (\Cref{boundedsandwichreduction}) allows us to assume that $$B^n(o,\Omega_n(1)) \subset K\subset A,B \subset (1+\eta)K\subset  B^n(o,O_n(1))$$ for some convex set $K$ and $\eta$ small in terms of $n$ and $t$. We consider a simplicial tube $U$ i.e., a set of the form $T\times \mathbb{R}^{+}$ with $T\subset \mathbb{R}^{n-1}$ a regular simplex centered at the origin. We insist that the diameter of $U$ is small in terms of $n$ but much larger than the diameter $\epsilon$ of the regions $X_i$. We take a random rotation of $U$. As each $X_i$ is completely contained inside $U$ with probability $\Omega(1)$, we get that $$|\co(A) \setminus A|= O(1)\sum|\co(A\cap X_i)\setminus (A\cap X_i)|= O(1)\mathbb{E}|\co(A\cap U) \setminus (A\cap U)|.$$
So it suffices to show that for every rotation of $U$ we have $|\co(A\cap U)\setminus (A\cap U)|=O_{n,t}(\delta)|A|$. After some reductions, we may assume $|A\cap U|=|B\cap U|$ and $|t(A\cap U)+(1-t)(B\cap U)|-|A\cap U|\leq |tA+(1-t)B|-|A|$. Partitioning $U$ into smaller parallel tubes $U_i$, we find that
$$\sum_i \Big(|t(A\cap U_i)+(1-t)(B\cap U_i)|-t|A\cap U_i|-(1-t)|B\cap U_i|\Big)\leq |t(A\cap U)+(1-t)(B\cap U)|-|A\cap U|,$$
although we might have $|A\cap U_i|\neq |B\cap U_i|$. In the direction of the tubes, every fibre of $A\cap U_i$ and $B\cap U_i$ starts with a long interval, so that extending one of the sets at the bottom does not affect $|t(A\cap U_i)+(1-t)(B\cap U_i)|-t|A\cap U_i|-(1-t)|B\cap U_i|$ (cf. \Cref{extendingTubes}). Extending appropriately (which is always quite little), we retrieve $|A\cap U_i|= |B\cap U_i|$, and by choosing $U_i$ appropriately we find that $\co(A\cap U_i)$ has few vertices. Applying \Cref{main_thm_2} concludes.\\

We turn to proof of \Cref{main_thm_1}. The first step is to reduce \Cref{main_thm_1} to \Cref{main_thm_6}, that is, the first step is to switch from showing that $A+x$ and $B+y$ are contained in a common convex set $K$ with almost the same volume, to show that $|(A+z) \triangle B|$ is small. This reduction follows quickly from \Cref{thm_int_convex} and \Cref{LinearThmGeneral}.

\smallskip

The starting point for the proof of \Cref{main_thm_6} is inspired by a classical proof of the Brunn-Minkowski inequality. Assume that the sets $A$ and $B$ are finite unions of boxes and assume that $|A|=|B|$ and $|tA+(1-t)B| \leq (1+\delta)|A|$. Our move is to choose a hyperplane $H$ that partitions $A$ into $A_1$ and $A_2$, and then to find a parallel hyperplane $G$ that partitions $B$ into $B_1$ and $B_2$ such that $|A_i|=|B_i|$. Repetition of this move creates a partition of $A$ and $B$ into sets $A_1, \dots, A_m$ and $B_1, \dots, B_m$ with the property that $|A_i|=|B_i|$ and the sets $\{tA_i+(1-t)B_i\}_{1\leq i \leq m}$ are all disjoint. Hence, if $|tA_i+(1-t)B_i|=(1+\delta_i)|A_i|$, then $\sum_i\delta_i|A_i|\leq\delta |A|$. Our aim is to show \Cref{main_thm_6} for each pair of sets $A_i$ and $B_i$, that is, $|(A_i+x_i) \triangle B_i| \leq O_{n,t}(\delta_i^{1/2})|A_i|$. This is motivated by the fact that, by cutting with sufficiently many hyperplanes, the sets $A_i$ and $B_i$ become simple enough (e.g. boxes). Now, under the much too optimistic assumption that all translates $x_i$ coincide, we can put all the pieces together to get $$|(A+x) \triangle B| = \sum_i|(A_i+x_i) \triangle B_i| \leq \sum_i O_{n,t}(\delta_i^{1/2})|A_i| \leq O_{n,t}(\delta^{1/2})|A|.$$ 
The main problem with the above plan is that there is no reason to believe that all the translates $x_i$ coincide. In general, given sets $Y$ and $Z$, it is difficult to grasp the optimal translation $x$ that minimizes $|(Y+x) \triangle Z|$. However, for certain classes of sets, which we call cone-like sets, this is possible.

A cone $C$ is the intersection of a finite family of half-spaces generated by hyperplanes through the origin $o$. In particular, the vertex of the cone is at the origin. We always assume that the cone is not too wide, that is, all angles at $o$ are less than $179^{\circ}$. We say that $Y \subset C$ is $C$-like if $$C \cap B^n(o,\Omega_n(1)) \subset Y \subset C \cap B^n(o,O_n(1))$$ 

Now, fix a cone $C$ and let $Y$ and $Z$ be convex (or approximately convex) subsets of $C$  that are $C$-like. For such sets, if for a translate $x$ we have $|(Y+x) \triangle Z| \leq \sqrt{\delta}|Y|$ then $||x||_2$ must be small so that also $|Y \triangle Z| \leq O_n(\sqrt{\delta})|Y|$. So, in effect, the essentially optimal translation for cone-like sets is $x=o$.

Motivated by this idea, the plan is to partition the whole space into cones $C_1, \dots C_m$ using hyperplane cuts, which induce the partitions of $A$ and $B$ into subsets $A_1, \dots, A_m$ and $B_1, \dots, B_m$, in order to reduce \Cref{main_thm_6} from $A$ and $B$ to $A_i$ and $B_i$, which we hope to be simpler to deal with. For this to work, we must impose that the subsets $A_i$ and $B_i$ of $C_i$ are $C_i$-like and have the same size. The first condition turns out to be easy to satisfy, as after an affine transformation we can assume without loss of generality that $B^n(o,\Omega_n(1)) \subset A,B \subset  B^n(o,O_n(1))$ (see \cref{boundedsandwichreduction}). The second condition is also rather easy to satisfy.

In \Cref{prop_fourth_step}, we show that we can always find such a partition into cones, each of which is the convex hull of a set of bounded size of rays through the origin. Moreover, all the rays are clustered arbitrarily close to a pair of rays.  Thus, each cone $C_i$ has a bounded number of faces and is arbitrarily narrow in all but one direction. (In general, one cannot hope for all cones to be arbitrarily narrow in all directions, as one can see already in the two-dimensional case when $A=[0,1-\varepsilon] \times [0,1+\varepsilon]$ and $B=[0,1+\varepsilon] \times [0,1-\varepsilon]$.) 


Now, recalling that $A$ and $B$ are finite unions of boxes and insisting that the narrow directions of the cones are much smaller than the sides of the boxes, we can assume without loss of generality that the subsets $A_i$ and $B_i$ of $C_i$ are such that each sectional cut in the short directions is either completely full or completely empty.

So far, we have reduced \Cref{main_thm_6} to rather simple sets $A_i$ and $B_i$. So, given such sets $A_i$ and $B_i$ with $|A_i|=|B_i|$ and $|tA_i+(1-t)B_i| \leq (1+\delta_i)|A_i|$, we need to show that $|A_i \triangle B_i| \leq O_{n}(\delta_i^{1/2}/t^{1/2})|A_i|$. 
\smallskip

We proceed akin to the final step for the linear result. It is enough to show that for each large tube $U$  loosely oriented in the same direction as $C_i$ we have $|(A_i\cap U) \triangle  (B_i\cap U)| \leq O_n(\delta_i^{1/2}/t^{1/2})|A_i|$ . After some reductions, we may assume $|A_i\cap U|=|B_i\cap U|$. Now by the linear result we know that $|\co(A_i \cap U) \setminus (A_i\cap U)| \leq O_{n,t}(\delta)|A_i|$. Recall that in $A_i$ and $B_i$ each sectional cut of $C_i$ in the short directions is either completely full or completely empty. Combining the two facts we get that $A_i \cap U$ and $B_i\cap U$ are essentially one-codimension compressed sets. For such sets we can prove a sharp quadratic stability and thus conclude.

\subsection{Structure of the paper}
Given the complexity of our proofs, we kindly encourage the reader to refer to the table of contents as a guide for navigation and orientation within the paper. We note that:\\
- The proof of \Cref{main_thm_1} uses
\Cref{LinearThmGeneral}, \Cref{main_thm_6}, and \Cref{thm_int_convex};\\
- The proof of \Cref{main_thm_6} uses \Cref{LinearThmGeneral};\\
- The proof of \Cref{LinearThmGeneral} uses \Cref{main_thm_2};\\
- \Cref{thm_int_convex} and \Cref{main_thm_2} are proved directly.

\begin{ack*}AF acknowledges the support of the ERC Grant No.721675 ``Regularity and Stability in Partial Differential Equations (RSPDE)'' and of the Lagrange Mathematics and Computation Research Center.
\end{ack*}

\section{Initial reductions of \Cref{LinearThmGeneral}, \Cref{main_thm_6}, and \Cref{main_thm_2}}\label{InitialReductionSec}

We shall deduce some basic allowed additional assumptions for \Cref{main_thm_6}, \Cref{LinearThmGeneral}, and \Cref{main_thm_2}, but before we do so, we need a few definitions.

\subsection{Setup}

\begin{defn}\label{defn_simple}
A set $X \subset \mathbb{R}^n$ is called  \textit{simple} if $$X=\bigsqcup_{i \leq k} x_i+[0,1]^n$$ for some $k \in \mathbb{N}$, that is, $X$ is a finite disjoint union of translates of the unit cube. 
\end{defn}

\begin{defn}\label{defn_cone}
A convex set $C\subset \mathbb{R}^n$ is called a \textit{cone} if there exists a hyperplane $H$ not containing the origin and a bounded convex set $P \subset H$ such that $$C=\bigsqcup_{t \geq 0} tP.$$
\end{defn}

\begin{defn}
We write $S^{v_0, \dots, v_n}$ for the simplex with vertices $v_0, \dots, v_n$.  Assuming that $S^{v_0, \dots, v_n}$ contains the origin in the interior, construct the family of cones $\mathfrak{C}^{v_0,\dots, v_n}:=\{C_i:0\leq i\leq n\}$, where
 $$C_i=\bigsqcup_{t\geq 0}t\co(v_0, \dots v_{i-1}, v_{i+1}, \dots, v_n).$$
\end{defn}
Note that the cones in $\mathfrak{C}^{v_0,\dots, v_n}$ form an essential partition of $\mathbb{R}^n$.

\begin{defn}\label{defn_simplex}
Fix vectors $e_0, \dots, e_n \in \mathbb{R}^n$ such that  $S^{e_0, \dots, e_n}$ is a regular unit volume simplex centered at the origin $o$. Denote $S=S^{e_0, \dots, e_n}$ and $\mathfrak{C}=\mathfrak{C}^{e_0,\dots, e_n}$.
\end{defn}

\begin{defn}\label{defn_lambdabounded}
A pair of sets $X,Y\subset\mathbb{R}^n$ is \emph{$\lambda$-bounded} if there exists an $r>0$ so that 
$$rS\subset X,Y\subset  \lambda r S.$$
\end{defn}

\begin{obs}\label{BoundedBallsObs}
If $A,B\subset\mathbb{R}^n$ are $\lambda$-bounded and $|A|=|B|=1$, then $B(o, (2\lambda n)^{-1})\subset A,B\subset B(o, 2\lambda n)$.
\end{obs}

\begin{defn}\label{defn_lambdaconebounded}
Given a cone $F\subset C'\in\mathfrak{C}$, a pair of sets $X,Y\subset\mathbb{R}^n$ is \emph{$(\lambda,F)$-bounded}  if there exists an $r>0$ so that 
$$r(F\cap S)\subset X,Y\subset  \lambda r (F\cap S).$$
\end{defn}

\begin{defn}
A pair of sets $X,Y\subset \mathbb{R}^k$ is called a \textit{$\eta$-sandwich} if there exists a convex set $P$ such that $o\in P \subset X,Y \subset (1+\eta) P$.
\end{defn}
Given a cone $F$ and an $\eta$-sandwich $X,Y\subset\mathbb{R}^n$, the pair $X\cap F,Y\cap F$ is also an $\eta$-sandwich.

\subsection{Proposition}

\begin{prop}\label{boundedsandwichreduction}
Fix $n \in \mathbb{N}$ and $t\in (0,1/2]$.
\begin{itemize}
\item Assume that there exist constants $\lambda^{\ref{boundedsandwichreduction}}_n>2n^2$ and $\eta^{\ref{boundedsandwichreduction}}_{n,t}>0$ such that \Cref{main_thm_6} holds for all simple $\lambda^{\ref{boundedsandwichreduction}}_n$-bounded  $\eta^{\ref{boundedsandwichreduction}}_{n,t}$-sandwiches $A,B\subset\mathbb{R}^n$. Then it holds for all measurable sets $A,B\subset\mathbb{R}^n$. 
\item Assume that there exist constants $\lambda^{\ref{boundedsandwichreduction}}_n>2n^2$ and $\eta^{\ref{boundedsandwichreduction}}_{n,t}>0$ such that \Cref{LinearThmGeneral} holds for all $\lambda^{\ref{boundedsandwichreduction}}_n$-bounded  $\eta^{\ref{boundedsandwichreduction}}_{n,t}$-sandwiches $A,B\subset\mathbb{R}^n$ so that $A$ is a simple set intersected with $\co(A)$ and analogously for $B$. Then it holds for all measurable sets $A,B\subset\mathbb{R}^n$.
\item Similarly, fix $n \in \mathbb{N}$ and $t\in (0,1)$. If  \Cref{main_thm_2} holds for all $\lambda^{\ref{boundedsandwichreduction}}_n$-bounded  $\eta^{\ref{boundedsandwichreduction}}_{n,t}$-sandwiches $A,B\subset\mathbb{R}^n$ so that $A$ is a simple set intersected with $\co(A)$, then it holds for all measurable sets $A,B\subset\mathbb{R}^n$.
\end{itemize}
\end{prop}

\subsection{Auxiliary Lemmas}

We first collect some auxiliary results that will be used to prove \Cref{boundedsandwichreduction}. The proof of such results will be given in \Cref{subsect:proof aux} below.

We recall the following result by Michael Christ. 
\begin{thm}[Christ 2012, \cite{christ2012near}]\label{thm_christ}
For all $n \in \mathbb{N}$, $t \in (0,1)$ and $\eta>0$, there exist constants $\Delta^{\ref{thm_christ}}_n>0$, so that for all measurable $X,Y\subset\mathbb{R}^n$ of equal volume with the property that $|tX+(1-t)Y|\leq (1+\Delta^{\ref{thm_christ}}_n)|X|$, then $$\min_{v\in\mathbb{R}^n}|\co(X\cup (v+Y))|\leq (1+\eta)|X|.$$
\end{thm}

We also need two lemmas.

\begin{lem}\label{prop_sandwichmaker}
For $n \in \mathbb{N}$, $t \in (0,1/2]$ and $\eta>0$, there exist constants $c^{\ref{prop_sandwichmaker}}$  and $\Delta^{\ref{prop_sandwichmaker}}_{n,t}(\eta)>0$ so that the following holds. If $X,Y\subset\mathbb{R}^n$ are measurable sets with $|X|=|Y|$ and $|tX+(1-t)Y|= (1+\delta)|X|$ with $\delta \in[0,\Delta^{\ref{prop_sandwichmaker}}_{n,t}]$, then, up to translation, there exist measurable sets $ X',Y' \subset \mathbb{R}^n$  so that 
\begin{enumerate}
\item $X',Y'$ is an $\eta$-sandwich,
\item $|X'|=|Y'|=|X|$,
\item $\co(X')=\co(X)$ and $\co(Y')=\co(Y)$, 
\item $|X'\triangle X|+|Y'\triangle Y|\leq c^{\ref{prop_sandwichmaker}}t^{-1}\delta|X|$,
\item $|tX'+(1-t)Y'|\leq (1+\delta)|X|$.
\end{enumerate}
Moreover, if $X\subset Y$, we additionally find $X'\subset Y'$.
\end{lem}

\begin{lem}\label{prop_boundedmaker}
For $n \in \mathbb{N}$, and $\eta>0$ the following holds. If $X,Y\ \subset \mathbb{R}^n$ is an $\eta$-sandwich, then there exists $v\in \mathbb{R}^n$ and there exists a linear transformation $\theta\colon \mathbb{R}^n\rightarrow \mathbb{R}^n$ such that $\theta(v+X), \theta(v+Y)$ is a $(n^2+n^3\eta)$-bounded $n\eta$-sandwich.
\end{lem}

\subsection{Proof of \Cref{boundedsandwichreduction}}

\begin{proof}[Proof of \Cref{boundedsandwichreduction}]
Choose $\eta$ sufficiently small in terms of $n$, $\eta_n^{\ref{boundedsandwichreduction}}$, and $\lambda_n^{\ref{boundedsandwichreduction}}$. Choose $\zeta$ small in terms of $\eta$, and choose $\xi$ sufficiently small in terms of $\zeta$ and $\eta$.

Apply \Cref{prop_sandwichmaker} with parameter $\eta$ to $A, B$ to find $A_1,B_1$. Apply \Cref{prop_boundedmaker} to $A_1,B_1$ to find $A_2,B_2$ which is a $(n^2+n^3\eta)$-bounded $n\eta$-sandwich. To additionally get that the sets are simple we use a standard approximation (see e.g. \cite[Lemma 3.13]{van2023locality}). Find compact subset $A_3\subset A_2$, so that $|A_2\setminus A_3| \to 0$ as $\zeta \to 0$. We can ensure $|\co(A_3)| \to |\co(A_2)|$ as $\zeta \to 0$, by requiring $A_3$ to contain a large finite subset of the vertices of $\co(A_2)$ (or all of them if $V(\co(A_2))$ is finite). Analogously define $B_3$. Note that we may additionally ask that $A_3,B_3$ have the same size and form a $(n^2+n^3\eta)$-bounded $n\eta$-sandwich.

From here we consider \Cref{main_thm_6}, \Cref{LinearThmGeneral}, and \Cref{main_thm_2} separately. First, consider \Cref{main_thm_6}. Let $$A_{4}:=\{x\in(\xi\mathbb{Z})^n: (x+[0,\xi]^n)\cap A_3\neq\emptyset\}+[0,\xi]^n\supset A_3$$ and $B_4$ analogously. Since $A_3$ is compact, $|A_4\setminus A_3|\to 0$ as $\xi\to 0$. Similarly, as $tA_3+(1-t)B_3$ is compact, $|(tA_4+(1-t)B_4)\setminus (tA_3+(1-t)B_3)|\to 0$ as $\xi\to 0$. We can construct subsets $A_5$ and $B_5$ of $A_4$ and $B_4$, respectively such that $A_5=X+[0, \xi]^n$ and $B_5=Y+[0,\xi]^n$ with $X, Y \subset (\xi \mathbb{Z})^n$, $|A_5|=|B_5|=\min\{|A_4|,|B_4|\}$ and $A_5,B_5$ is a simple $(n^2+n^3\eta)$-bounded $2n\eta$-sandwich as $\xi \to 0$. By the above, $$\lim_{\zeta \to 0} \lim_{\xi \to 0}|tA_5+(1-t)B_5|/|A_5| \leq |tA_2+(1-t)B_2|/|A_2|\leq 1+\delta.$$
Choosing $\eta$ sufficiently small, we can apply \Cref{main_thm_6} to find $\lim_{\zeta \to 0} \lim_{\xi \to 0}|A_5\triangle B_5|/|A_5|\leq c_n^{\ref{main_thm_6}}t^{-1/2}\delta^{\frac12}$ (up to a translation). This implies that, up to a translation,
$|A_2\triangle B_2|/|A_2|\leq 2c_n^{\ref{main_thm_6}}t^{-1/2}\delta^{\frac12}$, hence $|A\triangle B|/|A|\leq 3c_n^{\ref{main_thm_6}}t^{-1/2}\delta^{\frac12}$, as desired.

For \Cref{LinearThmGeneral} and \Cref{main_thm_2}, proceed analogously to the previous paragraph but with $$A_{4}:=(\{x\in(\xi\mathbb{Z})^n: (x+[0,\xi]^n)\cap A_3\neq\emptyset\}+[0,\xi]^n)\cap \co(A_3)\supset A_3$$ and $B_4$ similarly. 
\end{proof}

\subsection{Proofs of Auxiliary Lemmas}
\label{subsect:proof aux}
\subsubsection{Proof of \Cref{prop_sandwichmaker}}

The idea will be to show that a large homothetic copy of $\co(X\cup Y)$ is contained in $tX+(1-t)Y$, so that adding a slightly smaller homothetic copy of $\co(X\cup Y)$ to $X$ and $Y$ will not change $tX+(1-t)Y$.

\begin{proof}[Proof of \Cref{prop_sandwichmaker}]
Translate $X$ and $Y$ so that $|\co(X\cup Y)|$ is minimal. Additionally, taking an affine transformation if necessary, we may assume that the John ellipsoid $E\subset\co(X\cup Y)$ is a ball centered at the origin. Note that $|E|\geq n^{-n}|\co(X\cup Y)|\geq n^{-n}|X|$.

Let $\eta'=t\eta$ and $\xi=\min\left\{\frac13\left(\frac{t\eta'}{(1-t)n}\right)^n,2^{-n}\right\}$. Apply \Cref{thm_christ} with parameters $t$, $n$, and $\eta^{\ref{thm_christ}}=\xi$ to produce $\Delta^{\ref{thm_christ}}$, and choose $\Delta_n= \Delta_n^{\ref{thm_christ}} $, so that $|\co(X\cup Y)|\leq (1+\xi)|X|$.

\begin{clm}
$(1-\eta')\co(X\cup Y)\subset tX+(1-t)Y$
\end{clm}
\begin{proof}[Proof of claim]
Consider a point $p\in (1-\eta')\co(X\cup Y)$. Find the point $p'\in \partial \co(X\cup Y)$ so that $p=\lambda p'$ for some $\lambda\in [0,1-\eta']$. Let $H:\mathbb{R}^n\to \mathbb{R}^n$ be the homothety of ratio $1-\lambda>\eta'$ centered at $p'$. By convexity $H(E)\subset H(\co(X\cup Y))\subset \co(X\cup Y)$ and $H(o)=p'$, so that $\co(X\cup Y)$ contains a set of size $$|H(E)|\geq (1-\lambda)^n |E| \geq \eta'^n|E|\geq (\eta'/n)^n|X|$$ symmetric around $p$. Note that 
$$|H(E)\setminus Y|\leq |\co(X\cup Y)\setminus Y|\leq \xi |Y|\leq \frac13\left(\frac{t\eta'}{(1-t)n}\right)^n|Y| \leq \frac13|H(E)|,$$
 hence $|H(E)\cap Y|\geq \frac23|H(E)|$. Now consider the homothety $H':\mathbb{R}^n\to\mathbb{R}^n$ of ratio $-t/(1-t)$ centered at $p$. Consider the set $H'(H(E)\cap Y)\subset H(E)$ and note that $|H'(H(E)\cap Y)|\geq \frac23 |H'(H(E))|$. Similarly as before, we find
$$|H'(H(E)\cap Y)\setminus X|\leq |\co(X\cup Y)\setminus X|\leq \frac13\left(\frac{t\eta'}{(1-t)n}\right)^n|X|\leq \frac13 |H'(H(E))|\leq \frac12 |H'(H(E)\cap Y)|.$$
Hence, there exists a point $y\in H(E)\cap Y$, so that $H'(y)\in X\cap H'(H(E))$. Note that this implies that $tx+(1-t)y=p$, so that $p\in tX+(1-t)Y$. This concludes the claim.
\end{proof}

Consider the sets $X'':= X\cup (1-t^{-1}\eta')\co(X\cup Y)$ and $Y'':=Y\cup (1-t^{-1}\eta')\co(X\cup Y)$. Note that we get
\begin{align*}tX''+(1-t)Y''&\subset (tX+(1-t)Y)\cup (t(1-t^{-1}\eta')\co(X\cup Y)+(1-t)\co(X\cup Y))\\
&= (tX+(1-t)Y)\cup (1-\eta')\co(X\cup Y))=tX+(1-t)Y,
\end{align*}
where in the last equality we used the claim. We find $\frac{|X''\setminus X|}{|X|}\leq\frac{|\co(X\cup Y)\setminus X|}{|X|}\leq  \xi\leq 2^{-n}$.  Now using the Brunn-Minkowski inequality, we find
$(t|X''|^{\frac1n}+(1-t)|Y''|^{\frac1n})^n\leq |tX''+(1-t)Y''|= |tX+(1-t)Y|\leq (1+\delta)|X|,$
which yields
\begin{align*}
\left(1+t\frac{|X''\setminus X|}{2n|X|}+(1-t)\frac{|Y''\setminus Y|}{2n|Y|}\right)^n \leq 
\left(t\left(1+\frac{|X''\setminus X|}{|X|}\right)^{\frac1n}+(1-t)\left(1+\frac{|Y''\setminus Y|}{|Y|}\right)^{\frac1n}\right)^n\leq 1+\delta,
\end{align*}
so that 
$\left|X''\setminus X\right|+\left|Y''\setminus Y\right|\leq 2\delta t^{-1}|X|.$

Having gained control over the size of $X''$ and $Y''$, we remove some of it to get the equality in sizes. To this end choose subsets $X'\subset X''$ so that the vertices of $\co(X)=\co(X'')$ are in $X'$ (so that $\co(X')=\co(X)$), $(1-t^{-1}\eta')\co(X\cup Y)\subset X'$ and $|X'|=|X|$. Similarly, choose $Y'\subset Y''$ and ensure $|X'|=|Y'|$. Now note that
$$|X'\triangle X|+|Y'\triangle Y|=|X''\setminus X|+|X''\setminus X'|+|Y''\setminus Y|+|Y''\setminus Y'|=2\left(|X''\setminus X|+|Y''\setminus Y|\right)\leq 4\delta t^{-1}|X|,$$
and
$|tX'+(1-t)Y'|\leq |tX''+(1-t)Y''|=|tX+(1-t)Y|\leq (1+\delta)|X|.$

Finally, if $X\subset Y$, note that $|Y\setminus X|=0$. Constructing $X'$ as before and setting $Y'=X'\cup (Y\setminus X)$ concludes the proof.
\end{proof}

\subsubsection{Proof of \Cref{prop_boundedmaker}}

\begin{proof}[Proof of \Cref{prop_boundedmaker}]
We begin the proof with the following two claims.
\begin{clm}\label{calim_boundedmaker_1}
Let $\alpha,\eta\in(0,1)$, $P\subset\mathbb{R}^n$ convex and $X,Y$, so that $o\in P \subset X,Y \subset (1+\eta)P$  and let $o'\in (1-\alpha)P$. Let $P'=P-o'$, $X'=X-o'$ and $Y'=Y-o'$. Then $o \in P' \subset X', Y' \subset (1+\eta/\alpha)P'$.  
\end{clm}
\begin{proof}
The first and second containment is trivial so we focus on the last containment. By hypothesis, it is enough to show
$(1+\eta)P-o' \subset  (1+\eta/\alpha)(P-o').$
This is equivalent to
$(1+\eta)P+ (\eta/\alpha) o' \subset  (1+\eta/\alpha)P.$
Because $o' \in (1-\alpha)P$, it is enough to note that
$(1+\eta)P+  (\eta/\alpha)(1-\alpha)P=  (1+\eta/\alpha)P.$
This concludes the proof of the claim.
\end{proof}

\begin{clm}\label{calim_boundedmaker_2}
Assume that $o \in P$ is a convex set. Let $o'$ be the center of John ellipsoid $E$ of $P$. Then $o' \in (1-\frac{1}{n})P$.
\end{clm}

\begin{proof}
John ellipsoid has the following property.
$(E-o') \subset (P-o') \subset n(E-o').$
In particular, by symmetry of $E$, we have
$o-o' \in n(E-o')=-n(E-o'), $
which, is equivalent to
$ o' \in \frac{1}{n+1} o + \frac{n}{n+1}E.$
As $o\in P$ and $E\subset P$, this implies 
$o' \in \frac{n}{n+1}P.$
\end{proof}

We now return to the proof of \Cref{prop_boundedmaker}. Let $P\subset\mathbb{R}^n$ convex $o \in P \subset X,Y \subset (1+\eta)P$ using that $X,Y$ is an $\eta$ sandwich. Let $o'$ be the center of the John ellipsoid $E$ in $P$. Let $P'=P-o'$, $E'=E-o'$, $X'=X-o'$ and $Y'=Y-o'$.

By \Cref{calim_boundedmaker_2}, we get $o' \in (1-\frac{1}{n})P$. By  \Cref{calim_boundedmaker_1} we get $o \in P' \subset X', Y' \subset (1+n\eta)P'$. In particular, we deduce that $X',Y'$ is a $n\eta$-sandwich.

By construction, we get that $E'$ is the John ellipsoid of $P'$ and it is centered at $o$. It has the property that $E' \subset P' \subset nE'$. It follows that $E' \subset X',Y' \subset (n+n^2\eta)E'$.

As $E'$ is an ellipsoid centered at the origin $o$, there exists a linear transformation $\theta \colon \mathbb{R}^n \rightarrow \mathbb{R}^n$ such that $\theta(E')$ is a ball centered at the origin. In particular, we get that $\theta(E') \subset \theta(X'),\theta(Y') \subset (n+n^2\eta)\theta(E').$

By observing that the John ellipsoid of the regular simplex $S$ is a ball centered at the origin, we immediately get that $rS \subset \theta(E') \subset nrS$ for some $r>0$. Putting together all of the above, we get that $$rS \subset \theta(X'),\theta(Y') \subset (n^2+n^3\eta)rS.$$

We conclude that $\theta(X-o'), \theta(Y-o')$ is $(n^2+n^3\eta)$-bounded. As sandwiches are preserved under linear transformations, we also conclude that $\theta(X-o'), \theta(Y-o')$ is a $n\eta$-sandwich.
\end{proof}

\section{Outline of the proof of the Quadratic Theorem (\Cref{main_thm_6})}

In this section, we give an outline of the results contained in the next  four sections.
We shall assume that \Cref{LinearThmGeneral} has been proved. 

Thanks to the  simple reduction performed in the previous section, see \Cref{InitialReductionSec} and specifically \Cref{boundedsandwichreduction}, we can assume that $A$ and $B$ are already pretty convex, are sandwiched between two balls of comparable sizes, and are the finite union of axis aligned cubes. The strategy is now the following.

\begin{enumerate}
    \item Moving $A$ and $B$ slightly we can partition $\mathbb{R}^n$ into $n+1$ reasonably shaped (not too large or small) convex cones $C\in\mathcal{C}_1$ so that $|A\cap C|=|B\cap C|$ (see \Cref{lem_first_step}).
    \item We will refine this partition into cones using the following procedure. Given a cone $C$ and a codimension-two subspace $S$, we find a hyperplane $H\supset S$ so that $|C\cap H^\pm\cap A|=|C\cap H^\pm\cap B|$ (see  \Cref{lem_second_step}).
    \item Choosing the codimension-two subspaces carefully, we obtain a partition into convex cones $C\in\mathcal{C}_2$ essentially all of which satisfy the following properties (see \Cref{prop_fourandahalfth_step}, the engine of which is \Cref{prop_fourth_step}):
    \begin{itemize}
    \item $|A\cap C|=|B\cap C|$,
    \item $C$ is the convex hull of few half-lines through the origin, 
    \item $C$ is very narrow in all but one direction of an orthogonal basis depending on $C$.
\end{itemize}
These imply $\sum_{C\in \mathcal{C}_2} |t(A\cap C)+(1-t)(B\cap C)|\leq |tA+(1-t)B|$.
    \item Removing a negligible part of $A$ and $B$, we may additionally assume that the sections of $C$ in the narrow directions are completely contained in or disjoint from $A\cap C$ and $B\cap C$ for $C\in \mathcal{C}_2$ (see \Cref{lem_fifth_step}). In some sense this reduces $A\cap C$ and $B\cap C$ to two-dimensional sets, as all information about the sets is captured by $\pi(A\cap C)$ and $\pi(B\cap C)$, where $\pi$ is the projection along the narrow directions.
    \item In \Cref{cylindercoveringlem} we construct a bounded family $\mathcal{U}$ of  cylinders. All cylinders have the same simplex base, which is contained in a face of $T$ and inside $C$. Moreover, the cylinders cover a big ball intersected with $C$ (so, in particular, $A\cap C$ and $B\cap C$). Furthermore, for $U \in \mathcal{U}$, we have $|U\cap A\cap C| =\Omega_n(1)|A\cap C|$. 
    \item In \Cref{matchingcylinderslem}, for $U \in \mathcal{U}$, we find a matching cylinder $V$ parallel to $U$ such that $|U\cap A\cap C|= |V\cap B\cap C|$ and $|t(U\cap A\cap C)+(1-t)(V\cap A\cap C)|-|U\cap A\cap C| \leq |t(A\cap C)+(1-t)(B\cap C)|-|A\cap C|$. Moreover, the distance between $U$ and $V$ is small, namely $|(U \triangle V) \cap S \cap C| \leq O_n(1)\delta^{1/2}t^{-1/2} |S \cap C| $. This allows us to assume $U=V$ and further reduce the problem to the sets $A\cap U\cap C$ and $B\cap V\cap C$ which satisfy $|t(U\cap A\cap C)+(1-t)(V\cap A\cap C)|=(1+\delta')|U\cap A\cap C|$.
    \item  By \Cref{LinearThmGeneral}, we find $|\co(U\cap A\cap C)  \setminus (U\cap A \cap C)|=O_{n,t}(\delta')|(U\cap A \cap C)|$. We also know that that the sections of $C$ in the narrow directions are completely contained in or disjoint from $A\cap C$ and $B\cap C$. Combining these, we can deduce that $U\cap A\cap C$ and $V\cap B\cap C$ are essentially one-codimensional compressed.
    \item  In \Cref{symdiftubes} we resolve the problem for one-codimensional (in the long direction) compressed sets.
    
\end{enumerate}

\section{Intermediate results for the Quadratic Theorem (\Cref{main_thm_6}): Part I}

\subsection{Setup}

\begin{defn}
Let $\mathcal{K}^n$ be the family of convex sets in $\mathbb{R}^n$ with a finite number of vertices and let $\mathcal{S}_k^n$ be the set of codimension $k$ affine subspaces of $\mathbb{R}^n$.
\end{defn}

\begin{defn}
Say a function $f \colon \mathcal{K}^{n} \times \mathcal{S}^{n}_2 \rightarrow \mathcal{S}^{n}_1$ is a \emph{respectful} function if $L\subset f(K,L)$. A respectful function $f$ induces functions $f^-, f^+ \colon \mathcal{K}^{n} \times \mathcal{S}^{n}_2 \rightarrow \mathcal{K}^{n}$, where $f^-(K,L), f^+(K,L)$ are the convex sets the affine hyperplane $f(K,L)$ essentially partitions $K$ into.
\end{defn}

\begin{defn}
Given a respectful function $f \colon \mathcal{K}^{n} \times \mathcal{S}^{n}_2 \rightarrow \mathcal{S}^{n}_1$ and a convex set $P$, we say $\mathcal{F}$ is a \emph{valid partition of $P$} into convex subsets if there exists a sequence of families $\{P\}=\mathcal{G}_0, \dots, \mathcal{G}_j=\mathcal{F}$ such that if $\mathcal{G}_i=\{P_{1}, P_{2}, \dots, P_{k}\}$, then there exists codimension-two affine subspaces $L_1, \dots, L_{k}$ such that $\mathcal{G}_{i+1}=\bigcup_{j=1}^k R_j$, where $R_j=\{f^+(P_{j}, L_j), f^-(P_{j}, L_j)\}$ or $R_j=\{P_j\}$.
\end{defn}

We now consider the analogous definitions for cones (cf. \Cref{defn_cone}).

\begin{defn}
Let $\mathcal{C}^n$ be the family of cones  in $\mathbb{R}^n$ and let $\mathcal{T}_k^n$ be the set of codimension $k$ subspaces of $\mathbb{R}^n$.
\end{defn}

\begin{defn}
Say a function $f \colon \mathcal{C}^{n} \times \mathcal{T}^{n}_2 \rightarrow \mathcal{T}^{n}_1$ is a \emph{respectful} function if $L\subset f(C,L)$. A respectful function $f$ induces functions $f^-, f^+ \colon \mathcal{C}^{n} \times \mathcal{T}^{n}_2 \rightarrow \mathcal{C}^{n}$, where $f^-(C,L), f^+(C,L)$ are the cones the hyperplane $f(C,L)$ partitions $C$ into.
\end{defn}

\begin{defn}
Given a respectful function $f \colon \mathcal{C}^{n} \times \mathcal{T}^{n}_2 \rightarrow \mathcal{T}^{n}_1$ and a cone $C$, we say $\mathcal{F}$ is a \emph{valid partition of $C$} into cones if there exists a sequence of families $\{C\}=\mathcal{G}_0, \dots, \mathcal{G}_j=\mathcal{F}$ such that if $\mathcal{G}_i=\{C_{1}, C_{2}, \dots, C_{k}\}$, then there exists codimension-two subspaces $L_1, \dots, L_{k}$ such that $\mathcal{G}_{i+1}=\bigcup_{j=1}^k R_j$, where $R_j=\{f^+(C_{j}, L_j), f^-(C_{j}, L_j)\}$ or $R_j=\{C_j\}$.
\end{defn}

\begin{defn}
Given a cone $C\subset\mathbb{R}^n$, define $$\mu_n(C):=|C\cap S|,$$
where $S$ is the unit volume simplex from \Cref{defn_simplex}.
\end{defn}

\begin{defn}\label{linearfunctiondef}
A function $f:C\to \mathcal{K}^n$ from some convex domain $C\subset\mathbb{R}^m$ is \emph{linear} if $f(tx+(1-t)y)=tf(x)+(1-t)f(y)$ for all $x,y\in C$.
\end{defn}

For example, the function $f \colon \mathbb{R}_{\geq 0}^2 \to \mathcal{K}^2$ given by $f(x,y)= [x/2,3x/2] \times [y,3y]$ is linear. Note that given a linear function $f\colon C\to \mathcal{K}^n$ with $C\in \mathcal{C}^m$ and $f(o)=\{o\}$, the set $\bigcup_{x\in C}f(x)\times \{x\}$ is a cone in $\mathcal{C}^{m+n}$.

\begin{defn}\label{defn_cones_basis}
A cone $C \in \mathcal{C}^n $ is \emph{$(i,\ell, \varepsilon)$-good} if there exists a basis $e_1^C, \dots, e_n^C$ and there exists a cone $C' \in \mathcal{C}^i$ and a linear function $f\colon C' \rightarrow  \mathcal{K}^{n-i}$ such that $f(x)$ has at most $\ell$ vertices for all $x\in C',$ $$\sup_{y,z\in f(x)}||y-z||_2\leq \varepsilon||x||_2,\qquad \text{and}\qquad C=\cup_{x \in C'}  \{f(x) \times \{x\} \}.$$
\end{defn}

\begin{defn}\label{pidefn}
When a basis is established, let $\pi_i:\mathbb{R}^n\to \mathbb{R}$ be the projection onto the $i$th coordinate and let $\pi_{i,j}:\mathbb{R}^n\to \mathbb{R}^2$ be the projection onto the plane spanned by the $i$th and $j$th coordinate.
\end{defn}

\subsection{Theorem}
In this section, we will prove the following theorem. Recall from \Cref{defn_simplex} that $S$ is a regular simplex with unit volume. Let $F_0, \dots, F_n$ be the faces of $S$ defined by $F_i=C_i\cap \partial S$ where $C_i \in \mathfrak{C}$.

\begin{thm}
\label{prop_fourth_step}
There exist constants $\ell^{\ref{prop_fourth_step}}_{n}(m)$ such that for every $\varepsilon>0$ the following holds.  Given a respectful $f \colon \mathcal{C}^{n} \times \mathcal{T}^{n}_2 \rightarrow \mathcal{T}^{n}_1$ and a cone $C$ defined by $m$ lines which is a subcone of some $C_i\in\mathfrak{C}$, there exists a valid partition $\mathcal{F}$ of $C$ that can be written as $\mathcal{F}=\mathcal{F}_0 \sqcup \mathcal{F}_1 \sqcup \mathcal{F}_2$ such that 
\begin{enumerate}
    \item $\sum_{F \in \mathcal{F}_0} \mu_n(F) \leq \varepsilon.$
    \item Every cone $F \in \mathcal{F}_1$ is $(1,\ell^{\ref{prop_fourth_step}}_n(m), \varepsilon)$-good. 
    \item For every cone $F \in \mathcal{F}_2$ there exists a sub-cone $F'$ of $F$ with $\mu_n(F') \geq (1-\varepsilon)\mu_n(F)$ such that $F'$ is $(2,\ell^{\ref{prop_fourth_step}}_n(m), \varepsilon)$-good. 
\end{enumerate}
Furthermore, given $H+v$ the affine hyperplane containing $F_i$, where $H$ is a hyperplane through the origin, we can insist that for all cones $F \in \mathcal{F}_1 $ we have $e_{n}^F \in H^{\perp}$ and for all $F \in\mathcal{F}_2$ we have $e_{n}^{F'} \in H^{\perp}$.
\end{thm}

We note that a slightly more general result holds where we drop $S$. The motivation to include $S$ comes from the way we apply the theorem.

\subsection{Propositions}
To prove \Cref{prop_fourth_step}, we first state some propositions and lemmas that will be used in the proofs. All these results will be proved later below.

\begin{prop}
\label{prop_third_step2D}
There exists a constant $\ell_2^{\ref{prop_third_step2D}}$ such that for every $\varepsilon>0$ the following holds. Given a respectful $f \colon \mathcal{K}^{2} \times \mathcal{S}^{2}_2 \rightarrow \mathcal{S}^{2}_1$ and convex set $P$, there exists a valid partition $\mathcal{F}$ of $P$ that can be written as $\mathcal{F}=\mathcal{F}_0 \sqcup \mathcal{F}_1 \sqcup \mathcal{F}_2$ such that
\begin{enumerate}
    \item $\sum_{F \in \mathcal{F}_0} |F| \leq \varepsilon.$
    \item For every $F \in \mathcal{F}_1$ we have $|V(F)| \leq \ell^{\ref{prop_third_step2D}}_{2}$ and for a basis $e_1, e_2$ and for $i=1,2$ we have that $|\pi_i(F)| \leq \varepsilon$,
    \item For every $F \in \mathcal{F}_2$ we have $|V(F)| \leq \ell^{\ref{prop_third_step2D}}_{2}$ there exists a basis (dependent on $F$) $e_1, e_2$ such that we have $|\pi_1(F)| \leq \varepsilon$ and $|\pi_{2}(F)| \geq \varepsilon$ and $\pi_{2}(V(F))) \subset V(\pi_{2}(F))+(-\varepsilon^2, \varepsilon^2) $.
    
Here $\pi_i \colon \mathbb{R}^n \to \mathbb{R}$ is the projection onto the $i$-th coordinate (as in \Cref{pidefn}).
\end{enumerate}
\end{prop}

\begin{prop}
\label{prop_third_step}
There exists a constant $\ell^{\ref{prop_third_step}}_n(m)$ such that for every $\varepsilon>0$ the following holds. Given a respectful $f \colon \mathcal{K}^{n} \times \mathcal{S}^{n}_2 \rightarrow \mathcal{S}^{n}_1$ and convex set $P$ with $|V(P)|\leq m$, there exists a valid partition $\mathcal{F}$ of $P$ that can be written as $\mathcal{F}=\mathcal{F}_0 \sqcup \mathcal{F}_1 \sqcup \mathcal{F}_2$ such that
\begin{enumerate}
    \item $\sum_{F \in \mathcal{F}_0} |F| \leq \varepsilon.$
    \item For every $F \in \mathcal{F}_1$ we have $|V(F)| \leq \ell^{\ref{prop_third_step}}_{n}(m)$ and for a basis $e_1, e_2, \dots, e_{n}$ we have for every $i$  $|\pi_i(F)| \leq \varepsilon$.
    \item For every $F \in \mathcal{F}_2$ we have $|V(F)| \leq \ell^{\ref{prop_third_step}}_{n}(m)$ and there exists a basis (dependent on $F$) $e_1, e_2, \dots, e_{n}$ such that for every $i \neq n$ we have $|\pi_i(F)| \leq \varepsilon$ and $|\pi_{n}(F)| \geq \varepsilon$ and $\pi_{n}(V(F))) \subset V(\pi_{n}(F))+(-\varepsilon^2, \varepsilon^2) $.
\end{enumerate}
\end{prop}

\subsection{Auxiliary Lemmas}

\begin{lem}\label{lem_side_note}
There exist constants $\lambda^{\ref{lem_side_note}}_n>0$ so that given a convex set $P\subset\mathbb{R}^n$ and an affine subspace $L'\in\mathcal{S}^n_2$, there exists a translate $L\in\mathcal{S}_2^n$ of $L'$ so that any affine hyperplane $H\supset L$ essentially partitioning the space into two parts $H^+$ and $H^-$ satisfies
$$\frac{|H^+\cap P|}{|H^-\cap P|}\in \left[\lambda^{\ref{lem_side_note}}_n,(\lambda^{\ref{lem_side_note}}_n)^{-1}\right].$$
\end{lem}

\begin{lem}\label{lem_few_vertices}
For any convex body $P\subset\mathbb{R}^2$ with at least 7 vertices, there exists a point $p\in P$, so that for any line $\ell\ni p$ essentially partitioning the plane into halfplanes $\ell^+$ and $\ell^-$, we have that both convex sets $\ell^+\cap P$ and $\ell^-\cap P$ have fewer vertices than $P$.
\end{lem}

\begin{lem}\label{ReducingVerticesLem}
Let $P\subset \mathbb{R}^2$ be a convex set with at most $C$ vertices and a respectful function $f:\mathcal{K}^2\times \mathbb{R}^2\to \mathcal{S}_1^2$. Then there exists a valid partition of $P$ into at most $2^{C-6}$ parts so that all parts have at most 6 vertices.
\end{lem}

\begin{lem}\label{smallanglesmallset_lem}
Let $P\subset\mathbb{R}^n$ be a convex set. For any $\xi>0$, there exist $\zeta_0,\eta_0>0$ (depending on $P$) so that for any $\zeta<\zeta_0$ and $\eta<\eta_0$ the following holds. Say a line $L\subset\mathbb{R}^n$ is \emph{$\zeta$ permissible} if for one of the points $x\in L\cap\partial P$, there exists a line $L_x$ tangent to $P$ at $x$ with $\angle L,L_x\leq \zeta$. Let $$Q_{\zeta,\eta}:=P\cap \bigcup_{L:\zeta \text{ permissible}} L+B(o,\eta),$$
then $Q_{\zeta,\eta}\subset \partial P+ B(o,\xi)$
\end{lem}

\begin{lem}\label{radiusormiddleverticesLem}
There exists a constant $\alpha>0$, so that for any $\varepsilon>0$ the following holds. Given a basis $e_1, e_2$ in $\mathbb{R}^2$, a convex set $P\subset\mathbb{R}^2$ with at most six vertices and a respectful $f:\mathcal{K}^2\times \mathbb{R}^2\to \mathcal{S}_1^2$, there exists a valid partition $\mathcal{G}$ of $P$ and an element $P'\in\mathcal{G}$ so that $|P'|\geq \alpha |P|$ and
\begin{itemize}
\item either $|\pi_2(P')|\leq|\pi_2(P)|-\varepsilon^2$,
\item or $\pi_2(V(P'))\subset V(\pi_2(P'))+(-\varepsilon^2,\varepsilon^2).$
\end{itemize}
Furthermore, all convex sets in $\mathcal{G}$ have at most 10 vertices.
\end{lem}

\subsection{Proof of \Cref{prop_fourth_step}}
\Cref{prop_fourth_step} can be considered as \Cref{prop_third_step} coned off at the origin. 

\begin{proof}[Proof of \Cref{prop_fourth_step}]
Let $\ell^{\ref{prop_fourth_step}}_n(m):=(\ell^{\ref{prop_third_step}}_{n-1}(m))^2$. Let $c^{\ref{prop_fourth_step}}_n$ be $1/n+1$ times the radius of the largest ball inside $T$. Choose $\eta>0$ sufficiently small in terms of $\varepsilon$. 

Let $P:=(v+H)\cap C\in\mathcal{K}^{n-1}$. For any subset $P'\subset v+H$, we define $C(P'):=\bigcup_{t\geq 0}tP'$.  Note that if $P'$ is a bounded convex set in $H+v$, then $C(P')$ is a cone. In particular,  $C=C(P)$. Note moreover that for a codimension $2$ affine subspace $L$ of $v+H$, $C(L)$ is a codimension $2$ subspace of $\mathbb{R}^n$.

In the following we will slightly abuse notation to interpret $v+H$ as a copy $\mathbb{R}^{n-1}$ to aid the application of \Cref{prop_third_step}. For instance, we will write $\mathcal{K}^{n-1}$ to indicate the convex subsets of $v+H$.

We construct a respectful function $f':\mathcal{K}^{n-1}\times\mathcal{S}^{n-1}_2\to\mathcal{S}^{n-1}_1$ from $f:\mathcal{C}^{n}\times\mathcal{T}^{n}_2\to\mathcal{T}^{n}_1$ as follows. For any convex subset $P'\subset P$ and a codimension $2$ affine subspace $L$ of $v+H$, define
$f'(P',L)=(v+H) \cap f(C(P'),C(L))$. Note that indeed 
$$L= (v+H)\cap C(L)\subset (v+H) \cap f\left(C(P'),C(L)\right)=f'(P',L),$$
so that $f'$ is respectful.

Now apply \Cref{prop_third_step} to $P$ with respectful function $f'$ and parameter $\eta$ (as $\varepsilon$), to find valid partition $\mathcal{F}'=\mathcal{F}'_0\cup\mathcal{F}'_1\cup\mathcal{F}'_2$ of $P$, where all $P'\in\mathcal{F}'_1\cup \mathcal{F}'_2$ have at most $\ell^{\ref{prop_third_step}}_{n-1}(m)$ vertices. Define $\mathcal{F}=\mathcal{F}_0\cup\mathcal{F}_1\cup\mathcal{F}_2$ by $\mathcal{F}_i:=\{C(P'):P'\in\mathcal{F}'_i\}$. Note that clearly $\mathcal{F}$ is a valid partition by the construction of $f'$, so it remains to check that it indeed satisfies the conditions on $\mathcal{F}_0$, $\mathcal{F}_1$, and $\mathcal{F}_2$. 

By definition of $H$, we find that $\mu_n(C(P'))=c^{\ref{prop_fourth_step}}_n|P'|$, so that 
\begin{align*}\sum_{F\in\mathcal{F}_0}\mu_n(F)=c^{\ref{prop_fourth_step}}_n\sum_{P'\in\mathcal{F}'_0}|P'|\leq c^{\ref{prop_fourth_step}}_n\eta\leq \varepsilon,\end{align*}
as $c^{\ref{prop_fourth_step}}_n$ represents $1/(n+1)$ of the distance from $o$ to $F_i$.

Consider $F\in\mathcal{F}_1$ and choose a basis $e_1,\dots, e_n$ with $e_n$ perpendicular to $H$. Let $P'\in\mathcal{F}_1'$, so that $F=C(P')$. To show that $F$ is $(1,\ell,\varepsilon)$-good, let $C'=\mathbb{R}_{\geq0}e_n$ and $g:C'\to\mathcal{K}^{n-1}, te_n\mapsto (tv+H)\cap F=tP'$, which is clearly linear. Note that $g(te_n)$ is homothetic for all $t$, so always has the same number of vertices (as $P'$), in particular at most $\ell^{\ref{prop_third_step}}_{n-1}(m)\leq \ell^{\ref{prop_fourth_step}}_n(m)$. Note
$$\sup_{x,y\in g(te_n)}||x-y||_2= \sup_{x,y\in tP'}||x-y||_2\leq t\sqrt{n-1}\eta\leq \varepsilon ||te_n||_2.$$
This concludes that $F$ is $(1,\ell,\varepsilon)$-good.

Consider $F\in\mathcal{F}_2$ and the $P'\in\mathcal{F}_2'$, so that $F=C(P')$. Let $e_n$ be the unit vector orthogonal to $H$, and let $e_1,\dots,e_{n-1}$ be the basis of $H$ so that $|\pi_i(P')|\leq \eta$ for $i\leq n-2$ and $|\pi_{n-1}(P')|\geq \eta$ and $\pi_{n-1}(V(P'))\subset V(\pi_{n-1}(P'))+(-\eta^2,\eta^2)$. 
Let 
$$P'':=P'\cap \pi_{n-1}^{-1}\left(\pi_{n-1}(P')\setminus \left[V(\pi_{n-1}(P'))+\left(-\eta^2,\eta^2\right)\right]\right),$$
so that by construction $\pi_{n-1}(V(P''))=V(\pi_{n-1}(P''))$. By convexity, we have $|P'\setminus P''|\leq O_n(\eta|P'|)$. Now let $F':=C(P'')\subset F$, and note that $\mu_n(F')\geq (1-\varepsilon)\mu_n(F)$. Let $C':=C(\pi_{n-1,n}(F'))=\pi_{n-1,n}(F')$, where $\pi_{n-1,n}\colon\mathbb{R}^n\to\mathbb{R}^2$ is the projection onto the last two coordinates and let $$g\colon C'\to\mathcal{K}^{n-2}, x\mapsto \pi_{n-1,n}^{-1}(x)\cap F'.$$ To see that $g$ is linear, consider the auxiliary function $$g'\colon\pi_{n-1}(P'')\to \mathcal{K}^{n-2}, x\mapsto \pi_{n-1}^{-1}(x)\cap P''$$ and note that $g'$ is linear because $\pi_{n-1}(V(P''))=V(\pi_{n-1}(P''))$. Let $h>0$ be such that $\pi_n(v)=h$ and note that $\pi_{n}^{-1}(th)\cap F'=tP''$, so that for a point $(x,y)\in \pi_{n-1,n}(F'')$, we have $g(x,y)=\frac{y}{h}  g'(\frac{hx}{y})$. Hence, we get for $(x,y),(x',y')\in C'$, that
\begin{align*}
h g\left(tx+(1-t)x',ty+(1-t)y'\right)&=(ty+(1-t)y')g'\left(\frac{tx+(1-t)x'}{ty+(1-t)y'}h\right)\\
&=(ty+(1-t)y')\left(g'\left(\frac{ty}{ty+(1-t)y'}\cdot \frac{x}{y}h+\frac{(1-t)y'}{ty+(1-t)y'}\cdot \frac{x'}{y'}h\right)\right)\\
&=(ty+(1-t)y')\left(\frac{ty}{ty+(1-t)y'} g'\left(\frac{x}{y}h\right)+\frac{(1-t)y'}{ty+(1-t)y'} g'\left(\frac{x'}{y'}h\right)\right)\\
&=ty g'\left(\frac{x}{y}h\right)+(1-t)y' g'\left(\frac{x'}{y'}h\right)\\
&=thg(x,y)+(1-t)hg(x',y'),
\end{align*}
so that $g$ is linear. Finally, to count the number of vertices of $g(x,y)$, note that it is the intersection between a homothetic copy of $P''$ and a hyperplane, so that the number of vertices is bounded by the number of edges in $P''$ which is at most $\binom{\ell^{\ref{prop_third_step}}_{n-1}(m)}{2}\leq \ell^{\ref{prop_fourth_step}}_n(m)$.
\end{proof}

\subsection{Proof of Propositions}

\subsubsection{Proof of \Cref{prop_third_step2D}}
Note that $\mathcal{S}^2_2=\mathbb{R}^2$.

\begin{proof}[Proof of \Cref{prop_third_step2D}]
Choose $\ell_2^{\ref{prop_third_step2D}}=10$.
We first show that we may assume that $|\pi_1(P)|\leq \varepsilon$, per the following claim.
\begin{clm}
The result for convex sets $P$ with $|\pi_1(P)|\leq \varepsilon$ implies the result for general convex $P$.
\end{clm}
\begin{proof}[Proof of claim]
Iteratively produce a valid partition of $P$, starting with $\mathcal{G}_0=\{P\}$. Given $\mathcal{G}_i=\{P_1,\dots,P_{2^i}\}$, by \Cref{lem_side_note}, find points $p_j\in P_j$ and construct
$$\mathcal{G}_{i+1}:=\{f^{+}(P_1,p_1),f^{-}(P_1,p_1),\dots,f^{+}(P_{2^i},p_{2^i}),f^{-}(P_{2^i},p_{2^i})\}.$$
Now note that by construction $\max\{|P'|:P'\in\mathcal{G}_{i+1}\}\leq (1-\lambda_2)\max\{|P'|:P'\in\mathcal{G}_{i}\},$
where $\lambda_2>0$ is the constant from \Cref{lem_side_note}. Hence, for $i_0$ sufficiently large, we find that
$\max\{|P'|:P'\in\mathcal{G}_{i_0}\}< \frac{\varepsilon^2}{8}$. 
This implies that for all $P'\in\mathcal{G}_{i_0}$ there exists a basis $e_1,e_2$ so that $|\pi_1(P')|\leq \varepsilon$. Indeed, let $E'$ be the  outer Lowner-John ellipsoid of $P'$ and let $e_1, e_2$ be the short and long axis of $E'$. Hence, $|\pi_1(E')|^2/2\leq |\pi_1(E')| |\pi_2(E')|/2 \leq |E'| \leq 4 |P| \leq \frac{\varepsilon^2}{2}$. Therefore, $|\pi_1(P')|\leq |\pi_1(E')| \leq \varepsilon$. Now applying the result in each $P'\in\mathcal{G}_{i_0}$, we deduce  the result in $P$.
\end{proof}

Henceforth, assume $\pi_1(P)\leq \varepsilon$. We iteratively construct a sequence of valid partitions $\mathcal{G}_i$ and $\mathcal{G}_i'$ of $P$, starting with $\mathcal{G}_0=\{P\}$.

Given $\mathcal{G}_{i}$, apply \Cref{ReducingVerticesLem} to all elements of $\mathcal{G}_{i}$ to find refinement $\mathcal{G}_i'$ in which all elements have at most 6 vertices. Then apply \Cref{radiusormiddleverticesLem} to each of the elements of $\mathcal{G}_i'$ to produce refinement $\mathcal{G}_{i+1}$.
Given this sequence we will construct a valid partition of $P$ with the desired properties. Consider the uniform probability measure $\PP$ on $P$, so that $\sum_{P'\in\mathcal{G}_i}\PP(P')=1$ for all $i$. we are going to analyse the change in expected value of the following parameter. Let $f_0(P)=|\pi_2(P)|$. Given $f_i:\mathcal{G}_i\to\mathbb{R}$, construct $f_{i+1}:\mathcal{G}_{i+1}\to\mathbb{R}$ as follows. 
$$f_{i+1}(P'):=\begin{cases}0 &\text{ if $P'\subset P''\in\mathcal{G}_i$ with $f_i(P'')=0$, or if $\pi_2(V(P'))\subset V(\pi_2(P'))+(-\varepsilon^2,\varepsilon^2)$}\\
|\pi_2(P')| &\text{ otherwise.}
\end{cases}$$
Note that if $|\pi_2(P')|<2\varepsilon^2$, then $f_i(P')=0$ by the second clause, so if $f_i(P')\neq0$, then $f_i(P')\geq 2\varepsilon^2$. In terms of this function, we construct our families $\mathcal{F}^i_0,\mathcal{F}^i_1,\mathcal{F}^i_2,$ as follows;
$\mathcal{F}^0_0=\{P\},\ \mathcal{F}^0_1=\emptyset,\ \mathcal{F}^0_2=\emptyset,$
and
\begin{itemize}
    \item $\mathcal{F}^{i+1}_0=\{P'\in\mathcal{G}_{i+1}: f_{i+1}(P')>0\}$
    \item $\mathcal{F}^{i+1}_1=\mathcal{F}^{i}_1\cup \{P'\in\mathcal{G}_{i+1}: f(P')=0, |\pi_2(P')|\leq \varepsilon, \text{ and }\exists P''\in \mathcal{F}^i_0, P'\subset P''\}$
    \item $\mathcal{F}^{i+1}_2=\mathcal{F}^{i}_2\cup \{P'\in\mathcal{G}_{i+1}: f(P')=0, |\pi_2(P')|> \varepsilon, \text{ and } \exists P''\in \mathcal{F}^i_0, P'\subset P''\}$
\end{itemize}
It is easy to see that $\mathcal{F}^i=\mathcal{F}_0^i\sqcup \mathcal{F}_1^i \sqcup \mathcal{F}_2^i$ is a valid partition of $P$. We will show that for $i$ sufficiently large $\mathcal{F}^i$ satisfies the desired the conditions. We use shorthand $\PP\left(\mathcal{F}_0^{i}\right)$ for $\PP\left(\bigcup_{P'\in\mathcal{F}_0^{i}}P'\right)=\sum_{P'\in\mathcal{F}_0^{i}}\PP(P')$.
\begin{clm}
$\PP\left(\mathcal{F}_0^{i}\right)\to 0$ as $i\to \infty$.
\end{clm}
\begin{proof}[Proof of claim]
First note that $f_i$ is non-negative and non-increasing, in the sense that if $P'\in\mathcal{G}_i$ and $P''\in\mathcal{G}_j$ with $P'\subset P''$, then $f_j(P'')\geq f_i(P')$. Partition $\mathcal{G}_i$ into two parts $\mathcal{F}_0^{i}$ and  $\mathcal{H}_i:=\mathcal{G}_i\setminus\mathcal{F}_0^{i}$ (so that $\mathcal{H}_i$ is a refinement of $\mathcal{F}_1^{i}\cup\mathcal{F}_2^{i}$). Consider some $P'\in\mathcal{F}_0^{i}$ (i.e., with $f_i(P')>2\varepsilon^2$) and the valid partition $\{P_1,\dots,P_j\}\subset\mathcal{G}_i'$ of $P'$. Note that by \Cref{radiusormiddleverticesLem}, for every $1\leq k\leq j$, we can find $P_k\supset P'_k\in\mathcal{G}_{i+1}$ with $\PP(P'_k)\geq \alpha \PP(P_k)$ and
$$\text{either }\quad|\pi_2(P_k')|\leq |\pi_2(P_k)|-\varepsilon^2,\quad \text{ or }\quad \pi_2(V(P_k'))\subset V(\pi_2(P_k'))+(-\varepsilon^2,\varepsilon^2).$$ Either way, we find
$f_{i+1}(P_k')\leq f_i(P')-\varepsilon^2$.
Hence, if we let $\mathcal{G}_{i+1}(P'):=\{P''\in\mathcal{G}_{i+1}: P''\subset P'\}$, then
$$\sum_{P''\in \mathcal{G}_{i+1}(P')}\PP(P'')f_{i+1}(P'')\leq \PP(P')f_i(P')-\alpha\PP(P')\varepsilon^2. $$
Summing this over all $P'\in\mathcal{F}_0^{i}$ and using induction, we find
$$0\leq \sum_{P''\in \mathcal{G}_{i+1}}\PP(P'')f_{i+1}(P'')\leq \sum_{P'\in\mathcal{G}_i}\PP(P')f_i(P')-\alpha\varepsilon^2\PP\left(\mathcal{F}_0^{i}\right)\leq f_0(P)-\alpha\varepsilon^2 \sum_{j=0}^i\PP\left(\mathcal{F}_0^{i}\right). $$
This implies $\PP\left(\mathcal{F}_0^{i}\right)\to 0$ as $i\to 0$ and thus the conclusion follows.
\end{proof}

By this claim, we can find $i_0$ so that $\PP\left(\mathcal{F}_0^{i_0}\right)\leq \frac{\varepsilon}{|P|},$ thus $\sum_{P'\in\mathcal{F}_0^{i_0}}|P'|=\PP\left(\mathcal{F}_0^{i_0}\right)|P|\leq \varepsilon$. Hence, let $\mathcal{F}_0=\mathcal{F}_0^{i_0},\mathcal{F}_1=\mathcal{F}_1^{i_0}$, and $\mathcal{F}_2=\mathcal{F}_2^{i_0}$.
By construction, we have for every $F\in\mathcal{F}_1$, 
$|\pi_1(F)|,|\pi_2(F)|\leq \varepsilon$ and that, for $F\in\mathcal{F}_2$,

$$|\pi_1(F)|\leq \varepsilon,\  |\pi_2(F)|>\varepsilon,\text{ and }\pi_2(V(F))\subset V(\pi_2(F))+(-\varepsilon^2,\varepsilon^2).$$ Finally, by \Cref{radiusormiddleverticesLem} all parts in $\mathcal{G}_i$ have at most 10 vertices, so in particular so do the parts in $\mathcal{F}_1\sqcup \mathcal{F}_2$.
\end{proof}

\subsubsection{Proof of \Cref{prop_third_step}}

\begin{proof}[Proof of \Cref{prop_third_step}]
Let $\xi\gg \zeta\gg \eta$ be chosen sufficiently small in terms of $\varepsilon$ and $n$ to make various statements throughout the proof.

We first find a valid partition so that all parts are small in all but one direction.

We iteratively construct a sequence of valid partitions $\mathcal{G}^i$ for $i=0,\dots,n$, starting with $\mathcal{G}^0=\{P\}$. Each of the $\mathcal{G}^i$'s can be partitioned into two parts $\mathcal{G}^i_0$, and $\mathcal{G}^i_1$, so that the following hold
\begin{enumerate}
    \item $\sum_{F\in\mathcal{G}^i_0}|F|\leq i\eta$
    \item For every $F\in \mathcal{G}^i_1$, there exists a basis $e_1,\dots, e_n$ so that $|\pi_j(F)|\leq \eta$ for all $1<j\leq i$.
    \item For every $F\in \mathcal{G}^i_0\cup\mathcal{G}^i_1$, we have that $F=P\cap F'$ where $F'$ is the intersection of at most $i\ell_2^{\ref{prop_third_step2D}}$ halfspaces, each of which contains all but two of the basis vectors corresponding to $F$.
\end{enumerate}

Assume that $\mathcal{G}^i_0,\mathcal{G}^i_1$ have been constructed. Fix a $F\in \mathcal{G}^i_1$ and the corresponding basis $e_1,\dots, e_n$. Consider the plane $H$ spanned by $e_1$ and $e_{i+1}$ and the projection $\pi=\pi_{1,i+1}$ onto that plane. Now note that any translate $S$ of the codimension-two subspace of $\mathbb{R}^n$ spanned by $e_2,\dots,e_{i},e_{i+2},\dots,e_n$, has that $\pi(S)$ is a single point in $\mathbb{R}^2$. Hence, the respectful function $f:\mathcal{K}^n\times \mathcal{S}^n_2\to \mathcal{S}_1^n$ corresponds to a respectful function $f':\mathcal{K}^2\times \mathcal{S}^2_2\to \mathcal{S}_1^2$. Indeed, for $X\in \mathcal{K}^2$ and $x\in \mathcal{S}^2_2=\mathbb{R}^2$, let
$$f'(X,x):=\pi(f(\pi^{-1}(X)\cap F,\pi^{-1}(x))).$$

Now apply \Cref{prop_third_step2D} to $\pi(F)$ with respectful function $f'$ and parameter $\eta\cdot \min\left\{1,\frac{1}{\max_{x\in\mathbb{R}^2}|\pi^{-1}(x)\cap F|}\right\},$
to find a valid partition $\mathcal{H}=\mathcal{H}_0\sqcup\mathcal{H}_1\sqcup\mathcal{H}_2$ of $\pi(F)$ so that 
\begin{enumerate}
    \item $\sum_{Q \in \mathcal{H}_0} |Q| \leq \frac{\eta}{\max_{x\in\mathbb{R}^2}|\pi^{-1}(x)\cap F|}.$
    \item For every $Q \in \mathcal{H}_1$ there exists a basis $e'_1, e'_2$ such that for $i=1,2$ we have $|\pi'_i(Q)| \leq \eta$.
    \item For every $Q \in \mathcal{H}_2$ there exists a basis $e'_1, e'_2$ such that we have $|\pi'_1(Q)| \leq \eta$, $|\pi'_{2}(Q)| \geq \eta$ and $\pi'_{2}(V(Q))) \subset V(\pi'_{2}(Q))+(-\eta^2, \eta^2) $.
    \item For every $Q\in \mathcal{H}_1\cup\mathcal{H}_2$,  we have $|V(Q)| \leq \ell^{\ref{prop_third_step2D}}_{2}$, i.e., $Q$ is the intersection of at most $\ell_2$ halfplanes.
\end{enumerate}
Note that 2 and 3, imply the weaker statement that for every $Q\in\mathcal{H}_1\cup\mathcal{H}_2$, there exists a basis $e'_1, e'_2$ such that we have $|\pi'_1(Q)| \leq \eta$. 
The partition $\mathcal{H}$ naturally corresponds to a valid partition $\mathcal{H}'=\mathcal{H}'_0\sqcup\mathcal{H}'_1\sqcup\mathcal{H}'_2$ of $F$, where $\mathcal{H}'_i:=\{\pi^{-1}(Q)\cap F: Q\in\mathcal{H}_i\}$. This is a valid partition by construction of $f'$. The properties of $\mathcal{H}$ translate to the following properties of $\mathcal{H}'$.
\begin{enumerate}
    \item $\sum_{Q \in \mathcal{H}'_0} |Q| \leq \eta.$
    \item For every $Q \in \mathcal{H}_1'\cup\mathcal{H}'_2$ there exists a basis $e'_1, e'_{i+1}$ of the plane spanned by $e_1$ and $e_{i+1}$ such that we have $|\pi'_{i+1}(Q)| \leq \eta$.
    \item For every $Q\in \mathcal{H}'_1\cup\mathcal{H}'_2$,  we have $Q=F\cap Q'$, where $Q'$ is the intersection of at most $\ell^{\ref{prop_third_step2D}}_2$ halfspaces all of which contain $e_2,\dots,e_i,e_{i+2},\dots,e_n$.
\end{enumerate}
Now for $Q$, we choose the basis $e'_1,e_2,\dots,e_i,e_{i+1}',e_{i+2},\dots,e_n$. Given $F\in \mathcal{G}_1^i$, let $\mathcal{H}'_0(F)$, $\mathcal{H}'_1(F)$, and $\mathcal{H}'_2(F)$ be the sets produced here. We define
$$\mathcal{G}^{i+1}_0:=\mathcal{G}^{i}_0\cup \bigcup_{F\in\mathcal{G}^{i}_1}\mathcal{H}'_0(F),\quad \text{ and }\quad\mathcal{G}^{i+1}_1:= \bigcup_{F\in\mathcal{G}^{i}_1}\mathcal{H}'_1(F)\cup\mathcal{H}'_2(F).$$
Note that these satisfy the properties 1, 2 and 3 of $\mathcal{G}^{i+1}_0,\mathcal{G}^{i+1}_1$ set out above.

Now consider the elements of $\mathcal{G}^n_1$. These can only be 'long' in at most one direction, viz $e_1$. Now that we have established this direction, we will repeat essentially the same process to show that the sets are either short in the $e_1$ direction as well, or (most of) the vertices are close to the extremes in the $e_1$ direction.

We continue constructing a sequence of valid partitions $\mathcal{G}^{n+i}$ for $i=1,\dots,n$, starting with $\mathcal{G}^0=\{P\}$. Each of the $\mathcal{G}^{n+i}$'s can be partitioned into three parts $\mathcal{G}^{n+i}_0,\mathcal{G}^{n+i}_1,\mathcal{G}^{n+i}_2$, so that the following hold
\begin{enumerate}
    \item $\sum_{F\in\mathcal{G}^{n+i}_0}|F|\leq n\eta+i\zeta$.
    \item For every $F\in \mathcal{G}^{n+i}_1$, there exists a basis $e_1,\dots, e_n$ so that $|\pi_j(F)|\leq \zeta$ for all $j=1,\dots n$.
    \item For every $F\in \mathcal{G}^{n+i}_2$, there exists a basis $e_1,\dots, e_n$ so that $|\pi_j(F)|\leq \eta$ for all $j=2,\dots n$. Moreover, we have $\pi_1(V(\pi_{1,j}(F)))\subset V(\pi_1(F))+(-2\zeta^2,2\zeta^2)$ for all $j=2,\dots, i$.
    \item For every $F\in \mathcal{G}^{n+i}_1\cup\mathcal{G}^{n+i}_2$, we have that $F=P\cap F'$ where $F'$ is the intersection of at most $(n+i)\ell^{\ref{prop_third_step2D}}_2$ halfspaces, each of which contains all but two of the basis vectors corresponding to $F$.
\end{enumerate}

Assume that $\mathcal{G}^{n+i}_0,\mathcal{G}^{n+i}_1,\mathcal{G}^{n+i}_2$ have been constructed. Fix a $F\in \mathcal{G}^{n+i}_2$ and the corresponding basis $e_1,\dots, e_n$. As before, consider the plane spanned by $e_1$ and $e_{i+1}$ and the projection $\pi=\pi_{1,i+1}$ onto that plane. Now note that any translate $L$ of the codimension-two subspace of $\mathbb{R}^n$ spanned by $e_2,\dots,e_{i},e_{i+2},\dots,e_n$, has that $\pi(L)$ is a single point in $\mathbb{R}^2$. Hence, the respectful function $f:\mathcal{K}^n\times \mathcal{S}^n_2\to \mathcal{S}_1^n$ corresponds to a respectful function $f':\mathcal{K}^2\times \mathcal{S}^2_2\to \mathcal{S}_1^2$. Indeed, for $X\in \mathcal{K}^2$ and $x\in \mathcal{S}^2_2=\mathbb{R}^2$, let
$$f'(X,x):=\pi(f(\pi^{-1}(X)\cap F,\pi^{-1}(x))).$$
Now apply \Cref{prop_third_step2D} to $\pi(F)$ with respectful function $f'$ and parameter $\zeta\cdot \min\left\{1,\frac{1}{\max_{x\in\mathbb{R}^2}|\pi^{-1}(x)\cap F|}\right\},$
to find a valid partition $\mathcal{H}=\mathcal{H}_0\sqcup\mathcal{H}_1\sqcup\mathcal{H}_2$ of $\pi(F)$ so that 
\begin{enumerate}
    \item $\sum_{Q \in \mathcal{H}_0} |Q| \leq \frac{\zeta}{\max_{x\in\mathbb{R}^2}|\pi^{-1}(x)\cap F|}.$
    \item For every $Q \in \mathcal{H}_1$ there exists a basis $e'_1, e'_2$ such that for $i=1,2$ we have $|\pi'_i(Q)| \leq \zeta$.
    \item For every $Q \in \mathcal{F}_2$ there exists a basis $e'_1, e'_2$ such that we have $|\pi'_2(Q)| \leq \zeta$, $|\pi'_{1}(Q)| \geq \zeta$ and $\pi'_{1}(V(Q))) \subset V(\pi'_{1}(Q))+(-\zeta^2, \zeta^2) $.
    \item For every $Q\in \mathcal{H}_1\cup\mathcal{H}_2$,  we have $|V(Q)| \leq \ell^{\ref{prop_third_step2D}}_{2}$, i.e., $Q$ is the intersection of at most $\ell^{\ref{prop_third_step2D}}_2$ halfplanes.
\end{enumerate}
Let $Q$ satisfy property 3. Given the information we already have about $Q$ (viz $|\pi_2(Q)|\leq \eta$), we will show that we can get essentially the same property 3, with basis $e'_1,e'_2$ replaced by $e_1,e_2$.
\begin{clm}
 $|\pi_2(Q)| \leq \zeta/2$, $|\pi_1(Q)| \geq  \zeta/2$ and   $\pi_{1}(V(Q))) \subset V(\pi_{1}(Q))+(-4\zeta^2, 4\zeta^2)$.
\end{clm}
\begin{proof}[Proof of claim]

First note that, up to translation, in the $e_1, e_2$ basis we have
$Q\subset \left[0,|\pi_1(Q)|\right]\times \left[0,|\pi_2(Q)|\right]$. Hence, we get $$\zeta \leq |\pi_1'(Q)| \leq |\pi_1(Q)|+ |\pi_2(Q)| \leq |\pi_1(Q)| + \eta,$$ so $|\pi_1(Q)| \geq \zeta -\eta \geq \zeta/2$. Also, $|\pi_2(Q)| \leq \eta \leq \zeta/2$. 

For the last part, it is enough to show that there exists two vertices $x,y \in V(Q)$ such that for any other vertex $z \in V(Q)$ we have $\min(|xz|, |yz|) \leq 2\zeta^2$. 

It is easy to see that there exist two vertices $x,y \in V(Q)$ such that $V(\pi_1'(Q))=\{\pi_1'(x), \pi_1'(y)\}$. 
We have $|xy|\geq |\pi_1'(x)\pi_1'(y)|= |\pi_1'(Q)| \geq \zeta$. Moreover, $|\pi_1'(x) \pi_1'(y)| \geq \zeta$ and $|\pi_2'(x) \pi_2'(y)| \leq \zeta$ implies that $|\langle \frac{y-x}{|y-x|}, e_1' \rangle| \leq \sin(45^\circ)= 1/ \sqrt{2}$. 

Fix $z \in V(Q)$ and assume that $\pi_1'(z) \in (\pi_1'(x)- \zeta^2, \pi_1'(x)+\zeta^2)$, i.e., $|\pi_1'(x)\pi_1'(z)| \leq \zeta^2$. It is enough to show that $|zx|\leq 2\zeta^2$. Assume for a contradiction that $|zx| \geq 2\zeta^2$. Combining the last two inequalities, we deduce that  $|\langle \frac{z-x}{|z-x|}, e_1' \rangle| \geq \sin(60^\circ)= \sqrt{3}/ 2$. 

As $|\langle \frac{y-x}{|y-x|}, e_1' \rangle| \leq \sin(45^\circ)$ and $|\langle \frac{z-x}{|z-x|}, e_1' \rangle| \geq \sin(60^\circ)$ we deduce that $|\langle \frac{y-x}{|y-x|}, \frac{z-x}{|z-x|} \rangle| \geq \sin(15^\circ)$.

In the triangle $\co(\{x,y,z\})$ the radius $r$ of the inscribed circle has the formula $r=\frac{|\langle {y-x}, {z-x} \rangle|}{|xz|+|yz|+|xy|}$. Using the above, we deduce that $$ r \geq \frac{|xy||xz| \sin(15^\circ)}{2|xy|+2|xz|} = \frac{\sin(15^\circ)}{\frac{2}{|xy|}+\frac{2}{|xz|}} \geq \frac{\sin(15^\circ)}{4} \min(|xy|, |xz|) \geq \frac{ \zeta^2\sin(15^\circ)}{2}.$$ Hence $|\pi_2(Q)| \geq |\pi_2(\co(\{x,y,z\}))| \geq \zeta^2 \sin(15^\circ)$. This gives the desired result as $|\pi_2(Q)| \leq \eta <  \zeta^2 \sin(15^\circ)$. The conclusion follows. \end{proof}

As before, we can translate back to a valid partition $\mathcal{H}'=\mathcal{H}'_0\sqcup\mathcal{H}'_1\sqcup\mathcal{H}'_2$ of $F$, where $\mathcal{H} '_i:= {\pi -1}(Q)\cap F: Q\in\mathcal{H}_i\}$. This is a valid partition by construction of $f'$. The properties of $\mathcal{H}$ translate to the following properties of $\mathcal{H}'$.
\begin{enumerate}
    \item $\sum_{Q \in \mathcal{H}'_0} |Q| \leq \zeta.$
    \item For every $Q \in \mathcal{H}_1'$, there exists a basis $e'_1, e'_{i+1}$ of the plane spanned by $e_1$ and $e_{i+1}$ such that we have $|\pi'_1(Q)|,|\pi'_{i+1}(Q)| \leq \zeta$.
    \item For every $Q \in \mathcal{H}'_2$, we have  $\pi_{1}(V(\pi_{1,i+1}(Q))) \subset V(\pi_{1}(Q))+(-4\zeta^2, 4\zeta^2) $.
    \item For every $Q\in \mathcal{H}'_1\cup\mathcal{H}'_2$,  we have $Q=F\cap Q'$, where $Q'$ is the intersection of at most $\ell^{\ref{prop_third_step2D}}_2$ halfspaces all of which contain $e_2,\dots,e_i,e_{i+2},\dots,e_n$.
\end{enumerate}
Given $F\in \mathcal{G}_2^{n+i}$, let $\mathcal{H}'_0(F)$, $\mathcal{H}'_1(F)$, and $\mathcal{H}'_2(F)$ be the sets produced here. We define $\mathcal{G}^{n+i+1}=\mathcal{G}^{n+i+1}_0\sqcup\mathcal{G}^{n+i+1}_1\sqcup\mathcal{G}^{n+i+1}_2$ as follows:
\begin{align*}
    \mathcal{G}^{n+i+1}_0&:=\mathcal{G}^{n+i}_0\cup \bigcup_{F\in\mathcal{G}^{n+i}_2}\mathcal{H}'_0(F),\qquad
    \mathcal{G}^{n+i+1}_1:=\mathcal{G}^{n+i}_1\cup\bigcup_{F\in\mathcal{G}^{n+i}_2}\mathcal{H}'_1(F), \qquad
    \mathcal{G}^{n+i+1}_2:=\bigcup_{F\in\mathcal{G}^{n+i}_2}\mathcal{H}'_2(F)
\end{align*}
Change the basis to $e'_1,e_2,\dots,e_i,e_{i+1}',e_{i+2},\dots,e_n$ for those parts $Q\in \mathcal{H}'_1(F)$ for some $F \in\mathcal{G}^{n+i}_2$. It is easy to see that this valid partition $\mathcal{G}^{n+i+1}=\mathcal{G}^{n+i+1}_0\sqcup\mathcal{G}^{n+i+1}_1\sqcup\mathcal{G}^{n+i+1}_2$ satisfies the properties as set out before.

Consider the valid partition $\mathcal{G}^{2n}=\mathcal{G}^{2n}_0\sqcup\mathcal{G}^{2n}_1\sqcup\mathcal{G}^{2n}_2$ which has the following properties:
\begin{enumerate}
    \item $\sum_{F\in\mathcal{G}^{2n}_0}|F|\leq 2n\zeta$
    \item For every $F\in \mathcal{G}^{2n}_1$, there exists a basis $e_1,\dots, e_n$ so that $|\pi_i(F)|\leq \zeta$ for all $i=1,\dots n$.
    \item For every $F\in \mathcal{G}^{2n}_2$, there exists a basis $e_1,\dots, e_n$ so that $|\pi_i(F)|\leq \eta$ for all $i=2,\dots n$. Moreover, we have $\pi_1(V(\pi_{1,i}(F)))\subset V(\pi_1(F))+(-4\zeta^2,4\zeta^2)$ for all $i=2,\dots, n$.
    \item For every $F\in \mathcal{G}^{2n}_1\cup\mathcal{G}^{2n}_2$, we have that $F=P\cap F'$ where $F'$ is the intersection of at most $2n\ell^{\ref{prop_third_step2D}}_2$ halfspaces, each of which contains all but one of the basis vectors $e_2,\dots,e_n$ corresponding to $F$.
\end{enumerate}
This last property shows that $F$ is the intersection $2n\ell^{\ref{prop_third_step2D}}_2+\ell_P$ halfspaces, where $\ell_P\leq \binom{m}{k}$ is the number of halfspaces needed to construct $P$. Note that a vertex of $F$ arises from $n$ halfspaces, so that $|V(F)|\leq \binom{2n\ell^{\ref{prop_third_step2D}}_2+\ell_P}{n}\leq \ell^{\ref{prop_third_step}}_n(m)$. Hence, what remains is to strengthen property 3, to show that for $F\in\mathcal{G}^{2n}_2$, we have $\pi_1(V(F))\subset V(\pi_1(F))+(-4\zeta^2,4\zeta^2)$. By the above we can partition $V(F)$ into those vertices that are also a vertex of $F'$ and those that are not. 

First consider a vertex $v$ of $F$ that is also a vertex of $F'$. Note that as each of the defining hyperplanes of $F'$ contains all but one of $e_2,\dots,e_n$, and there are $n$ hyperplanes needed to define a vertex of $F'$, $v$ must be defined by at least two planes containing the same $n-2$ elements from $e_2,\dots,e_n$. Say  $v$ is defined by two planes containing $e_2,\dots,e_{i-1},e_{i+1}, e_n$, then we find that $\pi_{1,i}(v)\in V(\pi_{1,i}(F))$, so by property 3: $$\pi_1(v)=\pi_1(\pi_{1,i}(v))\in \pi_1(V(\pi_{1,i}(F)))\subset V(\pi_1(F))+(-4\zeta^2,4\zeta^2).$$

Hence, it remains to deal with vertices that lie on $\partial P$. we will show that if such a vertex lies far from $V(\pi_1(F))$, then $F$ is contained very close to the boundary of $P$.

Note that as we have that $|\pi_i(F)|\leq \eta$ for all $i=2,\dots n$, we find that up to translation, we have
$F\subset [0,|\pi_1(F)|]\times [0,\eta]^{n-1},$
so that if we let $C:=\mathbb{R}\times [0,\eta]^{n-1}$, then $F\subset C\cap P$.

\begin{clm}
If $\exists u,v,w\in C\cap \partial P$, so that $|\pi_1(u)-\pi_1(v)|,|\pi_1(v)-\pi_1(w)|,|\pi_1(w)-\pi_1(u)|>2\zeta^2$, then there exists a point $x$ in $C\cap \partial P$ and a line $L_x$ tangent to $P$ at $x$ so that $\angle L_x,e_1\leq \zeta$
\end{clm}
\begin{proof}
Note that by convexity of $P$, $C\cap \partial P$ can have at most two connected components, so we may assume $u,v$ are in the same connected component. Let $\pi_1^c:\mathbb{R}^n\to\mathbb{R}^{n-1}$ be the projection away from the first coordinate. Note that $|\pi_1^c(u)-\pi_1^c(v)|\leq \sqrt{n-1}\eta$ as $\pi_1^c(u),\pi_1^c(v)\in [0,\eta]^{n-1}.$ Hence, 
$$\frac12\angle (u-v),e_1\leq \sin(\angle (u-v),e_1)\leq \frac{\sqrt{n-1}\eta}{\zeta^2}\leq \frac12\zeta,$$ using that $\eta$ is sufficiently small in terms of $\zeta$.

Consider the plane $H$ containing $u,v$ and direction $e_1$. Note that again $H\cap C\cap \partial P$ has at most two connected components with $u$ and $v$ being in the same component. By the mean value theorem, there exists a point $x\in H\cap C\cap \partial P$ between $u$ and $v$ with a line $L_x$ tangent to $H\cap P$ (and thus to $P$) parallel to $u-v$. Clearly, $\angle L_x,e_1=\angle (u-v),e_1\leq \zeta$.
\end{proof}

Note that because $F \subset C \cap P$, we have the intervals $\pi_1(F) \subset \pi_1(C \cap P)$. Moreover, because $C= \mathbb{R} \times [0,\eta]^{n-1}$, there exist $u, w \in C\cap \partial P$ such that  $\pi_1(C \cap P) =[\pi_1(u), \pi_1(w)]$. Assume now $v$ is a vertex of $F$ that lies on $\partial P$ with the property that $\pi_1(v)\not\in  V(\pi_1(F))+(-2\zeta^2,2\zeta^2) $. It immediately follows that $\pi_1(v)\not\in  \{\pi_1(u), \pi_1(w)\}+(-2\zeta^2,2\zeta^2) $ i.e.,   $|\pi_1(u)-\pi_1(v)|,|\pi_1(v)-\pi_1(w)|,|\pi_1(w)-\pi_1(u)|>2\zeta^2$.
Hence, this claim shows that $F\subset L+B(0,\sqrt{n} \cdot\eta)$ for some line $L$ which makes an angle at most $\zeta$ with a tangent line at a point in $L\cap \partial P$. By \Cref{smallanglesmallset_lem}, this implies $F\subset \partial P+B(0,\xi)$.

With this setup, we are ready to conclude. Define $\mathcal{F}$ as follows:
\begin{align*}
\mathcal{F}_0&:=\mathcal{G}^{2n}_0\cup \left\{F\in\mathcal{G}^{2n}_2: \exists v\in V(F): \pi_1(v)\not\in  V(\pi_1(F))+(-2\zeta^2,2\zeta^2)\right\}\\
\mathcal{F}_1&:=\mathcal{G}^{2n}_1\\
\mathcal{F}_2&:=\left\{F\in\mathcal{G}^{2n}_2: \forall v\in V(F): \pi_1(v)\in  V(\pi_1(F))+(-2\zeta^2,2\zeta^2)\right\}.
\end{align*}
Indeed, 
$$\sum_{F\in \mathcal{F}_0}|F|= 2n\zeta+ |\partial P+B(0,\xi)|\leq \varepsilon,$$
where we used that $\xi$ is sufficiently small in terms of $\varepsilon$ and $P$.
\end{proof}

\subsection{Proof of Auxiliary Lemmas}

\subsubsection{Proof of \Cref{lem_side_note}}

\begin{proof}[Proof of \Cref{lem_side_note}]
Let $\lambda_n=n^{-n}$
Translate $P$ so that the John ellipsoid $E\subset P\subset nE$ is centred at the origin. Consider a translate $L$ of $L'$ that contains the origin. Let $H\supset L$ be any hyperplane containing $L$. Note that as $H$ also contains the centre of $E$, we find
$$\lambda_n=n^{-n}=\frac{\frac12 |E|}{\frac12|nE|}=\frac{|H^+\cap E|}{|H^-\cap nE|}\leq \frac{|H^+\cap P|}{|H^-\cap P|}\leq \frac{|H^+\cap nE|}{|H^-\cap E|}=\frac{\frac12 |nE|}{\frac12|E|}=n^n=\lambda_n^{-1},$$
which concludes the proof of the lemma.
\end{proof}

\subsubsection{Proof of \Cref{lem_few_vertices}}

\begin{proof}[Proof of \Cref{lem_few_vertices}]
Let $v_1,\dots, v_k$ where $k\geq 7$ be the vertices of $P$ appearing around $\partial P$ in that order. For every $i\in\{1,\dots k\}$, consider two vertices $v_i$ and $v_{i+3}$ where the indices are considered mod $k$. The line through $v_i$ and $v_{i+3}$ partitions the plane into two halfplanes. Let $H_i$ be the halfplane that contains all vertices except for $v_{i+1}$ and $v_{i+2}$.
\begin{clm}
$\bigcap_{i=1}^k H_i\neq \emptyset$
\end{clm}
\begin{proof}[Proof of Claim]
By Helly's theorem it suffices to show that any three of these halfplanes have non-empty intersection. Note that $|H_i^c\cap \{v_1,\dots,v_k\}|=2$, so for any $i,i',i''$ we find
$$|(H_i\cap H_{i'}\cap H_{i''})\cap \{v_1,\dots,v_k\}|=|\{v_1,\dots,v_k\}\setminus (H_i^c\cup H_{i'}^c\cup H_{i''}^c)|\geq k-6\geq 1.$$
Hence, the intersection of all $H_i$ is non-empty.\end{proof}

Choose some $p\in \bigcap_{i=1}^k H_i$ and consider any line $\ell$ containing $p$. Consider the two intersection points between $\ell$ and $\partial P$, say they lie on the line segment between $v_i$ and $v_{i+1}$ and on the line segment between $v_{j}$ and $v_{j+1}$. Assume for a contradiction one of the sets $\ell^+\cap P$ and $\ell^-\cap P$ has at least $k$ vertices. Then we find that $|i-j|\leq 2$ (again mod $k$), say $j\in\{i+1,i+2\}$. However, this implies that $p\in \ell\cap P \subset H_i^c$, which is clearly a contradiction. \end{proof}

\subsubsection{Proof of \Cref{ReducingVerticesLem}}
\begin{proof}[Proof of \Cref{ReducingVerticesLem}]
Iteratively construct a sequence of valid partitions $\mathcal{G}_i$ of $P$, starting with $\mathcal{G}_0=\{P\}$. Given $\mathcal{G}_i=\{P_1,\dots,P_j\}$, use \Cref{lem_few_vertices} to find points $p_i\in P_i$ (in those $P_i$ with at least 7 vertices) with the property that every line through them partitions $P_i$ into polygons with fewer vertices than $P_i$. Now let $$\mathcal{G}_{i+1}=\left\{f^+(P_1,p_1),f^-(P_1,p_1),\dots,f^+(P_j,p_j),f^-(P_j,p_j)\right\}.$$ Note that if we let $v(\mathcal{G}_i)$ be the maximal number of vertices among all parts in $\mathcal{G}_i$, then $v(\mathcal{G}_{i+1})\leq \max\{v(\mathcal{G}_i)-1,6\}$. Also, if we let $n(\mathcal{G}_i)$ be the number of parts of $\mathcal{G}_i$, then $n(\mathcal{G}_i) \leq 2^i$. Hence, we find $v(\mathcal{G}_{C-6})\leq 6$ and $n(\mathcal{G}_i) \leq 2^{c-6}$. Clearly $\mathcal{G}_{C-6}$ is a valid partition with the desired properties.
\end{proof}

\subsubsection{Proof of \Cref{smallanglesmallset_lem}}
\begin{proof}[Proof of \Cref{smallanglesmallset_lem}]
Let $D(P):=\sup_{x,y\in P} |x-y|$, we will show that $Q_{\zeta,\eta}\subset \partial P+ B(0,\eta+\zeta D(P))$.
Indeed, consider a $\zeta$-permissible line $L$ with the corresponding point $x\in L\cap \partial P$ and tangent line $L_x$. For any point $y\in L\cap P$, we have $|x-y|\leq D(P)$. By the definition of the sin, we have
$$\min_{y'\in L_x} |y-y'|= |x-y|\sin(\angle L,L_x)\leq D(P)\zeta.$$
If we let $y'\in L_x$ the point realizing $\min_{y'\in L_x} |y-y'|$, then  the line segment between $y$ and $y'$ intersects $\partial P$, so that
$\min_{y'\in \partial P} |y-y'|\leq \min_{y'\in L_x} |y-y'|\leq \zeta D(P).$
Choosing $\zeta$ and $\eta$ sufficiently small so that $\eta+\zeta D(P)\leq \xi$, the lemma follows.
\end{proof}

\subsubsection{Proof of \Cref{radiusormiddleverticesLem}}

\begin{proof}[Proof of \Cref{radiusormiddleverticesLem}]
Let $\alpha=\lambda_2^4$, where $\lambda_2$ is the constant from \Cref{lem_side_note}.

Iteratively produce a valid partition of $P$, starting with $\mathcal{G}_0=\{P\}$. Given $\mathcal{G}_i=\{P_1,\dots,P_{2^i}\}$, by \Cref{lem_side_note}, find points $p_j\in P_j$ and construct
$$\mathcal{G}_{i+1}:=\{f^{+}(P_1,p_1),f^{-}(P_1,p_1),\dots,f^{+}(P_{2^i},p_{2^i}),f^{-}(P_{2^i},p_{2^i})\}.$$
Consider $\mathcal{G}_4$, which has $2^4=16$ elements. Note that by construction, each $P'\in\mathcal{G}_4$ has $|P'|\geq \lambda_2^4|P|=\alpha |P|$.

Consider the set of newly created vertices $\bigcup_{P'\in\mathcal{G}_4}V(P')\setminus V(P)$. If there is a $v\in \bigcup_{P'\in\mathcal{G}_4}V(P')\setminus V(P)$ so that $\pi_2(v)\not\in V(\pi_2(P))+(-\varepsilon^2,\varepsilon^2)$ (or a line $f(P_j,p_j)$ through a vertex in $V(P)$), then consider the line $f(P_j,p_j)$ that created this vertex. Note that 
$$\min\{|\pi_2(f^+(P_j,p_j))|,|\pi_2(f^-(P_j,p_j))|\}\leq |\pi_2(P)|-\varepsilon^2,$$ so that at least one element $P'$ of $\mathcal{G}_4$ has $|\pi_2(P')|\leq |\pi_2(P)|-\varepsilon^2$. 

Alternatively, we find that $$\pi_2\left(\bigcup_{P'\in\mathcal{G}_4}V(P')\setminus V(P)\right)\subset V(\pi_2(P))+(-\varepsilon^2,\varepsilon^2).$$ As there are at most 4 vertices of $P$ whose projection is not in $V(\pi_2(P))+(-\varepsilon^2,\varepsilon^2)$, each of which is in at most element of $\mathcal{G}_4$, and there are $16$ sets in $\mathcal{G}_4$, there is an set $P'\in \mathcal{G}_4$ with $V(P')\subset V(\pi_2(P))+(-\varepsilon^2,\varepsilon^2)$. Note that 
$$\left(V(\pi_2(P))+(-\varepsilon^2,\varepsilon^2)\right)\cap \pi_2(P')\subset V(\pi_2(P'))+(-\varepsilon^2,\varepsilon^2).$$
Finally, note that every element of $\mathcal{G}_4$ has at most $10$ vertices. The lemma follows.\end{proof}

\section{Intermediate results for Quadratic Theorem (\Cref{main_thm_6}): Part II}

\subsection{Setup}

\begin{defn}
Given a $(i, \ell, \varepsilon)$-good cone $C \in \mathcal{C}^n$ (see \Cref{defn_cones_basis}), we say a measurable subset $X \subset C$ is \emph{filled} if for all $x\in C'$ we have
$$ f(x) \times \{x\} \subset X \text{ or } (f(x) \times \{x\}) \cap X=\emptyset. $$
\end{defn}

\begin{defn}
A pair of sets $X,Y\subset \mathbb{R}^n$ is a \emph{$(\eta,\gamma)$-approximate sandwich} if there exists a convex set $P\subset\mathbb{R}^n$ containing the origin, so that $X,Y\subset (1+\eta)P$ and $|P\setminus X|+|P\setminus Y|\leq \gamma |P|$.
\end{defn}

\subsection{Theorem}

\begin{thm}\label{appliedconeconclusion}
There exists an $\ell=\ell_n^{\ref{appliedconeconclusion}}$ so that for every $\xi,\lambda, \eta,\gamma>0$ the following holds. Given $A,B\subset \mathbb{R}^n$ a simple $\lambda_n$-bounded $\eta_n$-sandwich with the property that for all cones 
$C_i\in\mathfrak{C},$ we have $|A\cap C_i|=|B\cap C_i|$. 
Then there exist measurable subsets $A' \subset A$ and $B' \subset B$ and there exists a family of cones $\mathcal{G}$ essentially partitioning $\mathbb{R}^n$ refining $\mathfrak{C}$ and a partition $\mathcal{G}=\mathcal{G}_0 \sqcup \mathcal{G}_1 \sqcup \mathcal{G}_2$ such that 
\begin{enumerate}
    \item $|A'|=|B'|\geq (1-\xi)|A|$
    \item Every cone $F\in \mathcal{G}_0$ has $|A'\cap F|=|B'\cap F|=0$.
    \item Every cone $F\in \mathcal{G}_1\sqcup \mathcal{G}_2$ has $|A'\cap F|=|B'\cap F|$.
    \item Every cone $F\in \mathcal{G}_1\sqcup \mathcal{G}_2$ has $A'\cap F$ and $B'\cap F$ are $(2\lambda,F)$-bounded $(\eta,\gamma)$-approximate sandwiches inside $F$.
    \item Every cone $F \in \mathcal{G}_1$ is $(1,\ell, \infty)$-good and  every cone $F \in \mathcal{G}_2$ is $(2,\ell, \infty)$-good.
    \item For every cone $F\in \mathcal{G}_1\sqcup \mathcal{G}_2$, $A'\cap F$ and $B'\cap F$ are filled in $F$.
    \item For every cone $F \in \mathcal{G}_1\cup\mathcal{G}_2$, if $F\subset C_i$, we have that $e_n^F$ is perpendicular to the face of $S$ contained in $C_i$.
\end{enumerate}

\end{thm}

\subsection{Propositions}
Again, before proving \Cref{appliedconeconclusion}, we collect a list of results that will be used in the proof.

\begin{prop}
\label{lem_first_step}
Let $v_0,\dots, v_n\subset\mathbb{R}^n$ be vectors not contained in a halfspace and let $A,B\subset \mathbb{R}^n$ be measurable sets with equal volume. 
Then there exists a vector $v \in \mathbb{R}^n$ such that for every 
cone $C\in\mathfrak{C}^{v_0,\dots, v_n}$ we have $$|A \cap C|=|(B+v)\cap C|.$$
Moreover, for every $\eta,\lambda>0$, there is a computable constant $\eta'^{\ref{lem_first_step}}>0$ such that the following holds. If $\{v_0, \dots, v_n\}= \{e_0, \dots, e_n\}$ (as in \Cref{defn_simplex}) and if $A,B\subset \mathbb{R}^n$ is a $\lambda$-bounded $\eta'^{\ref{lem_first_step}}$-sandwich, then $A,B+v$ is a $2\lambda$-bounded $\eta$-sandwich.
\end{prop}

\begin{prop}
\label{lem_second_step}
 Assume that $C\subset \mathbb{R}^n$ is a cone and assume that $A,B\subset \mathbb{R}^n$ are measurable sets with the property that
$$|A \cap C|= |B\cap C|.$$ 
Then given a codimension-two subspace $L$ (through the origin), there exists a hyperplane $H$ through $L$ which essentially partitions the cone 
$$C=C_1 \cup C_2 $$ with the property that  
$$|A \cap C_1|= |B\cap C_1| \quad\text{ and }\quad |A \cap C_2|= |B\cap C_2|$$ 
\end{prop}

\begin{prop}\label{prop_fourandahalfth_step}
There exists an $\ell_n^{\ref{prop_fourandahalfth_step}}$so that for every $\varepsilon,\lambda, \eta>0$ the following holds. Given $A,B\subset \mathbb{R}^n$ a simple $\lambda_n$-bounded $\eta_n$-sandwich with the property that for all cones 
$C_i\in\mathfrak{C},$ we have $|A\cap C_i|=|B\cap C_i|$. There exists a family of cones $\mathcal{F}$ essentially partitioning $\mathbb{R}^n$ refining $\mathfrak{C}$ and a partition $\mathcal{F}=\mathcal{F}_0 \sqcup \mathcal{F}_1 \sqcup \mathcal{F}_2$ such that 
\begin{enumerate}
    \item Every cone $F\in \mathcal{F}$ has $|A\cap F|=|B\cap F|$.
    \item For every cone $F\in\mathcal{F}$, the pair $A\cap F,B\cap F$ is a $(\lambda,F)$-bounded $\eta$-sandwich.
    \item $\sum_{F \in \mathcal{F}_0} \mu_n(F) \leq \varepsilon.$
    \item Every cone $F \in \mathcal{F}_1$ is $(1,\ell_n^{\ref{prop_fourandahalfth_step}}, \varepsilon)$-good.
    \item For every cone $F \in \mathcal{F}_2$ there exists a sub-cone $F'$ of $F$ with $\mu_n(F') \geq (1-\varepsilon)\mu_n(F)$ such that $F'$ is $(2,\ell_n^{\ref{prop_fourandahalfth_step}}, \varepsilon)$-good. 
    \item For every cone $F \in \mathcal{F}_1$ (or $F'$ for $F\in\mathcal{F}_2$), if $F\subset C_i$, we have that $e_n^F$ (or $e_n^{F'}$) is perpendicular to the face of $S$ contained in $C_i$.
\end{enumerate}
\end{prop}

\begin{prop}\label{lem_fifth_step}
For any $\ell,\xi, \lambda, \eta,\gamma>0$, and $A, B \subset \mathbb{R}^n$ simple sets with equal volume, there exist $\varepsilon^{\ref{lem_fifth_step}}>0$ such that the following holds. Let $\mathcal{F}$ be a family of cones essentially partitioning $\mathbb{R}^n$ refining $\mathfrak{C}$ and let $\mathcal{F}=\mathcal{F}_0 \sqcup \mathcal{F}_1 \sqcup \mathcal{F}_2$ be a partition  such that

\begin{enumerate}
    \item Every cone $F\in \mathcal{F}$ has $|A\cap F|=|B\cap F|$.
    \item For every cone $F\in\mathcal{F}$, the pair $A\cap F,B\cap F$ is a $(\lambda,F)$-bounded $\eta$-sandwich.
    \item $\sum_{F \in \mathcal{F}_0} \mu_n(F) \leq \varepsilon^{\ref{lem_fifth_step}}.$
    \item Every cone $F \in \mathcal{F}_1$ is $(1,\ell, \varepsilon^{\ref{lem_fifth_step}})$-good.
    \item For every cone $F \in \mathcal{F}_2$ there exists a sub-cone $F'$ of $F$ with $\mu(F') \geq (1-\varepsilon^{\ref{lem_fifth_step}})\mu_n(F)$ such that $F'$ is $(2,\ell, \varepsilon^{\ref{lem_fifth_step}})$-good. 
    \item For every cone $F \in \mathcal{F}_1$ (or $F'$ for $F\in\mathcal{F}_2$), if $F\subset C_i$, we have that $e_n^F$ (or $e_n^{F'}$) is perpendicular to the face of $S$ contained in $C_i$.
\end{enumerate} 
Then there exist measurable subsets $A' \subset A$ and $B' \subset B$ and there exists a family of cones $\mathcal{G}$ essentially partitioning $\mathbb{R}^n$ refining $\mathcal{F}$ and a partition $\mathcal{G}=\mathcal{G}_0 \sqcup \mathcal{G}_1 \sqcup \mathcal{G}_2$ such that 
\begin{enumerate}
    \item $|A'|=|B'|\geq (1-\xi)|A|$
    \item Every cone $F\in \mathcal{G}_0$ has $|A'\cap F|=|B'\cap F|=0$.
    \item Every cone $F\in \mathcal{G}_1\sqcup \mathcal{G}_2$ has $|A'\cap F|=|B'\cap F|$.
    \item Every cone $F\in \mathcal{G}_1\sqcup \mathcal{G}_2$ has $A'\cap F$ and $B'\cap F$ are $(2\lambda,F)$-bounded $(\eta,\gamma)$-approximate sandwiches inside $F$.
    \item Every cone $F \in \mathcal{G}_1$ is $(1,\ell, \infty)$-good and  every cone $F \in \mathcal{G}_2$ is $(2,\ell, \infty)$-good.
    \item For every cone $F\in \mathcal{G}_1\sqcup \mathcal{G}_2$, $A'\cap F$ and $B'\cap F$ are filled in $F$.
\end{enumerate}
\end{prop}

\subsection{Proof of Theorem}

 \subsubsection{Proof of \Cref{appliedconeconclusion}}

\begin{proof}[Proof of \Cref{appliedconeconclusion}]
Let $\ell_n^{\ref{appliedconeconclusion}}:=\ell_n^{\ref{prop_fourandahalfth_step}}$.
Choose $\epsilon$ sufficiently small to be able to apply \Cref{lem_fifth_step}. Apply \Cref{prop_fourandahalfth_step} with this $\epsilon$ and then apply \Cref{lem_fifth_step}.
\end{proof}

\subsection{Proofs of Propositions}

\subsubsection{Proof of \Cref{lem_first_step}}
\begin{proof}[Proof of \Cref{lem_first_step}]
We begin by proving the first conclusion. Applying an affine transformation if necessary, we may assume $S^{v_0,\dots, v_n}=S$  and $\mathfrak{C}^{v_0,\dots, v_n}=\mathfrak{C}=\{C_0, \dots, C_n\}$. Recall that $S$ is the simplex containing the origin in its interior with vertices $V(S)=\{e_0, \dots, e_n\}$ (as in \Cref{defn_simplex}). Denote by $F_i=\co(\{e_j \colon i \neq j\})$ the face of $S$ opposite $e_i$. 

First, assume for a contradiction that for all $v \in \mathbb{R}^n$, there exists $i \in [0,n]$ such that
$|(B-v)\cap C_i| \neq  |A\cap C_i|.$
Note that for any $v \in \mathbb{R}^n$, we have
$$\sum_{i \in [0,n]} |(B-v)\cap C_i| = |B|=|A| = \sum_{i \in [0,n]} |A\cap C_i|. $$ Hence, if we define the closed sets $$X_i:=\{v\in\mathbb{R}^n: |(B-v)\cap C_{f(v)}| \geq  |A\cap C_{f(v)}|\}$$ for $i=0,\dots,n$, then $\bigcup_iX_i=\mathbb{R}^n$.
Consider $rS$, a large blow-up of $S$, so that $A,B\subset rS$. Note that $rS$ is a simplex containing the origin with vertices $V(rS)=\{rv_0, \dots, rv_n\}$. Moreover, the face of $rS$ opposite to vertex $rv_i$ is $rF_i$. It is easy to check that if $v \in rF_i$, then $|(B-v)\cap C_i|=0$. By the definition of the $X_i$'s, this implies that $v\in\bigcup_{j\neq i} X_i$. By the Knaster-Kuratowski-Mazurkiewicz lemma \cite{knaster1929beweis}, we find a point $v\in \mathbb{R}^n$, such that for all $i \in [0,n]$, $|(B+v)\cap C_i| =  |A\cap C_i|.$

We now prove the second conclusion. Fix $v \in \mathbb{R}^n$ as above and assume, without loss of generality, $v \in C_i$ for some $i \in \{0, \dots, n\}$,
$ P \subset A,B \subset (1+\eta')P$,
and
$ S \subset A,B \subset \lambda S. $
It is easy to check that  $(X \cap C_i)+v \subset (X+v) \cap C_i.$ Also, there exist universal constants $\alpha_n,\beta_n>0$ such that
$|(S+v)\cap C_i| \geq (1+\min(\alpha_n, \beta_n||v||_2))|S\cap C_i|$, where we used that $\{v_0,\dots, v_n\}=\{e_0,\dots,e_n\}$.

From the second and third inclusions and the above inequality, we get that
\begin{equation*}
    \begin{split}
        |(B+v)\cap C_i| &\geq |((B\setminus S)+v) \cap C_i|+|(S+v)\cap C_i|\\
        &\geq |(B\setminus S) \cap C_i|+(1+\min(\alpha_n, \beta_n||v||_2))|S\cap C_i|\\
        &= |(B\setminus S) \cap C_i|+ |S\cap C_i| +\min(\alpha_n, \beta_n||v||_2)|S\cap C_i|\\
        &\geq |B\cap C_i|+ (n+1)^{-1}\min(\alpha_n, \beta_n||v||_2)|S|\\
        &\geq |B\cap C_i|+ \lambda^{-n}(n+1)^{-1}\min(\alpha_n, \beta_n||v||_2)|B|.
    \end{split}
\end{equation*}
From the first inclusion and the choice of $v$, we get that
\begin{equation*}
    \begin{split}
    |(1+\eta'P) \setminus P| &\geq |A \setminus B|
    \geq |(A\cap C_i) \setminus (B\cap C_i)|
    \geq |A\cap C_i | - |B\cap C_i|= |(B+v)\cap C_i | - |B\cap C_i|.
    \end{split}
\end{equation*}
Finally, from the first inclusion, we also get that
$((1+\eta')^n-1)|B|\geq |(1+\eta')P \setminus P|. $ Combining the last three inequalities, we get that
$(1+\eta')^n-1 \geq \lambda^{-n}(n+1)^{-1}\min(\alpha_n, \beta_n||v||_2). $
By  choosing $\eta'>0$ such that $((1+\eta')^n-1)\lambda^{n}(n+1)\alpha_n^{-1}<1$, we further get that
$ ||v||_2\leq ((1+\eta')^n-1)\lambda^{n}(n+1)\beta_n^{-1}.$

Let $Z$ is the unit volume ball centered at the origin, let $\tau_n$ be its radius and set $\zeta=((1+\eta')^n-1)\lambda^{n}(n+1)\beta_n^{-1}\tau_n^{-1}.$ The above inequality implies that
$v \in \zeta Z.$
As $S$ is a unit volume regular simplex and $Z$ is a unit volume ball, both centered at the origin, we find
$ v\in\zeta Z\subset n\zeta  S.$
Moreover, from the first two inclusions, it easily follows that $S \subset 2P $
which implies
$ v \in 2n \zeta P. $

From the first two inclusions and the last two inclusions, we get
$$ (1-n\zeta)S\subset B+v \subset (\lambda+n\zeta)S \quad\text{ and }\quad (1-2n\zeta)P\subset B+v \subset (1+\eta'+2n\zeta)P.$$
Provided $\zeta \leq 10^{-1} n^{-1}\min(1, \eta )$, and thus $\eta'$ small, the conclusion follows. Thus, we shall choose $\eta'$ such that $((1+\eta')^n-1)\lambda^{n}(n+1)^{1}\alpha_n^{-1}<1$ and $((1+\eta')^n-1)\lambda^{n}(n+1)^{1}\beta_n^{-1}\tau_n^{-1}< 10^{-1} n^{-1}\min(1, \eta ).$
\end{proof}

\subsubsection{Proof of \Cref{lem_second_step}}

\begin{proof}[Proof of \Cref{lem_second_step}]
For notational convenience, assume that $A,B\subset C$, so that $|A|=|B|$.
Consider any hyperplane $H_0$ containing $L$. For $\theta\in [0,2\pi]$, let $\rho_\theta:\mathbb{R}^n\to\mathbb{R}^n$ be the rotation of the space fixing $L$ by angle $\theta.$ Let $H_\theta:=\rho_\theta(H_0)$ and let $H_\theta^+$ and $H_\theta^-$ be the halfspaces generated by $H_\theta$. Now consider the function $f:[0,\pi]\to \mathbb{R}$ defined by 
$f(\theta):=|H_\theta^+\cap A|-|H_\theta^+\cap B|.$
Note that as $A$ and $B$ are bounded sets, $f$ is continuous. Note moreover that as $H_\pi^+=H_0^-$, we find:
\begin{align*}
f(\pi)&=|H_\pi^+\cap A|-|H_\pi^+\cap B|=|H_0^-\cap A|-|H_0^-\cap B|=(|A|-|H_0^+\cap A|)-(|B|-|H_0^+\cap B|)=-f(0).
\end{align*}
By continuity, this implies the existence of a $\theta_0\in[0,\pi]$ so that $f(\theta_0)=0$. Hence, taking $H=H_{\theta_0}$, and $C_1:=H^+\cap C$ and $C_2:=H^-\cap C$, we find
$$|A \cap C_1|= |B\cap C_1| \quad\text{ and }\quad |A \cap C_2|= |B\cap C_2|.$$ 
This concludes the proof of the lemma.
\end{proof}

\subsubsection{Proof of \Cref{prop_fourandahalfth_step}}
Let $\ell_n:=\ell^{\ref{prop_fourth_step}}_n(n)$, where $\ell^{\ref{prop_fourth_step}}_n(n)$ is the constant from \Cref{prop_fourth_step}.

For this proposition, we apply \Cref{prop_fourth_step} to the context of $A$ and $B$.
\begin{proof}[Proof of \Cref{prop_fourandahalfth_step}]
We will construct a respectful function $f:\mathcal{C}^n\times \mathcal{T}_2^n\to\mathcal{T}_1^n$ as follows. Given a cone $C\in \mathcal{C}_n$ and a codimension-two subspace $L\in\mathcal{T}_2^n$, distinguish two case; either $|C\cap A|=|C\cap B|$ or not. In the latter case, let $f(C,L)$ be any hyperplane containing $L$; this will not be a case that is going to affect us. If $|C\cap A|=|C\cap B|$, let $f(C,L)$ be the hyperplane given by  \Cref{lem_second_step}.

Apply \Cref{prop_fourth_step} to each of the $C_i\in\mathfrak{C}$ with parameter $\varepsilon/(n+1)$, to find a partitions $\mathcal{F}^{C_i}=\mathcal{F}^{C_i}_0\sqcup\mathcal{F}^{C_i}_1\sqcup\mathcal{F}^{C_i}_2$. Let $\mathcal{F}_j=\bigcup_{i=0}^n \mathcal{F}^{C_i}_j$ for $j=0,1,2$, which then satisfies the conclusions 3, 4, 5, and 6, with the note that each of the $C_i$'s is defined by $n$ lines. It remains to check the first two conclusions.

For conclusion 1, note that by \Cref{lem_second_step}, a valid partition of $C_i$ contains only cones $F$ so that $|A\cap F|=|B\cap F|$, so the same holds in particular for all $F\in\mathcal{F}$.

For conclusion 2, note that being a $\lambda$-bounded $\eta$ sandwich is inherited by taking subcones, so in particular for all $F\in\mathcal{F}$, we have that $A\cap F,B\cap F$ is a $(\lambda,F)$-bounded $\eta_n$ sandwich.
\end{proof}

\subsubsection{Proof of \Cref{lem_fifth_step}}

\begin{proof}[Proof of \Cref{lem_fifth_step}]
We first construct $\mathcal{G}'=\mathcal{G}'_0\sqcup\mathcal{G}'_1\sqcup\mathcal{G}'_2$. First set $\mathcal{G}'_1:=\mathcal{F}_1$. For all $F\in\mathcal{F}_2$, let $F'\subset F$ be the subcone with $\mu_n(F')\geq (1-\varepsilon)\mu_n(F)$, which is $(2,\ell,\varepsilon)$-good. Then let $\mathcal{G}'_2:=\{F':F\in\mathcal{F}_2\}$. For every $F\in\mathcal{F}_2$, we can partition $F\setminus F'$ into a number of convex cones. Let $\mathcal{F}(F)$ be that collection of convex cones. Let $\mathcal{G}'_0:=\mathcal{F}_0\cup \bigcup_{F\in\mathcal{F}_2}\mathcal{F}(F)$.

Clearly, $\mathcal{G}'$ is a refinement of $\mathcal{F}$. Note that
$\sum_{G\in\mathcal{G}'_0}\mu_n(G)=\sum_{F\in\mathcal{F}_0}\mu_n(F)+\sum_{F\in\mathcal{F}_2}\mu_n(F\setminus F')\leq 2\varepsilon.$

We construct $A\supset A_1\supset A'$ and $B\supset B_1\supset B'$ with increasingly more structure as follows. 

Let $A_1\subset A$ and $B_1\subset B$ maximal so that $A_1\cap F$ and $B_1\cap F$ are filled for all $F\in\mathcal{G}'_1\cup\mathcal{G}'_2$. To this end, for a given $F\in\mathcal{G}'_i$ with $i=1,2$, consider the cone $C\in\mathcal{C}^i$ and linear function $f:C\to\mathcal{K}^{n-i}$ so that $F=\bigcup_{x\in C} \{x\}\times f(x)$. Let $g,h:C\to\mathcal{P}(\mathbb{R}^{n-i})$ so that $A\cap F=\bigcup_{x\in C} \{x\}\times g(x)$ and $B\cap F=\bigcup_{x\in C} \{x\}\times h(x)$. Define $g',h':C\to\mathcal{K}^{n-i}$ by 
$$g'(x):=\begin{cases}f(x) &\text{if }g(x)=f(x)\\
\emptyset &\text{otherwise}\end{cases},$$
and $h'$ analogously. Let $$A_1\cap F:=\bigcup_{x\in C} \{x\}\times g'(x)\quad\text{ and }\quad B_1\cap F:=\bigcup_{x\in C} \{x\}\times h'(x).$$ For $F\in\mathcal{G}'_0$, remove everything in the interior of $F$ from $A$ and $B$, so that $A_1\cap F^{\circ}:=\emptyset$ and $B_1\cap F^{\circ}:=\emptyset$. 

By $\lambda$-boundedness, we get 
$$\sum_{F\in\mathcal{G}_0}|F\cap (A\setminus A_1)|\leq \sum_{F\in\mathcal{G}'_0}O_n(|A|)\mu_n(F)\leq O_n(\varepsilon|A|),$$
and analogously for $B$.
To control $|A\setminus A_1|$ in $\mathcal{G}'_1\cup\mathcal{G}'_2$, we use the fact that the cones $F\in\mathcal{G}'_1\cup\mathcal{G}'_2$ are $(i,\ell,\varepsilon)$-good and that $A$ and $B$ are simple. Find $Q_A\subset\mathbb{Z}^n$ so that
$A=\bigcup_{x\in Q_A}x+[0,1]^n$. Note that if $x\in F\cap (A\setminus A_1)$ for some cone $F\in\mathcal{G}'_1\cup\mathcal{G}'_2$, then consider the basis $e_1^F,\dots,e_n^F$ and write $x=x'\times x''$ so that $x\in x'\times g(x')$, with $g:C\to\mathcal{K}^{n-i}$ as before. As $x\not\in A_1$, we find that $g(x')\neq \emptyset$ and $g(x')\neq f(x')$, so that $g(x')$ contains an element of $\partial A$. As $F$ was $(i, \ell,\varepsilon)$-good, we find that the radius of $f(x')$ is at most $\varepsilon ||x'||_2\leq  O_n(\varepsilon|A|^{1/n}) $. This implies that all points $x$ in $A\setminus A_1$ must be close to the boundary of $A$, i.e.,
$F\cap (A\setminus A_1)\subset \partial A + B(o,O_n(\varepsilon|A|^{1/n})).$ Hence, $$\sum_{F\in\mathcal{G}_1\cup\mathcal{G}_2}|F\cap(A\setminus A_1)|\leq |\partial A + B(o, O_n(\varepsilon|A|^{1/n}))|= O_n(\varepsilon|A|^{1/n}\cdot |Q_A|)=O_n(\varepsilon |A|^{(n+1)/n}),$$
so that we find
$$|A\setminus A_1|=\sum_{F\in\mathcal{G}_0}|F\cap(A\setminus A_1)|+\sum_{F\in\mathcal{G}_1\cup\mathcal{G}_2}|F\cap(A\setminus A_1)|=O_n(\varepsilon |A|^{(n+1)/n})\leq \varepsilon',$$
for any $\varepsilon'$ by choosing $\varepsilon$ sufficiently small in terms of $\varepsilon'$, $|A|$ and $n$. Let $B_1\subset B$ be constructed analogously. Note that the sets $A_1$ and $B_1$ satisfy conclusions 1, 2, 5, and 6 in the lemma.

We continue to construct $A_2\subset A_1$ and $B_2\subset B_1$. We aim to remove part of $A_1\cap F'$ or $B_1\cap F'$ (whichever is larger), outside of $\frac12 rS$, where $r$ is such that $rS\subset A,B\subset \lambda rS$, while maintaining the property that the sets remain filled in the cones. Indeed note that as $$A\setminus A_1\cap F'\subset \partial A+ B(o,O_n(\varepsilon |A|^{1/n}))\quad \text{ and }\quad rS\subset A,$$ we find that choosing $\varepsilon$ sufficiently small in terms of $|A|,r$, and $n$, we have $\frac23 r (S\cap F')\subset A_1,B_1$. Hence, (assuming that $A_1\cap F'$ is larger than $B_1\cap F'$), there exists a filled subset $A_2\cap F'\subset A_1\cap F'$ so that if we let $B_2\cap F'=B_1\cap F'$, then $|A_2\cap F'|=|B_2\cap F'|$ and $\frac12 r (S\cap F')\subset A_2\cap F', B_2\cap F'$. Let $A_2$ and $B_2$ be defined thus in all $F'\in\mathcal{G}'_1\cup \mathcal{G}_2'$.

To estimate $|A_2\setminus A_1|+|B_1\setminus B_2|$, for every $F'\in\mathcal{G}_1\cup\mathcal{G}_2$, compare $|A_1\cap F'|$ and $|B_1\cap F'|$. Find the $F\in\mathcal{F}_1\cup\mathcal{F}_2$ so that $F'\subset F$. Note that $|A_1\cap F'|,|B_1\cap F'|\leq |F\cap B|=|F\cap A|$, so that we find
\begin{align*}\sum_{F'\in\mathcal{G}'_1\cup\mathcal{G}'_2}\big||A_1\cap F'|-|B_1\cap F'|\big|&\leq \sum_{F'\in\mathcal{G}'_1\cup\mathcal{G}'_2}|F\cap A|-|F'\cap A_1|+|F\cap B|-|F'\cap B_1|.
\end{align*}
We consider the contribution from $A$ and $B$ independently to find
\begin{align*}
\sum_{F'\in\mathcal{G}'_1\cup\mathcal{G}'_2}|F\cap A|-|F'\cap A_1|&\leq \sum_{F'\in\mathcal{G}'_1\cup\mathcal{G}'_2}|(F\setminus F')\cap A|+|F'\cap (A\setminus A_1)|\leq |A\setminus A_1|\leq \varepsilon'.
\end{align*}
Hence, we find $$|A_2\setminus A_1|+|B_1\setminus B_2|\leq \sum_{F'\in\mathcal{G}'_1\cup\mathcal{G}'_2}\big||A_1\cap F'|-|B_1\cap F'|\big|=2\varepsilon'.$$

Note that $A_2$ and $B_2$ satisfy conclusions 1, 2, 3, 5, and 6 from the lemma. To get conclusion 4 as well, we construct $A'\subset A_2$ and $B'\subset B_2$ as follows. Let $\mathcal{H}\subset \mathcal{G}'_1\cup\mathcal{G}'_2$ be the collection of those $F'$ so that $A_2\cap F'$ and $B_2\cap F'$ is not a $(\eta,\gamma)$-approximate sandwich in $F'$. If we let $F\in\mathcal{F}$ so that $F'\subset F$, then we find that $A\cap F$ and $B\cap F$ is a $(\eta,F)$ sandwich, so there exists a $P$ so that $P\subset A\cap F,B\cap F\subset (1+\eta)P$. This immediately implies that $A_2\cap F',B_2\cap F'\subset (1+\eta)P\cap F',$ so as $A_2\cap F',B_2\cap F'$ is not $(\eta,\gamma)$-approximate sandwich, we must find $|(P\cap F')\setminus A_2|+|(P\cap F')\setminus B_2|\geq \gamma|P\cap F'|.$
Notice that as $P\subset A,B$, we find that
\begin{align*}
|(A\cap F')\setminus A_2|+|(B\cap F')\setminus B_2|\geq \gamma|P\cap F'|&\geq \frac12(1-\eta)^n\gamma (|\co(F'\cap A_2)|+|\co(F'\cap B_2)|)\\ &\geq \Omega_n(|F'\cap A_2|+|F'\cap B_2|),
\end{align*}
so that, summing over all $F'$, we find
$$\sum_{F'\in\mathcal{H}}|F'\cap A_2|+|F'\cap B_2|=O_n\left(\sum_{F'\in\mathcal{H}}|(A\cap F')\setminus A_2|+|(B\cap F')\setminus B_2|\right)\leq O_n(|A\setminus A_2|+|B\setminus B_2|)\leq O_n(\varepsilon').$$
Hence, the result follows by choosing $\mathcal{G}_0=\mathcal{G}'_0\cup \mathcal{H}$, $\mathcal{G}_1=\mathcal{G}_1'\setminus \mathcal{H}$, $\mathcal{G}_2=\mathcal{G}_2'\setminus \mathcal{H}$, and $A'\subset A_2$ and $B'\subset B_2$ are obtained by removing everything in the interior of the cones in $\mathcal{H}$, i.e.,
$A'=A_2\setminus \bigcup_{F\in\mathcal{H}}F^{\circ}$ and $B'=B_2\setminus \bigcup_{F\in\mathcal{H}}F^{\circ}$.
\end{proof}

\section{Intermediate results for the Quadratic Theorem (\Cref{main_thm_6}): Part III}

\subsection{Setup}

\begin{defn}
Let $e_1, \dots, e_n$ be an orthogonal basis.  Let $\pi_{i_1, \dots, i_j} \colon \mathbb{R}^n \rightarrow \mathbb{R}^j$ be the projection onto the $j$-dimensional subspace spanned by $e_1, \dots, e_{i_j}$. For a set $Y$ and $x_{i_1}, \dots, x_{i_j} \in \mathbb{R}$, let $Y_{x_{i_1}, \dots, x_{i_j}}$ be the fibre of $Y$ above $(x_{i_1}, \dots, x_{i_j})$.
\end{defn}

\begin{defn}\label{segment}
For a set $Y$ let $Y'=Y \cap \pi_1^{-1}[1/4,1/2]$.
\end{defn}

\begin{defn}
Given a $(n-1)$-dimensional simplex $X\subset\mathbb{R}^{n}$, a \emph{cylinder over $X$} is a set of the form $X+\mathbb{R}^+v$ for some direction $v\in S^{n-1}$.
\end{defn}

\begin{defn}\label{centralcylinder}
Given $\rho>0$ and a subcone $C\subset\mathbb{R}^n$ of a cone in $\mathfrak{C}$, we say a cylinder $U$ over a simplex $T\subset C\cap\frac12\partial S$ is \emph{$\rho$-central} if the $n+1$ defining hyperplanes $H_1,\dots, H_{n+1}$ (i.e., so that $U=\bigcap_{i=1}^{n+1} H_i^{+})$ of the cylinder have the property that  for all choices $*_i\in\{+,-\}$, we have $|C\cap S\cap \bigcap_{i=1}^{n+1} H_i^{*_i})|\geq \rho|C\cap S|$.
\end{defn}

\begin{defn}
Say a set $X\subset \mathbb{R}^n$ is $\alpha$-almost convex if $|\co(X)\setminus X|\leq \alpha |X|$.
\end{defn}

\subsection{Propositions}

 \begin{prop}\label{symdiftubes}
For any $w>1$, there exists $c_{n,w} >0 $ such that the following holds. Let $T=\{o, e_1, \dots, e_{n-1}\}$ be an orthogonal simplex. Let $U=T\times \mathbb{R}$. Let $ T\times [0,1] \subset A,B \subset T\times [0, w]$. Assume that $A$ and $B$ are filled in the sense that $A_{x_1,x_n}=U_{x_1,x_n}$ or $A_{x_1,x_n}=\emptyset$ and similarly  $B_{x_1,x_n}=U_{x_1,x_n}$ or $B_{x_1,x_n}=\emptyset$. If $|tA+(1-t)B| \leq (1+\delta)|A|$ with $0<\delta\ll_{n,t}1$, then $|A' \triangle B'| \leq c_{n, w}t^{-1/2}\delta^{1/2}|A|$. Here $A'$ and $B'$ are as in \Cref{segment}. 
 \end{prop}

 \begin{prop}\label{matchingcylinderslem}
Let $C\subset\mathbb{R}^n$ be a subcone of a cone in $\mathfrak{C}$, let $U$ be a $\rho$-central cylinder over a simplex $T\subset C\cap\frac12\partial S$, and let $\alpha$-almost convex sets $X,Y\subset C$ be so that
\begin{itemize}
    \item $|X|=|Y|$,
    \item $S\cap C\subset X,Y\subset \lambda S$,
    \item $|tX+(1-t)Y|\leq (1+\delta) |X|$.
\end{itemize}
Then there exists a cylinder $V=(1+\beta)U+x$ so that 
    \begin{itemize}
        \item $|\beta|\leq O_{n,\rho}\left(\delta^{1/2}t^{-1/2}+\alpha\right)$,
        \item $|(U\triangle V)\cap S|=O_{n,\rho}\left(\delta^{1/2}t^{-1/2}+\alpha\right)|C\cap S|$,
        \item $|X\cap U|=|Y\cap V|$, 
        \item $|t(X\cap U)+(1-t)(Y\cap V)|-|X\cap U|\leq |tX+(1-t)Y|-|X|$.
    \end{itemize}
\end{prop}

\begin{prop}\label{cylindercoveringlem}
For all $n\in\mathbb{N}$, $\lambda\geq 1$, and $\epsilon>0$, there exists $\rho=\rho_{n,\lambda,\epsilon}^{\ref{cylindercoveringlem}}>0$ so that the following holds.
Let $C\subset\mathbb{R}^n$ be a subcone of a cone in $\mathfrak{C}$. Let $f_1,\dots,f_n$ be an orthonormal basis of $\mathbb{R}^n$ so that $f_n$ is perpendicular to the face of $S$ intersecting $C$.

Then there exists a $(n-1)$-simplex $T\subset C\cap \frac12\partial S$ of radius less than $\epsilon$ and a collection $\mathcal{U}$ of $\rho$-central cylinders over $T$ so that 
\begin{itemize}
    \item $|\mathcal{U}|=O_{n,\lambda,\epsilon}(1)$,
    \item $\lambda (C\cap S)\setminus (C\cap S)\subset \bigcup_{U\in\mathcal{U}}U$,
    \item one facet of $T$ is parallel to the subspace spanned by $f_1,\dots,f_{n-2}$.
\end{itemize}
\end{prop}

\subsection{Proof of Propositions}

\subsubsection{Proof of Quadratic Theorem: \Cref{symdiftubes}}
\begin{proof}[Proof of \Cref{symdiftubes}]
Let $C$, $D$ be the sets $A$, $B$, respectively, Steiner symmetrized around $e_1$, i.e., $C_{x_1}$, $D_{x_1}$ are discs with the same size a $A_{x_1}$, $B_{x_1}$, respectively. 

\begin{clm}\label{symmetrized2}
   $ |A' \triangle B'| \leq |C' \triangle D'|+ O_{n,t}(\delta)|A|$ 
\end{clm}
\begin{proof}
By \Cref{LinearThmGeneral}, we know that $|\co(A)\setminus A|, |\co(B)\setminus B| \leq O_{n,t}(\delta)|A|$. Therefore, it is enough to show that
$$|A' \triangle B'| \leq |C' \triangle D'|+|\co(A)\setminus A|+ |\co(B)\setminus B|.$$
For a fixed fiber, we have the elementary inequality
\begin{multline*}|A_{x_1, \dots, x_{n-1}}\triangle B_{x_1, \dots, x_{n-1}}| \leq \big||A_{x_1, \dots, x_{n-1}}|-|B_{x_1, \dots, x_{n-1}}|\big|\\
+ |\co(A)_{x_1, \dots, x_{n-1}}|- |A_{x_1, \dots, x_{n-1}}|+ |\co(B)_{x_1, \dots, x_{n-1}}|- |B_{x_1, \dots, x_{n-1}}|.
\end{multline*}
By the hypothesis that $A$ and $B$ are filled, for fixed $x_1$ and varying $x_2, \dots, x_{n-1}$, we have $|A_{x_1, \dots, x_{n-1}}|$ is constant and $|B_{x_1, \dots, x_{n-1}}|$ is constant. Hence, for fixed $x_1$, we get that
\begin{equation*}
    \begin{split}
        \int \big||A_{x_1, \dots, x_{n-1}}|-|B_{x_1, \dots, x_{n-1}}|\big| dx_2 \dots dx_{n-1} = \bigg|\int |A_{x_1, \dots, x_{n-1}}|-|B_{x_1, \dots, x_{n-1}}|dx_2\dots dx_{n-1}\bigg|= \big||A_{x_1}|-|B_{x_1}|\big|, 
    \end{split}
\end{equation*}
therefore,
\begin{equation*}
    \begin{split}
        \int |A_{x_1, \dots, x_{n-1}}\triangle B_{x_1, \dots, x_{n-1}}| dx_2 \dots dx_{n-1} &\leq  \int \big||A_{x_1, \dots, x_{n-1}}|-|B_{x_1, \dots, x_{n-1}}|\big| dx_2 \dots dx_{n-1}\\
        &+ \int|\co(A)_{x_1, \dots, x_{n-1}}|- |A_{x_1, \dots, x_{n-1}}|dx_2 \dots dx_{n-1}\\
        &+\int|\co(B)_{x_1, \dots, x_{n-1}}|- |B_{x_1, \dots, x_{n-1}}|dx_2 \dots dx_{n-1}\\
        &=\big||A_{x_1}|-|B_{x_1}|\big|+ |\co(A)_{x_1}|-|A_{x_1}|+|\co(B)_{x_1}|-|B_{x_1}|. 
    \end{split}
\end{equation*}
Thus, we deduce
\begin{equation*}
    \begin{split}
        |A'\triangle B'|&=\int_{T'} |A_{x_1, \dots, x_{n-1}}\triangle B_{x_1, \dots, x_{n-1}}| d{x_1}dx_2 \dots dx_{n-1} \\
        &=\int_{1/4}^{1/2}\Big(\big||A_{x_1}|-|B_{x_1}|\big|+ |\co(A)_{x_1}|-|A_{x_1}|+|\co(B)_{x_1}|-|B_{x_1}|\Big)dx_1\\
        &=\int_{1/4}^{1/2}\Big(\big||C_{x_1}|-|D_{x_1}|\big|+ |\co(A)_{x_1}|-|A_{x_1}|+|\co(B)_{x_1}|-|B_{x_1}|\Big)dx_1\\
        &\leq |C' \triangle D'|+ |\co(A) \setminus A|+|\co(B) \setminus B|.\\
    \end{split}
\end{equation*}
This concludes the proof of the claim.

\end{proof}

Recall  that the sets $C$ and $D$ are Steiner symmetrized around $e_1$. In particular, as $|tA+(1-t)B|\leq (1+\delta)|A|$, we deduce $|tC+(1-t)D| \leq (1+\delta)|C|$.

Finally, because $T\times[0,1] \subset A,B\subset T\times [0,w]$, there exists a parameter $p$ such that $p(1-{x_1})^n\leq |C_{x_1}|,|D_{x_1}| \leq wp(1-x_1)^n$ for $x_1 \in [0,1]$ . 


 For a point $x \in [0,1]$ there exists a unique point $y_x \in [0,1]$ such that if we denote by $I=[0,x]$ and $J=[0,y]$, then $|C_I|=|D_J|$.

 \begin{clm}\label{symmetrized3}
The function $y_x\colon [0,2/3]\rightarrow [0,1]$ is continuous and Lipschitz with parameter $\Theta_{n,w}(1)$. 
 \end{clm}
 \begin{proof}
This follows immediately from the fact that, for an interval $I \subset [0,2/3]$, we have $|C_I|, |D_I|=\Theta_{n,w}(p)|I|$. 
 \end{proof}

\begin{clm}\label{symmetrized4}
For $x \in [1/4,1/2]$, $y=y_x$ satisfies $|y-x| =O_{d,w}(t^{-1/2}\delta^{1/2})$.
\end{clm}

\begin{proof}
 Let $I=[0,x]$ and $J=[0,y]$. As $tC_I+(1-t)D_J$ and $tC_{I^c}+(1-t)D_{J^c}$ are disjoint, we get  $|tC_I+(1-t)D_J|\leq |C_I|+\delta |C|$. Moreover, as $|A|=O_w(|C_I|)$, we further get $|tC_I+(1-t)D_J|\leq (1+O_{n,w}(\delta))|C_I|$.

 First note the qualitative bound that $y \leq 2/3$ (provided $\delta$ is small). Recall that for every $0 \leq z \leq 2/3$ we have $|C_z|,|D_z|=\Theta_{n,w}(p)$, and that the sets $C_I$ and $D_J$ are symmetrized. It is easy to see that, by doing $d+1$ parallel hyperplane cuts (containing the direction $e_1$), we can construct convex cylinders $C_I'$ and $D_J'$  inside $C_I$ and $D_J$, respectively, such that $$|C_I'|=|D_J'|=\Theta_{n,w} (1)|C_I|\quad \text{ and }\quad |tC_I'+(1-t)D_J'|-|C_I'| \leq |tC_I+(1-t)D_J|-|C_I|.$$
In particular, we get that the convex cylinders $C_I'$ and $D_J'$ have  heights $|I|$ and $|J|$ and $|tC_I'+(1-t)D_J'| \leq (1+\Theta_{n,w}(\delta))|C_I'|$. By the sharp stability of the Brunn-Minkowski for convex sets, we then deduce $|I|=(1+O(t^{-1/2}\delta^{1/2}))|J|$. As $x \in (1/4,1/2)$ We conclude $|y-x| =O_{n,w}(t^{-1/2}\delta^{1/2})$.
\end{proof}

\begin{clm}\label{symmetrized5}
    For points $x \in (1/4,1/2)$, if $|y-x| =O_{n,w}(t^{-1/2}\delta^{1/2})$, then $|\co(C)_x|-|\co(C)_y|\leq O_{n,w}(t^{-1/2}\delta^{1/2})$ and $|\co(D)_x|-|\co(D)_y|\leq O_{n,w}(t^{-1/2}\delta^{1/2})p$.
\end{clm}
\begin{proof}
    It is enough to notice that for $x \in (1/4,2/3) $, we have $|\co(C)_x|$ is Lipschitz with parameter $\Theta_{n,w}(p)$, which follows from the fact that $|\co(C)_x|^{1/n}$, $|\co(D)_x|^{1/n}$ are convex and $p(1-{x_1})^n\leq |C_{x}|,|D_{x}| \leq wp(1-x_1)^n$ for $x_1 \in [0,1]$.
\end{proof}

Consider the partitions into very small  consecutive intervals $[0,1]=I_0\cup\dots \cup I_{\ell}$ and $[0,1]=J_0\cup\dots \cup J_{\ell}$  such that $C_{I_k}$ and $D_{J_k}$ are all cylinders with the same volume. Moreover, set $|tC_{I_k}+(1-t)D_{J_k}|=(1+\delta_k)|C_{I_k}|$. It is trivial to check that $\sum_k \delta_k|C_{I_k}|\leq \delta |C|$. In particular, $\sum_{k: I_k \subset [1/4,1/2]} \delta_k\leq O_{n,w}(\delta) $. 

\begin{clm}\label{symmetrized6}
Fix $k$ such that $I_k, J_k \subset [0,2/3]$. For $x \in I_k$ and $y \in J_k$ we have $|C_x|=(1+O_{n,w}(t^{-1/2}\delta_k^{1/2}))|D_y|$.   
\end{clm}
\begin{proof}
Because $C_{I_k}$ and $D_{J_k}$ are convex cylinders, the result follows immediately from the sharp stability of Brunn-Minkowski inequality for convex sets \cite{Figalli09}. 
\end{proof}

It is easy to see that the function $ x \rightarrow y_x$ is picewise linear and maps the interval $I_k$ linearly into the interval $J_k$. Fix $x\in I_k \subset [1/4,1/2]$. We have the elementary bound
$$\big||C_x|-|D_{x}|\big|\leq \big||C_x|-|D_{y_x}|\big|+|\co(D)_x\setminus D_x|+|\co(D)_{y_x}\setminus D_{y_x}|+ \big||\co(D)_x|-|\co(D)_{y_x}|\big|.$$
From \Cref{symmetrized6}, we deduce $$\big||C_x|-|D_{y_x}|\big|=O_{n,w}(t^{-1/2}\delta_k^{1/2})p.$$ 
From \Cref{symmetrized4} and \Cref{symmetrized5}, we get 
$$\big||\co(D)_x|-|\co(D)_{y_x}|\big| \leq O_{n,w}(t^{-1/2}\delta^{1/2})p.$$
Therefore
\begin{equation*}
    \begin{split}
        |C'\triangle D'|= \int_{1/4}^{1/2} \big||C_x|-|D_x|\big|dx&= \sum_{k: I_k\subset [1/4,1/2]}O_{n,w}(t^{-1/2}\delta_k^{1/2})p\times |I_k| + O_{n,w}(t^{-1/2}\delta^{1/2})p\times 1/4\\ &+\int_{1/4}^{1/2}\Big(|\co(D)_x\setminus D_x|+|\co(D)_{y_x}\setminus D_{y_x}|\Big)dx\\
        &=\sum_{k: I_k\subset [1/4,1/2]}O_{n,w}(t^{-1/2}\delta_k^{1/2})|C_{I_k}|+ O_{n,w}(t^{-1/2}\delta^{1/2})|C_{[1/4,1/2]}|\\
        &+\int_{1/4}^{1/2}\Big(|\co(D)_x\setminus D_x|+|\co(D)_{y_x}\setminus D_{y_x}|\Big)dx\\
        &\leq O_{n,w}(t^{-1/2}\delta^{1/2})|C|+\int_{1/4}^{1/2}\Big(|\co(D)_x\setminus D_x|+|\co(D)_{y_x}\setminus D_{y_x}|\Big)dx.
    \end{split}
\end{equation*}
By the linear stability result we have $|\co(D)\setminus D| \leq O_{n,w}(\delta)|C|$, and by \Cref{symmetrized3} we deduce
$$\int_{1/4}^{1/2}\Big(|\co(D)_x\setminus D_x|+|\co(D)_{y_x}\setminus D_{y_x}|\Big)dx \leq O_{n,w}(\delta)|C|.$$
Therefore, we get
$|C'\triangle D'| \leq O_{n,w}(t^{-1/2}\delta^{1/2})|C|, $
that together with \Cref{symmetrized2} allows us to conclude that 
$|A'\triangle B'|\leq O_{n,w}(t^{-1/2}\delta^{1/2})|A|.$
\end{proof}

\subsubsection{Proof of \Cref{matchingcylinderslem}}
\begin{proof}[Proof of \Cref{matchingcylinderslem}]
First affinely transform so that the simplex formed by the vertices of $T$ together with the origin form a regular simplex.

Write $H_0$ for the hyperplane containing the face of $\frac12 S$ intersecting $C$. Writing $H_0^+$ for the halfspace defined by $H_0$ not containing the origin, we find that by $S\cap C\subset X,Y\subset \lambda S$, we have $|H_0^+\cap X|=|H_0^+\cap Y|$. Let $X':=H_0^+\cap X$ and $Y':=H_0^+\cap Y$, so that $|X'|=|Y'|$. By the Brunn-Minkowski inequality, we have $$|t(H_0^-\cap X)+(1-t)(H_0^-\cap Y)|\geq |H_0^-\cap X|.$$ Since $t(H_0^-\cap X)+(1-t)(H_0^-\cap Y)$ is a subset of $tX+(1-t)Y$ disjoint from $tX'+(1-t)Y'$, we find that $|tX'+(1-t)Y'|-|X'|\leq |tX+(1-t)Y|-|X|$.

We repeatedly apply this last argument for different hyperplane cuts.

Let $H_1,\dots, H_n$ be the hyperplane cuts defining $U$. Find parallel hyperplanes $G_1,\dots,G_n$ so that $|X'\cap \bigcap_{i=1}^j H_i^+|=|Y'\cap \bigcap_{i=1}^j G_i^+|$ for all $j=0,\dots,n$, which exist by continuity. We will show that the distance between $H_i$ and $G_i$ is at most $O_n\left(\delta^{1/2}t^{-1/2}\right)$. We show this for $i=1$, the other cases follow analogously by induction.

Consider the sets $X^\pm:= X'\cap H_1^{\pm}$ and $Y^\pm:= Y'\cap G_1^{\pm}$. By the fact that $U$ is central, we know that $|X^-\cap S|\geq \rho|C\cap S|$. Write $H'$ for the hyperplane $2H_0$ containing the face of $S$ intersecting $C$. Let $$X'':=X^-\cap (H')^-=S\cap H_0^+\cap H_1^-\cap (H')^-$$ and find a plane $G'$ parallel to $H'$ so that $Y'':=Y^-\cap (G')^-$ satisfies $|X''|=|Y''|$. By the argument before, we find that 
$$|tX''+(1-t)Y''|-|X''|\leq |tX'+(1-t)Y'|-|X'|\leq |tX+(1-t)Y|-|X|\leq \delta |X|\leq \rho^{-1}\delta |X''|.$$
Note that $X''$ is convex, so by the stability result for the Brunn-Minkowski inequality in \cite{Figalli09} we find that $|X''\triangle (Y''+x)|\leq O_n\left(\delta^{1/2}t^{-1/2}\right)|X''|$ for some translate $x\in \mathbb{R}^n$.

Consider the component $x^\perp$ of $x$ perpendicular to $H_0$. Note that the band between $H_0$ and $H_0+x^\perp$ of $S\cap C\cap H_1^{-}\cap G_1^{-1}$ is completely contained in $X''\triangle Y''$, so as $U$ is central, this implies $||x^\perp||_2=O_n\left(\delta^{1/2}t^{-1/2}\right)$. As $X''$ and $Y''$ are $\rho^{-1}\alpha$-almost convex, we find that translating them by $x^\perp$ changes the symmetric difference only little, i.e.,
$$|X''\triangle (Y''+(x-x^\perp))|\leq |X''\triangle (Y''+x)|+O_n\left(\delta^{1/2}t^{-1/2}+\alpha\right)|X''|\leq O_n\left(\delta^{1/2}t^{-1/2}+\alpha\right)|X''|.$$
It is now easy to see that the translate $z$ with no component perpendicular to $H_0$ minimizing $|X''\triangle (Y''+z)|$ is $z=0$. Indeed, for $z=0$, in every plane $P$ parallel to $H_0$ we have that $P\cap X''$ and $P\cap Y''$ are nested (which one is contained in which depends on the relative positions of $H_1$ and $H'$). Hence, we find
$$|X''\triangle Y''|\leq O_n\left(\delta^{1/2}t^{-1/2}+\alpha\right)|X''|.$$
Finally, akin to the previous argument, we now find that $(G_1\triangle H_1)\cap S\cap C\cap H_0^+\cap H'^-\cap G'^-\subset X''\triangle Y''$, so that $|(G_1\triangle H_1)\cap S|\leq O_n\left(\delta^{1/2}t^{-1/2}+\alpha\right)|C\cap S|$. This concludes the induction.

Letting $V:=\bigcap_{i=0}^{n} H_i^+$, we find that $U\triangle V\subset \bigcup_{i=1}^n G_i^+\triangle H_i^+$, therefore $$|(U\triangle V)\cap S|\leq O_n\left(\delta^{1/2}t^{-1/2}+\alpha\right)|C\cap S|,$$ and consequently $|\beta|=O_n\left(\delta^{1/2}t^{-1/2}+\alpha\right)$.
\end{proof}

\subsubsection{Proof of \Cref{cylindercoveringlem}}
\begin{proof}[Proof of \Cref{cylindercoveringlem}]
We first show that we may assume $C\cap \frac12\partial S$ is a simplex.

Consider the set $P:=C\cap \frac12\partial S$ which lies in a subspace spanned by $f_1,\dots, f_{n-1}$. Let $L$ be the subspace spanned by $f_1,\dots, f_{n-2}$. Let $x,x'\in \partial P$ so that $x+L$ and $x'+L$ are the two translates of $L$ tangent to $P$. Write $o':=\frac{x+x'}{2}$. Let $P'$ be the projection of $P$ onto $o'+L$ along the direction $xx'$. Note that $P\cap (o'+L)$ contains $\frac{o'+P'}{2}$. Let $T'$ be the largest simplex contained in that translate of $\frac{o'+P'}{2}$, so that $|T'|=\Omega_n(|P'|)$ and $P'$ is contained in some translate of $2n T'$. Let $T''=\co(T'\cup\{x\})$.

This construction gives a set $T''$ so that there exists a translate of $4nT''$ which contains $P$. Indeed, for every point $y$ on the line segment between $o'$ and $\frac{o'+x}{2}$ we have that $L+y\cap T''$ is a homothetic copy of $T'$ larger than $\frac12 T'$. Hence, for all points $y$ on the line segment between $4no'$ and $2n(o'+x)$ we have that $L+y\cap 4nT''$ is a homothetic copy of $T'$ larger than $2nT'$ (which contains a translate of $P'$). Hence, we can translate $4nT''$ so that $P\subset z+4nT''$ and $|T''|\geq \Omega_n(|P|)$.

Let $T$ be the translate of $\min\{\epsilon,\frac{1}{4n}\}T''$ centred at the barycenter of $T''$, so that we still have $|T|\geq \Omega_{n,\epsilon}(|P|)$ (for notational convenience assume $\min\{\epsilon,\frac{1}{4n}\}=\frac{1}{4n}$). Note that one of the facets of $T$ is parallel to $T'$, i.e., to the subspace spanned by $f_1,\dots,f_{n-2}$. Since the $T$ is contained in the proper interior of $T''\subset z+4nT''$ (at a distance lower bounded in terms of $n$ from $\partial(z+4nT'')$), we find that there exists some $\sigma_n$ so that $z+4nT''$ is contained in the translate of $\sigma_n T$ centred at the barycenter of $T$. Write $p$ for the barycenter of $T$, so that $H_{p,4n}(T)=T''\subset P\subset H_{p,\sigma_n}(T)$, where $H_{q,\xi}$ is the homothety with ratio $\xi$ centred at $q$.

Write $S'$ for the halfspace containing $S$ defined by the hyperplane which contains the face of $S$ intersecting $C$. Note that $C\cap S=C\cap S'$. Write $C':=\bigcup_{s>0} sH_{p,\sigma_n}(T)$ for the cone generated by the set $H_{p,\sigma_n}(T)\supset P$ and write $C'':=\bigcup_{s>0} sH_{p,4n}(T)$ for the cone generated by the set $H_{p,4n}(T)=T''\subset P$, so that $C''\subset C\subset C'$.

We have that 
$$ \lambda (C\cap S)\setminus (C\cap S)\subset  \lambda (C'\cap S')\setminus (C'\cap S').$$
Up to an affine transformation, we may assume $\co(T\cup\{o\})$ is a regular simplex. Note that the problem of covering $\lambda (C'\cap S')\setminus (C'\cap S')$ with cylinders over $T$ which are central in the cone $C''$ depends only on $n$, $\lambda$, and $\epsilon$ and no longer on the particular $C$, and $f_1,\dots,f_n$ that we started with. Hence, we can simply choose any collection $\mathcal{U}$ of cylinders (which are $\xi$-central for some $\xi>0$) so that $ \lambda (C'\cap S')\setminus (C'\cap S')\subset \bigcup_{U\in\mathcal{U}}U$ and choose $\rho'$ appropriately so that the cylinders are $\rho'$-central in $C''$. We automatically have $|\mathcal{U}|=O_{n,\lambda,\epsilon}(1)$ and if a cylinder is $\rho'$-central in $C''$ it is also $\rho$ central in $C$ for some $\rho$ depending only on $\lambda,n,\epsilon,$ and $\rho'$. 
\end{proof}

\section{Proof of the Quadratic Symmetric Difference result (\Cref{main_thm_6})}

\begin{proof}[Proof of \Cref{main_thm_6}]
Choose parameters according to the following hierarchies. First choose 
$$n\gg \lambda^{-1}\gg \rho\gg \epsilon\gg \lambda'^{-1}\gg w^{-1},$$
so that $\epsilon,\rho,\lambda,\lambda',w=\theta_n(1)$. Then choose
$$t,n,\epsilon,\rho,\lambda,\lambda',w\gg \Delta,\eta''\gg \eta\gg \gamma\gg \eta',$$
where every variable is understood to be chosen sufficiently small in terms of all of the preceding variables.

By \Cref{boundedsandwichreduction}, we may assume that $A,B$ form a simple $\lambda/2$-bounded $\eta'$-sandwich. By \Cref{lem_first_step} we may assume (after translating) that, for $C_i\in\mathfrak{C}$, $|C_i\cap A|=|C_i\cap B|$ at the cost of $A,B$ form a simple $\lambda$-bounded $\eta$-sandwich.

Now apply \Cref{appliedconeconclusion} with parameters $\xi=\delta,\lambda, \eta, \gamma>0$, to find $A'\subset A$, $B'\subset B$, and a family of cones $\mathcal{G}$ essentially partitioning $\mathbb{R}^n$ refining $\mathfrak{C}$ and a partition $\mathcal{G}=\mathcal{G}_0 \sqcup \mathcal{G}_1 \sqcup \mathcal{G}_2$ such that:
\begin{enumerate}
    \item $|A'|=|B'|\geq (1-\xi)|A|$;
    \item For every cone $F\in \mathcal{G}_0$ it holds $|A'\cap F|=|B'\cap F|=0$;
    \item For every cone $F\in \mathcal{G}_1\sqcup \mathcal{G}_2$ it holds $|A'\cap F|=|B'\cap F|$;
    \item For every cone $F\in \mathcal{G}_1\sqcup \mathcal{G}_2$, $A'\cap F$ and $B'\cap F$ are $(2\lambda,F)$-bounded $(\eta,\gamma)$-approximate sandwiches inside $F$;
    \item Every cone $F \in \mathcal{G}_1$ is $(1,\ell_n^{\ref{appliedconeconclusion}}, \infty)$-good, and  every cone $F \in \mathcal{G}_2$ is $(2,\ell_n^{\ref{appliedconeconclusion}}, \infty)$-good;
    \item For every cone $F\in \mathcal{G}_1\sqcup \mathcal{G}_2$, $A'\cap F$ and $B'\cap F$ are filled in $F$;
    \item For every cone $F \in \mathcal{G}_1\cup\mathcal{G}_2$, if $F\subset C_i$ then $e_n^F$ is perpendicular to the face of $S$ contained in $C_i$.
\end{enumerate}
By our choice of $\xi$, it suffices to show $|A'\triangle B'|\leq O_n( t^{-1/2}\delta^{1/2})|A|$, as $|A\triangle B|\leq |A'\triangle B'|+|A\setminus A'|+|B\setminus B'|\leq |A'\triangle B'|+2\delta |A|$.

For every cone $F\in \mathcal{G}_1\sqcup \mathcal{G}_2$, write $\delta_F:=\frac{|t(A'\cap F)+(1-t)(B'\cap F)|-|A'\cap F|}{|A'\cap F|}$. For notational convenience, write $\delta_F=0$ for $F\in \mathcal{G}_0$. In terms of this parameter, we find
$$\sum_{F\in\mathcal{G}}\delta_F |A'\cap F|=\sum_{F\in\mathcal{G}}|t(A'\cap F)+(1-t)(B'\cap F)|-|A'\cap F|\leq |tA'+(1-t)B'|-|A'|\leq|tA+(1-t)B|-|A'|\leq 2\delta |A|,$$
where we used that $t(A'\cap F)+(1-t)(B'\cap F)$ are essentially disjoint subsets of $tA'+(1-t)B'$. We also have the trivial observation that $\sum_{F\in\mathcal{G}}|(A'\cap F)\triangle(B'\cap F)|=\sum_{F\in\mathcal{G}}|(A'\triangle B')\cap F|=|A'\triangle B'|$. By concavity of the square root function, it thus suffices to show $|(A'\cap F)\triangle(B'\cap F)|\leq O_n(t^{-1/2}\delta^{1/2})|A'\cap F|$ for every cone $F\in\mathcal{G}$. Henceforth, fix a cone $C\in\mathcal{G}_2$ for which we will show this bound on the symmetric difference (the case that $C\in\mathcal{G}_0\sqcup\mathcal{G}_1$ follows analogously, though more easily). For notational convenience write $X:=A'\cap C$ and $Y:=B'\cap C$.

As $X,Y$ form an $(\eta,\gamma)$-approximate sandwich inside $C$, we find that  for sufficiently small $\eta$ and $\gamma$, we have that $\delta_C$ is smaller than $\Delta$ for all $C\in\mathcal{G}$. Hence, by \Cref{LinearThmGeneral}, we know that $X$ and $Y$ are $O_{n,t}(\delta)$-almost convex.

Now apply \Cref{cylindercoveringlem} to the cone $C$ and the basis $e_1^C,\dots,e_n^C$ to find a $(n-1)$-simplex $T\subset C\cap \frac12\partial S$ and a collection $\mathcal{U}$ of $\rho$-central cylinders over $T$ of radius less than $\epsilon$, so that 
\begin{itemize}
    \item $|\mathcal{U}|=O_{n,\lambda}(1)=O_n(1)$,
    \item $\lambda (C\cap S)\setminus (C\cap S)\subset \bigcup_{U\in\mathcal{U}}U$, 
    \item one facet of $T$ is parallel to the subspace spanned by $e_1^C,\dots,e_{n-2}^C$.
\end{itemize}
Since $A'\cap C$ and $B'\cap C$ are $\lambda$-bounded inside the cone $C$, we have 
$$X\triangle Y\subset \lambda (C\cap S)\setminus (C\cap S)\subset \bigcup_{U\in\mathcal{U}}U,$$
so that 
$$|X\triangle Y|\leq\sum_{U\in\mathcal{U}}|U\cap (X\triangle Y)|\leq |\mathcal{U}|\max_{U\in\mathcal{U}}|U\cap (X\triangle Y)|\leq O_n\left( \max_{U\in\mathcal{U}}|U\cap (X\triangle Y)|\right).$$
Hence, it suffices to show $|U\cap (X\triangle Y)|\leq O_n(t^{-1/2}\delta^{1/2})|X|$.
Because one facet of $T$ is parallel to the subspace spanned by $e_1^C,\dots,e_{n-2}^C$, we find that $X\cap U$ filled.

Now apply \Cref{matchingcylinderslem} to $C,U$, and the $O_{n,t}(\delta)$-almost convex pair $X,Y$, to find a cylinder $V=(1+\beta)U+x$ so that:
    \begin{itemize}
        \item $|\beta|\leq O_{n,\rho}\left(\delta^{1/2}t^{-1/2}+O_{n,t}(\delta)\right)=O_n\left(\delta^{1/2}t^{-1/2}\right)$,
        \item $|(U\triangle V)\cap S|=O_{n,\rho}\left(\delta^{1/2}t^{-1/2}+O_{n,t}(\delta)\right)|C\cap S|=O_n\left(\delta^{1/2}t^{-1/2}\right)|C\cap S|$,
        \item $|X\cap U|=|Y\cap V|$, 
        \item $|t(X\cap U)+(1-t)(Y\cap V)|-|X\cap U|\leq |tX+(1-t)Y|-|X|=\delta_C|X|$.
    \end{itemize}
Note that as $V$ is homothetic to $U$, we also have that $V\cap Y$ is filled.

As $T$ lies well within $S$ and has radius less than $\epsilon$, we can apply \Cref{ChangeoftLemma} to find $t'\in (t/2,2t)$ and $X',Y'$ that form a tubular $\lambda'$-bounded $\eta''$-sandwich, where $\lambda'=\theta_n(\lambda)=\theta_n(1)$ and $\eta''$ tend to zero with $\eta$, so that 
$$|t'X'+(1-t')Y'|-t'|X'|-(1-t')|Y'|\leq |t(X\cap U)+(1-t)(Y\cap V)|\leq \delta_C|X|,$$
and $|(X\cap U)\triangle(Y\cap V)|\leq|X'\triangle Y'|+O_n\left(\delta_C^{1/2}t^{-1/2}\right)|X\cap U|$. Moreover, as we have only taken affine transformations, we find that $X',Y'$ are both filled in their common tube, say $W$. After affine transformation, we may assume $W=S'\times \mathbb{R}$ for an orthogonal simplex $S'$, so that $ S'\times [0,1] \subset X',Y' \subset S'\times [0, w/2]$ for some $w=O_{n,\lambda'}(1)=O_n(1)$.

Note that by construction $\big||X'|-|Y'|\big|\leq O_n\left(\delta_C^{1/2}t^{-1/2}\right)|X'|$. Extend (i.e., append a set of the form $S'\times [-\zeta,0)$ for an appropriate $\zeta\leq O_n\left(\delta_C^{1/2}t^{-1/2}\right)$ and then translating that set up by $\zeta$) the smaller of the two sets to find $X''\supseteq X'$ and $Y''\supseteq Y'$ with $|X''|=|Y''|$, so that by \Cref{extendingTubes}, we find
$$|t'X''+(1-t')Y''|-|X''|=|t'X'+(1-t')Y'|-t'|X'|-(1-t')|Y'|\leq \delta_C|X|.$$
Moreover, we have $|X'\triangle Y'|\leq|X''\triangle Y''|+O_n\left(\delta_C^{1/2}t^{-1/2}\right)|X''|$ and $$ S'\times [0,1] \subset X'',Y''\subset S'\times [0, w/2+\zeta] \subset S'\times [0, w].$$ Clearly, $X''$ and $Y''$ are still filled.

Recall that as $U$ was $\rho$-central, we find that $|X''|\geq 2^{-n}|X'|\geq\Omega_{n,\rho}(|X|)=\Omega_{n}(|X|)$. Hence, as $\delta_C\leq \Delta$, we find that $\delta_C\frac{|X|}{|X''|}$ is sufficiently small in terms of $n$ and $t$ so that we can apply \Cref{symdiftubes} to find that $|X''\triangle Y''|\leq O_{n,w}(t^{-1/2}\delta^{1/2})|X''|=O_n(t^{-1/2}\delta^{1/2})|X''|$.

\end{proof}

\section{Intermediate results for Linear Theorem with few vertices (\Cref{main_thm_2})}

For notational convenience, consider the following definitions.
Given $A,B\subset\mathbb{R}^n$ with $|A|=|B|$ and $t\in(0,1)$, let 
$$D_t(A,B):=tA+(1-t)B, \quad\text{ and }\quad\delta_t(A,B):=\frac{|D_t(A,B)|}{|A|}-1.$$
Throughout this section, we will always consider $t\in(0,1)$. We will show the following theorem.

\begin{thm}\label{LinearThm}
There exists a constants $c^{\ref{LinearThm}}$  such that the following holds.
Given $A,B\subset\mathbb{R}^n$ of equal volume and $t\in[\tau,1-\tau]$ such that $\co(A)$ is a simplex, $A$ is the intersection between a simple set and a simplex and $\delta_t(A,B) \leq \min\{t,(1-t)\}^n$, then
$$|\co(A)\setminus A|\leq \min\{t,(1-t)\}^{-c^{\ref{LinearThm}}n^8}\delta_t(A,B) |A|.$$
\end{thm}

\Cref{main_thm_2} follows easily from \Cref{LinearThm} which is the same result with the stronger assumption that $\co(A)$ is a simplex.

\subsection{Outline of the proof of \Cref{LinearThm}}

The proof of \Cref{LinearThm} follows the following steps.
\begin{enumerate}
    \item Writing  $S=\co(x_0,\dots,x_n)$, find a point $v\in A$ so that the density in each of the subsimplices $$\co(x_0,\dots, x_{i-1},v,x_{i+1},\dots, x_n)$$ does not decrease too much (cf \Cref{densitycontrolledpartition}).
    \item The doubling of $A$ in each of the subsimplices is subadditive (cf \Cref{linearityofdoubling}), i.e., there exists a matching partition $B=\bigcup_{S'} B_{S'}$ with $|B_{S'}|=|A\cap S'|$ so that
    $$|tA+(1-t)B|\geq \sum |t(A\cap S')+(1-t) B_{S'}|.$$
    \item Iterating this process, we end up with two types of simplices: those in which $A$ has low density and those with small radius.
    \item For a simplex in which $A$ has a sufficiently low density (but not too low), i.e., $|S'\cap A|/|S'|$ is small, a recent result by van Hintum and Keevash \cite{SharpDelta} shows that the doubling is $\Omega_{n,t}(1)$ (cf \Cref{SharpDelta}). Since we can guarantee the density of $A$ inside $S'$ is $O_{n,t}(1)$, we find that $|S'|$ is controlled by the doubling of $A\cap S'$.
    \item Assuming that without loss of generality that $A$ is a finite union of boxes intersected with a simplex (cf \Cref{boundedsandwichreduction}), the combined volume of simplices with small radius that are not completely filled by $A$ goes to zero with the number of iterations.
    \item Therefore, we conclude
    \begin{align*}
|\co(A)\setminus A|&=\sum_{S'} |S'\setminus A|\leq \sum_{\text{low density } S'} O_{n,t}(\left|t(A\cap S')+(1-t)B_{S'}\right|-|A\cap S'|)+ \sum_{\text{small radius } S'}|S'|\\
&=O_{n,t}(|tA+(1-t)B|-|A|).
    \end{align*}
\end{enumerate}

\subsection{Auxiliary propositions}

A crucial ingredient is the following proposition.

\begin{prop}\label{linearityofdoubling}
If $\co(A)$ is a simplex $S'$ with vertex set $\{x_0,\dots x_n\}$ and $x\in A$ is some point in the interior of $\co(A)$, then we partition $S'$ into simplices 
$$S_i:= \co(x_0,\dots, x_{i-1}, x, x_{i+1}, \dots, x_n).$$
For any $B\subset \mathbb{R}^n$ with $|B|=|A|$ there exist sets $B_i\subset \mathbb{R}^n$ with $|B_i|=|A\cap S_i|$, so that
$$\sum_i  \frac{|A\cap S_i|}{|A|}\ \delta_t(A\cap S_i, B_i)\leq \delta_t(A,B)$$
\end{prop}

We recall here the following result of van Hintum and Keevash \cite{SharpDelta}.

\begin{prop}[\cite{SharpDelta}]\label{SharpDelta}
There exists a constant $c^{\ref{SharpDelta}}$ so that if $t\in(0,\frac12]$ and $A,B\subset \mathbb{R}^n$ of equal volume satisfy $|\co(A)|+|\co(B)|\geq t^{-c^{\ref{SharpDelta}}n^2}|A|$, then $\delta_t(A,B)\geq t^n$.
\end{prop}

The last crucial ingredient is the following proposition.
\begin{prop}\label{densitycontrolledpartition}
For all $n \in \mathbb{N}$ and $\alpha>0$ there exists a constant $\eta^{\ref{densitycontrolledpartition}}_{\alpha, n}>0$ (we can take $\eta^{\ref{densitycontrolledpartition}}_{\alpha, n}= c^{\ref{densitycontrolledpartition}}_n\alpha^{c^{\ref{densitycontrolledpartition}}n^6}$ for some constants $c^{\ref{densitycontrolledpartition}}_n, c^{\ref{densitycontrolledpartition}}>0 $) so that the following holds. Let $X\subset \mathbb{R}^n$ be a set and $S'=\co\{x_0, \dots, x_n\} \subset \mathbb{R}^n$ be a simplex with $X \subset S'$. If $|X|\geq \alpha |S'|$, then there exists a point $x\in X$ such that for all $0\leq i\leq n$
$$|\co\{x_0,\dots,x_{i-1},x,x_{i+1},\dots,x_n\}\cap X| \geq \eta^{\ref{densitycontrolledpartition}}_{\alpha, n} |S'|.$$
Moreover, there exists a constant $\rho^{\ref{densitycontrolledpartition}}_{\alpha,n}>0$ (we can take $\rho^{\ref{densitycontrolledpartition}}_{\alpha,n}=\eta^{\ref{densitycontrolledpartition}}_{\alpha, n}$) such that  $$\max_i d(x,x_i)\leq (1-\rho^{\ref{densitycontrolledpartition}}_{\alpha,n}) \max_{i,j} d(x_i,x_j),$$
\end{prop}

A weaker version of \Cref{densitycontrolledpartition} in the dense domain follows later in \Cref{findpointcopy}, which has a simpler proof.

\subsection{Auxiliary lemmas}

\begin{lem}\label{partialdensitycontrolledpartition}
For all $n \in \mathbb{N}$ and $\alpha>0$ there exist constants $\eta^{\ref{partialdensitycontrolledpartition}}_{1,\alpha, n}, \eta^{\ref{partialdensitycontrolledpartition}}_{2,\alpha, n}>0$ (we can take $\eta^{\ref{partialdensitycontrolledpartition}}_{1,\alpha, n}=\eta^{\ref{partialdensitycontrolledpartition}}_{2,\alpha, n}=c_n^{\ref{partialdensitycontrolledpartition}}\alpha^{c^{\ref{partialdensitycontrolledpartition}} n^3}$, for some constants $c_n^{\ref{partialdensitycontrolledpartition}}, c^{\ref{partialdensitycontrolledpartition}}>0$)  so that the following holds.  Let $X\subset \mathbb{R}^n$ be a set and $S'=\co\{x_0, \dots, x_n\} \subset \mathbb{R}^n$ be a simplex  with $X \subset S'$. If $|X|\geq \alpha |S'|$, then there exists a subset  $Y \subset X$ with $|Y| \geq \eta^{\ref{partialdensitycontrolledpartition}}_{1,\alpha, n} |S'| $ such that for all $x \in Y$ and $1\leq i\leq n$
$$|\co\{x_0,\dots,x_{i-1},x,x_{i+1},\dots,x_n\}\cap X| = \eta^{\ref{partialdensitycontrolledpartition}}_{2,\alpha, n} |S'|.$$
\end{lem}

\begin{lem}\label{boxdensitycontrolledpartition}
For all $n \in \mathbb{N}$ and $\alpha>0$ there exist constants $\eta^{\ref{boxdensitycontrolledpartition}}_{1,\alpha, n}, \eta^{\ref{boxdensitycontrolledpartition}}_{2,\alpha, n}>0$ (we can take them $\eta^{\ref{boxdensitycontrolledpartition}}_{1,\alpha, n}= \eta^{\ref{boxdensitycontrolledpartition}}_{2,\alpha, n}= c_n^{\ref{boxdensitycontrolledpartition}} \alpha^{c^{\ref{boxdensitycontrolledpartition}n^2}}$, for some constants $c_n^{\ref{boxdensitycontrolledpartition}n^2},c^{\ref{boxdensitycontrolledpartition}}>0$ )  so that the following holds. Let $X\subset Q=[-1,1]^n$. If $|X|\geq \alpha |Q|$, then there exists a subset  $Y \subset X$ with $|Y| \geq \eta^{\ref{boxdensitycontrolledpartition}}_{1,\alpha, n} |Q| $ such that for all $x \in Y$ and all faces $F$ of the box $Q$
$$|\co(\{x\} \cup F)\cap X| = \eta^{\ref{boxdensitycontrolledpartition}}_{2,\alpha, n} |Q|.$$
\end{lem}
Before the next lemma we need a definition. 

\begin{defn}
For $\alpha>0$ and $1 \leq j \leq n$ construct the $(n-1)$-dimensional box 
$$Q_{j, \alpha}=\co\left(\left\{\sum_{i=1}^n  \varepsilon_ie_i \colon \varepsilon_i\in \{-\alpha,\alpha\} \text{ for all } i\in [n] \setminus \{j\} \text{ and } \varepsilon_j=1 \right\}\right)$$ and construct the cone 
$$C_{j, \alpha}=\cup_{t \geq 0} tQ_{j, \alpha}.$$ Moreover, for $-n \leq j \leq -1$ let $Q_{j, \alpha}=-Q_{-j, \alpha}$ and $C_{j, \alpha}=-C_{-j, \alpha}$.
\end{defn}

\begin{lem}\label{conesdensitycontrolledpartition}
For all $n \in \mathbb{N}$ and $\alpha>0$ there exist constants $\eta^{\ref{conesdensitycontrolledpartition}}_{1,\alpha, n}, \eta^{\ref{conesdensitycontrolledpartition}}_{2,\alpha, n}>0$ (we can take $\eta^{\ref{conesdensitycontrolledpartition}}_{1,\alpha, n}=\eta^{\ref{conesdensitycontrolledpartition}}_{2,\alpha, n}=c_n^{\ref{conesdensitycontrolledpartition}}\alpha^{c^{\ref{conesdensitycontrolledpartition}}n^2}$, for some constants $c_n^{\ref{conesdensitycontrolledpartition}}, c^{\ref{conesdensitycontrolledpartition}}>0$)  so that the following holds. Let $X\subset Q=[-1,1]^n$. If $|X|\geq \alpha |Q|$, then there exists a subset  $Y \subset X$ with $|Y| \geq \eta^{\ref{conesdensitycontrolledpartition}}_{1,\alpha, n} |Q| $ such that for all $x \in Y$ and all $j \in [-n,n] \setminus \{0\}$
$$|(x+C_{j, \alpha}) \cap X| \geq \eta^{\ref{conesdensitycontrolledpartition}}_{2,\alpha, n} |Q|.$$
\end{lem}

\begin{lem}\label{stepdensitycontrolledpartition}
For all $n \in \mathbb{N}$ and $\alpha>0$ the following holds. Let $P=[0,a_1]\times \dots \times [0,a_n]+x$ be a box with $x \in \mathbb{R}^n$ and $a_n \alpha^2/8 \leq a_1, \dots,  a_{n-1} \leq a_n \alpha^2/4 $. Let $X$ be a subset of the box with $|X| \geq \alpha |P|$. Then there exists a box $$Q=[0,b_1]\times \dots \times [0,b_n]+y\qquad \text{with $y \in \mathbb{R}^n$ and $b_1 (\alpha/4)^2/8 \leq b_2, \dots,  b_{n} \leq b_1 (\alpha/4)^2/4 $}$$ such that, if we set $Y=Q\cap X$, then the following holds: $$|Q| \geq \alpha^{2n}|P|/2^{7n},\quad |Y| \geq \alpha|Q|/4,\quad \text{ and }\quad 
|(y+C_{n,\alpha}) \cap X| \geq \alpha |P|/4\quad \text{for all $y \in Y$.}$$
\end{lem}

We shall need one more lemma, for the proof \Cref{main_thm_2}.

\begin{lem}\label{subsetdoubling}Let $t\in(0,1)$ and $A,B,C\subset\mathbb{R}^n$ so that $|A|=|B|$ and $C$ is convex with a finite number of vertices, then there exists a subset $B'\subset B$ so that $|B'|=|A\cap C|$ and
    $$|t(A\cap C)+(1-t) B'|-|A\cap C|\leq |tA+(1-t)B|-|A|.$$
\end{lem}

\subsection{Proof of \Cref{LinearThm}}

\begin{proof}[Proof of \Cref{LinearThm}] we will consider the case $t\leq \frac12$, the other case follows analogously. Let $\varepsilon=t^{c^{\ref{SharpDelta}}n^2}$. 
Let $c^{\ref{SharpDelta}}$ be the constant from \Cref{SharpDelta}. Let $c^{\ref{densitycontrolledpartition}},c_n^{\ref{densitycontrolledpartition}}$ and $\rho=\rho^{\ref{densitycontrolledpartition}}_{\varepsilon,n}$ the constants from \Cref{densitycontrolledpartition}. 

Consider the following iterative process. First set $\mathcal{T}_0=\{\co(A)\}$ and $\mathcal{S}_0=\emptyset$, and note that $|\co(A)|\leq \varepsilon^{-1}|A|$ by \Cref{SharpDelta}.
At a given stage $i$ with $\mathcal{T}_i,\mathcal{S}_i$, look at every element $S'\in \mathcal{T}_i$ and distinguish two cases: either $|S'\cap A|\leq \varepsilon |S'|$ or $|S'\cap A|> \varepsilon |S'|$.

For each simplex $S'=\co\{x_0,\dots,x_n\}\in\mathcal{T}_i$ with $|S'\cap A|\geq \varepsilon |S'|$ we construct the $n+1$ simplices $f_0(S'),\dots, f_n(S')$ as follows. We apply \Cref{densitycontrolledpartition} to find a central point $x\in S'\cap A$ and we construct the simplex $f_j(S')=\co\{x_0,\dots,x_{j-1},x,x_{j+1},\dots,x_n\}$.

Now let
$$\mathcal{T}_{i+1}:=\bigcup_{S'\in \mathcal{T}_i:\  |S'\cap A|\geq \varepsilon |S'|} \left\{f_0(S'),\dots,f_n(S')\right\}\qquad \text{and} \qquad \mathcal{S}_{i+1}:=\mathcal{S}_i \cup \left\{S'\in \mathcal{T}_i: |S'\cap A|< \varepsilon |S'|\right\}.$$

Using the fact that $A$ is closed, it follows  by induction that for $i \in \mathbb{N}$ and  $S' \in \mathcal{T}_i \sqcup \mathcal{S}_i$ we have $\co(A\cap S')=S'$. Moreover,  $\mathcal{T}_i \sqcup \mathcal{S}_i$ forms an essential partition of $\co(A)$.

\begin{clm}\label{clm_size_0}
    For all $S' \in \mathcal{T}_i$ and $j \in [0,n]$, we have $|f_j(S')| \geq c_n^{\ref{densitycontrolledpartition}}\varepsilon^{c^{\ref{densitycontrolledpartition}}n^6}|S'| $.
\end{clm}
\begin{proof}
By \Cref{densitycontrolledpartition} (and our choice of $x \in S' \cap A$), we have $|f_j(S')|\geq |f_j(S')\cap A|\geq c_n^{\ref{densitycontrolledpartition}}\varepsilon^{c^{\ref{densitycontrolledpartition}}n^6}|S'|$.
\end{proof}

\begin{clm}\label{clm_size1}
For all $S'\in \mathcal{S}_i\cup\mathcal{T}_i$, we have $|S'\cap A|\geq c_n^{\ref{densitycontrolledpartition}}\varepsilon^{c^{\ref{densitycontrolledpartition}}n^6}|S'|$.
\end{clm}
\begin{proof}
Every simplex $S'\in\mathcal{S}_i\cup\mathcal{T}_i$ is $f_j(S'')$ for some $S''$ with $|S''\cap A|\geq \varepsilon|S''|$. By \Cref{densitycontrolledpartition} (and our choice of $x \in S' \cap A$), we have $|S'\cap A|\geq c_n^{\ref{densitycontrolledpartition}}\varepsilon^{c^{\ref{densitycontrolledpartition}}n^6}|S''|$.
\end{proof}

\begin{clm}
At stage $i$ there exists a function $g_i:\mathcal{T}_i\cup\mathcal{S}_i\to \mathcal{P}(\mathbb{R}^n)$, so that  for all $S'\in \mathcal{T}_i\cup\mathcal{S}_i$, $g_i(S')$ is a measurable subset of $B$ such that $|g_i(S')|=|S'\cap A|$ and
$$ \sum_{S'\in \mathcal{T}_i\cup\mathcal{S}_i}|A\cap S'|\cdot \delta_t(A\cap S', g_i(S'))\leq |A|\cdot\delta_t(A,B).$$
\end{clm}
\begin{proof}
This follows by induction from \Cref{linearityofdoubling}.
\end{proof}

\begin{clm}
For every $i$, we have that
$$\sum_{S'\in \mathcal{S}_i} |S'|\leq \frac{1}{c_n^{\ref{densitycontrolledpartition}}} t^{-c^{\ref{densitycontrolledpartition}}c^{\ref{SharpDelta}}n^8-n}|A|\cdot \delta_t(A,B).$$
\end{clm}
\begin{proof}[Proof of Claim]
For $S'\in\mathcal{S}_i$, we know by construction and by the second claim that $c_n^{\ref{densitycontrolledpartition}}\varepsilon^{c^{\ref{densitycontrolledpartition}}n^6}|S'|\leq |S'\cap A|< \varepsilon |S'|$. Moreover, $\co(A\cap S')=S'$. Hence, by \Cref{SharpDelta}, 
$\delta_t(A\cap S', g_i(S'))\geq t^n,$ where $g_i$ is the function from the previous claim. Hence, combining the results from the previous two claims, we find
\begin{align*}
\sum_{S'\in \mathcal{S}_i} |S'|&\leq \frac{1}{c_n^{\ref{densitycontrolledpartition}}}\varepsilon^{-c^{\ref{densitycontrolledpartition}}n^6}\sum_{S'\in \mathcal{S}_i}|S'\cap A|\leq \frac{1}{c_n^{\ref{densitycontrolledpartition}}}\varepsilon^{-c^{\ref{densitycontrolledpartition}}n^6}\sum_{S'\in \mathcal{S}_i} |S'\cap A|\cdot\frac{\delta_t(A\cap S', g_i(S'))}{t^n}\\
&\leq \frac{1}{c_n^{\ref{densitycontrolledpartition}}}t^{-n}\varepsilon^{-c^{\ref{densitycontrolledpartition}}n^6} |A|\cdot \delta_t(A,B)=\frac{1}{c_n^{\ref{densitycontrolledpartition}}}t^{-c^{\ref{densitycontrolledpartition}}c^{\ref{SharpDelta}}n^8-n}|A|\cdot \delta_t(A,B),
\end{align*}
which concludes the claim.
\end{proof}

Before we conclude we need one more claim. Given a simplex $S'$, let radius $rad(S')$ be the maximal length among its edges. 

\begin{clm}\label{volumebigradius}
For all $r>0$ the following holds:
$$\sum_{S'\in \mathcal{T}_i: rad(T)>r} |S'| \to 0 \quad \text{ as } i \to \infty.$$
\end{clm}
\begin{proof}[Proof of Claim]
We make the convention $\bigcup \mathcal{T}_i:= \bigcup_{S' \in \mathcal{T}_i} S'$ Let $k=\lceil\log_{(1-\rho)}(r)\rceil$. We can distinguish two types of elements in $\mathcal{T}_i$. Either $S'\subset \bigcup \mathcal{T}_{i+k(n+1)} $ or not. Collect the former in $\mathcal{T}_i'$ and the latter in $\mathcal{T}_i''$.

For $S'\in\mathcal{T}''_i$, at least some simplex originating from $S'$ is in $\mathcal{S}_{i+k(n+1)}$. By the first claim, that simplex will have size at least $(c_n^{\ref{densitycontrolledpartition}}\varepsilon^{c^{\ref{densitycontrolledpartition}}n^6})^{k(n+1)}|S'|$. Hence,
$$\sum_{S'\in \mathcal{T}_{i+k(n+1)}:\ S'\subset\bigcup\mathcal{T}''_i } |S'|\leq \left(1-(c_n^{\ref{densitycontrolledpartition}}\varepsilon^{c^{\ref{densitycontrolledpartition}}n^6})^{k(n+1)}\right)\left|\bigcup\mathcal{T}''_i \right|.$$

For $S'\in\mathcal{T}_i'$, we will find an element $S''\in\mathcal{T}_{i+k(n+1)}$ with $S''\subset S'$ and $rad(S'')<r$. Let $S'=S^0$, and consider $S^{j+1}:=f_n(f_{n-1}(\dots f_0(S^j)\dots)$. Crucially, $rad(S^{j+1})\leq (1-\rho)rad(S^j)$. Indeed, none of the edges of $S^j$ remain and all of the edges added have length at most $(1-\rho)rad(S^j)$ by \Cref{densitycontrolledpartition}. Hence, $rad(S^k)\leq r\cdot rad(S^0)\leq r$. Note that again by the first claim $|S^k|\geq (c_n^{\ref{densitycontrolledpartition}}\varepsilon^{c^{\ref{densitycontrolledpartition}}n^6})^{k(n+1)}|S'|$.

Combining these two cases gives 
$$\sum_{S'\in \mathcal{T}_{i+k(n+1)}:\ rad(S')>r} |S'| \leq \left(1-(c_n^{\ref{densitycontrolledpartition}}\varepsilon^{c^{\ref{densitycontrolledpartition}}n^6})^{k(n+1)}\right)\sum_{S'\in \mathcal{T}_{i}:\ rad(S')>r} |S'|.$$
The conclusion follows.
\end{proof}

To bound the contribution from the simplices in $\mathcal{T}_i$, we use the assumption that $A$ is the intersection between a finite union of cubes and a simplex. In particular, this implies that $\partial A$ has finite $(n-1)$-dimensional Hausdorff measure. Hence, $|\partial A+ B^n(o,r)|\to 0$ as $r\to 0$. Fix $r$ to be such that $|\partial A+ B^n(o,r)|\leq \delta_t(A,B)^2 |A|$. Using \Cref{volumebigradius}, find an $i$, so that $\sum_{S'\in \mathcal{T}_i: rad(S')>r} |S'|\leq \delta_t(A,B)^2 |A|.$

Note that for $S'\in\mathcal{T}_i$, we have $S'\cap A\neq \emptyset$, so if $S'\not\subset A$, we find that $S'\cap \partial A\neq \emptyset$. We conclude that 
\begin{align*}
\sum_{S'\in\mathcal{T}_i}|S'\setminus A|&\leq\sum_{S'\in\mathcal{T}_i:\ S'\subset A}|S'\setminus A|+\sum_{S'\in\mathcal{T}_i:\ rad(S')>r}|S'|+\sum_{S'\in\mathcal{T}_i:\ S'\cap \partial A\neq \emptyset,\ rad(S')\leq r}|S'|\\
&\leq 0+\delta_t(A,B)^2 |A| +|\partial A+ B^n(o,r)|\leq 2 \delta_t(A,B)^2 |A|.
\end{align*}
Combining this with the bound on $\mathcal{S}_i$, we find
$$|\co(A)\setminus A|= \sum_{S'\in \mathcal{S}_i\cup \mathcal{T}_i} |S'\setminus A|\leq c_n^{\ref{densitycontrolledpartition}}t^{-c^{\ref{densitycontrolledpartition}}c^{\ref{SharpDelta}}n^8-n} \delta_t(A,B)|A|+2\delta_t(A,B)^2|A|\leq t^{-c^{\ref{LinearThm}}n^8}\delta_t(A,B)|A|,$$
and the result follows.\end{proof}

\subsection{Proof of Propositions}

\subsubsection{Sublinearity of doubling in subsimplices; \Cref{linearityofdoubling}}
\Cref{linearityofdoubling} hinges on a geometric lemma. 

Let $S_i$ as in the statement of \Cref{linearityofdoubling} and translate so that $x$ is the origin. Consider the cones generated by the $S_i$, i.e., $\mathfrak{C}^{x_0,\dots,x_n}$.

We will use \Cref{lem_first_step} repeatedly and \Cref{linearityofdoubling} will follow quickly.
\begin{proof}[Proof of \Cref{linearityofdoubling}]
Apply \Cref{lem_first_step}, find a $v$ so that $|A\cap S_i|=|A\cap C_i|=|(B-v)\cap C_i|$ for all $C_i\in\mathfrak{C}^{x_0,\dots,x_n}$. We will show that the union
$\bigcup_{i=0}^n D_t(A\cap S_i, B\cap (v+C_i))\subset D_t(A,B)$
is an essentially disjoint union. Indeed note that, as $C_i$ is a convex set, $$D_t(A\cap S_i, B\cap (v+C_i)) \subset D_t( C_i, v+C_i) = (1-t)v+C_i$$ and $\bigcup_i (1-t)v+C_i$ is an essentially disjoint union.
We can conclude by setting $B_i= B\cap (v+C_i)$ and expanding
\begin{align*}
   \sum_i  |A\cap S_i|\cdot \delta_t(A\cap S_i, B_i)&=\sum_i  |D_t(A\cap S_i,B\cap (v+C_i))|-|A\cap S_i|\leq  |D_t(A,B)|-|A|\leq |A|\cdot \delta_t(A,B).
\end{align*}
The proposition follows.
\end{proof}

\subsubsection{Proof of \Cref{densitycontrolledpartition}}

\begin{proof}[Proof of \Cref{densitycontrolledpartition}]
Set $\eta^{\ref{densitycontrolledpartition}}_{\alpha, n}=\min(\eta^{\ref{partialdensitycontrolledpartition}}_{2,\alpha, n}, \eta^{\ref{partialdensitycontrolledpartition}}_{2,\eta^{\ref{partialdensitycontrolledpartition}}_{1,\alpha, n}, n})$. We apply \Cref{partialdensitycontrolledpartition} to the set $X$ together with the special vertex $x_n$. Thus we construct a subset $Y \subset X$ with $|Y| \geq \eta^{\ref{partialdensitycontrolledpartition}}_{1,\alpha, n} |S'| $ such that for all $x \in Y$ and $0\leq i\leq n-1$, $|\co\{x_0,\dots,x_{i-1},x,x_{i+1},\dots,x_n\}\cap X| \geq \eta^{\ref{partialdensitycontrolledpartition}}_{2,\alpha, n} |S'|.$
Now we apply again \Cref{partialdensitycontrolledpartition} to the set $Y$ together with the special vertex $x_{n-1}$. Thus we construct a subset $Z \subset Y$ with $|Z| \geq \eta^{\ref{partialdensitycontrolledpartition}}_{1,\eta^{\ref{partialdensitycontrolledpartition}}_{1,\alpha, n}, n} |S'|$ such that for all $x \in Z$ and $0 \leq i \leq n-2$ or $i=n$,
$$|\co\{x_0,\dots,x_{i-1},x,x_{i+1},\dots,x_n\}\cap Y| \geq \eta^{\ref{partialdensitycontrolledpartition}}_{2,\eta^{\ref{partialdensitycontrolledpartition}}_{1,\alpha, n}, n} |S'|.$$
Fix $z \in Z  \subset Y \subset X$. By the above inequalities, we conclude that for
$0\leq i\leq n$,
$$|\co\{x_0,\dots,x_{i-1},x,x_{i+1},\dots,x_n\}\cap X| \geq \min\left(\eta^{\ref{partialdensitycontrolledpartition}}_{2,\alpha, n}, \eta^{\ref{partialdensitycontrolledpartition}}_{2,\eta^{\ref{partialdensitycontrolledpartition}}_{1,\alpha, n}, n}\right) |S'| \geq \eta^{\ref{densitycontrolledpartition}}_{\alpha, n}|S'|.$$
For the last part, fix $0 \leq k \leq n$. Let $q_k$ be the intersection of the ray $x_kx$ with the face opposite $x_k$. 
On the one hand, a simple computation gives
$\frac{|xq_k|}{|x_kq_k|}= \frac{|\co\{x_0,\dots,x_{i-1},x,x_{i+1},\dots,x_n\}|}{|\co\{x_0,\dots,x_n\}|}.$
From the first part, we know that
$\frac{|\co\{x_0,\dots,x_{i-1},x,x_{i+1},\dots,x_n\}|}{|\co\{x_0,\dots,x_n\}|} \geq \eta^{\ref{densitycontrolledpartition}}_{\alpha, n}.$
Combining the last two inequalities, we get
$\frac{|xq_k|}{|x_kq_k|} \geq \eta^{\ref{densitycontrolledpartition}}_{\alpha, n} \text{ i.e., } \frac{|xx_k|}{|x_kq_k|} \leq 1-\eta^{\ref{densitycontrolledpartition}}_{\alpha, n}.$
On the other hand, because the diameter of a simplex is realized between two vertices, we have
$|x_kq_k| \leq \max_{i,j}(x_ix_j).$
Combining the last two inequalities, we find
$d(x,x_i)\leq (1-\rho^{\ref{densitycontrolledpartition}}_{\alpha,n}) \max_{i,j} d(x_i,x_j),$
which concludes the proposition.
\end{proof}

\subsection{Proof of Auxiliary Lemmas}

\subsubsection{Proof of \Cref{partialdensitycontrolledpartition}}

\begin{proof}[Proof of \Cref{partialdensitycontrolledpartition}]
Set $\eta^{\ref{partialdensitycontrolledpartition}}_{1,\alpha, n}=\eta^{\ref{boxdensitycontrolledpartition}}_{1,\frac{\alpha^{n+1}}{(2n)^{n+1}}, n}  \frac{\alpha^{n+1}}{(2n)^{n+1}} $ and $\eta^{\ref{partialdensitycontrolledpartition}}_{2,\alpha, n}= \eta^{\ref{boxdensitycontrolledpartition}}_{2,\frac{\alpha^{n+1}}{(2n)^{n+1}}, n} \frac{\alpha^{n+1}}{(2n)^{n+1}}$.
After an affine transformation, without loss of generality, we can assume that $S'=\co(\{o, e_1, \dots, e_n\})$. 
For $(\lambda_1, \dots, \lambda_n)\in S'$ construct the box $Q_{\lambda_1, \dots, \lambda_n}=\{(p_1, \dots, p_n) \colon \lambda_1 \geq p_1 \geq 0, \dots, \lambda_n \geq p_n \geq 0\}.$

\begin{clm}\label{clm_dens_1}
There exist $\lambda_1 \geq 0, \dots, \lambda _n \geq 0$ with $\lambda_1 + \dots + \lambda_n=1$ such that $Q:= Q_{\lambda_1, \dots, \lambda_n}$ satisfies $|X\cap Q| \geq \frac{\alpha^{n+1}}{(2n)^{n+1}} |S'|$.
\end{clm}

\begin{proof}[Proof of \Cref{clm_dens_1}]
It is enough to prove that for $\lambda=(\lambda_1,\dots, \lambda_n)\in S'$ chosen uniformly at random, we have
$$\mathbb{E}_\lambda |X \cap Q| \geq \frac{\alpha^{n+1}}{(2n)^{n+1}} |S'|.$$
Note that for $1 \geq s \geq 0$ we have $(1-s)S' \subset S'$ and  $|(1-s)S'|=(1-s)^n|S'| \geq (1-ns)|S'|$. Therefore
\begin{align*}
    \begin{split}
        \left|X\cap \left(1-\frac{\alpha}{2n}\right) S'\right| &\geq |X\cap S'| - \left|S'\setminus \left(1-\frac{\alpha}{2n}\right) S'\right|\geq \alpha|S'|-|S'|+ \left|\left(1-\frac{\alpha}{2n}\right) S'\right|\\
        &\geq \alpha |S'| -|S'|+\left(1-\frac{\alpha}{2}\right)|S'|
        =\frac{\alpha}{2}|S'|.
    \end{split}
\end{align*}
Fix $x=(x_1, \dots, x_n) \in S'$ and choose a random $\lambda=(\lambda_1, \dots \lambda_n) \in F$. We have
\begin{align*}
    \mathbb{P}_\lambda(x\in Q_\lambda)&= \frac{|\{(\lambda_1, \dots \lambda_n) \colon \lambda_i \geq x_i \text{ for all $i$ and } \lambda_1+ \dots +\lambda_n \leq 1\}|}{|\{(\lambda_1, \dots \lambda_n) \colon \lambda_i \geq 0 \text{ for all $i$ and } \lambda_1+ \dots +\lambda_n \leq 1\}|}\\
    &=\frac{|\{(\lambda_1, \dots \lambda_n) \colon \lambda_i \geq 0 \text{ for all $i$ and } \lambda_1+ \dots +\lambda_n \leq 1-(x_1+\dots+x_n)\}|}{|\{(\lambda_1, \dots \lambda_n) \colon \lambda_i \geq 0 \text{ for all $i$ and } \lambda_1+ \dots +\lambda_n \leq 1\}|}\\
    &= (1-(x_1+\dots+x_n))^n.
\end{align*}
In particular, if $x \in (1- \frac{\alpha}{2n})S'$, we have $x_1+\dots+x_n \leq 1- \frac{\alpha}{2n}$ and thus
$ \mathbb{P}_\lambda(x\in Q_\lambda) \geq \frac{\alpha^n}{(2n)^n}.$
Putting all together, we conclude
\begin{align*}
    \begin{split}
        \mathbb{E}_\lambda |X \cap Q| &= \int_{x \in X} \mathbb{P}_\lambda(x\in Q_\lambda)\geq \int_{x \in X\cap (1-\frac{\alpha}{2n}) S'} \mathbb{P}_\lambda(x\in Q_\lambda)\geq \int_{x \in X\cap (1-\frac{\alpha}{2n}) S'} \frac{\alpha^n}{(2n)^n}\\
        &= \left| X\cap \left(1-\frac{\alpha}{2n}\right) S'\right| \frac{\alpha^n}{(2n)^n}\geq \frac{\alpha}{2} \frac{\alpha^n}{(2n)^n}|S'|\geq \frac{\alpha^{n+1}}{(2n)^{n+1}}.
    \end{split}
\end{align*}
\end{proof}
 By \Cref{clm_dens_1}, we find $\lambda=(\lambda_1, \dots, \lambda_n) \in S'$ and $Q=Q_{\lambda}$ such that $|X\cap Q| \geq \frac{\alpha^{n+1}}{(2n)^{n+1}} |S'|$.  Note that $Q \subset S'$. Hence $|X\cap Q| \geq \frac{\alpha^{n+1}}{(2n)^{n+1}} |Q|$. Set $X'=X \cap Q$. 
 
 We apply \Cref{boxdensitycontrolledpartition} to the set $X'$. Thus, we construct a subset $Y \subset X' \subset X$ with $|Y| \geq \eta^{\ref{boxdensitycontrolledpartition}}_{1,\frac{\alpha^{n+1}}{(2n)^{n+1}}, n} |Q| $ such that for all $x \in Y$ and all faces $F$ of the box $Q$, we have
$|\co(\{x\} \cup F)\cap X'| = \eta^{\ref{boxdensitycontrolledpartition}}_{2,\frac{\alpha^{n+1}}{(2n)^{n+1}}, n} |Q|.$
As  $|Q| \geq |X \cap Q| \geq \frac{\alpha^{n+1}}{(2n)^{n+1}} |S'|$, this implies that $$|Y| \geq \eta^{\ref{boxdensitycontrolledpartition}}_{1,\frac{\alpha^{n+1}}{(2n)^{n+1}}, n} \frac{\alpha^{n+1}}{(2n)^{n+1}} |S'|\quad\text{ and } \quad |\co(\{x\} \cup F)\cap X| = \eta^{\ref{boxdensitycontrolledpartition}}_{2,\frac{\alpha^{n+1}}{(2n)^{n+1}}, n} \frac{\alpha^{n+1}}{(2n)^{n+1}} |S'|.$$
 Moreover, note that each face of $S'$, except the face opposite to $o$, contains a face of $Q$ inside it. Fix  $1 \leq i \leq n$ and let $F$ be the face of $Q$ such that  $F\subset \co(0,e_1,\dots,e_{i-1},e_{i+1},\dots,e_n)$. We conclude  that for $y \in Y$, we have 
 $$|\co(x_0, \dots, x_{i-1}, y, x_{i+1}, \dots, x_n) \cap X| \geq |\co(\{y\} \cup F) \cap X| \geq \eta^{\ref{boxdensitycontrolledpartition}}_{2,\frac{\alpha^{n+1}}{(2n)^{n+1}}, n} \frac{\alpha^{n+1}}{(2n)^{n+1}} |T|.$$
\end{proof}

\subsubsection{Proof of \Cref{boxdensitycontrolledpartition}}

\begin{proof}[Proof of \Cref{boxdensitycontrolledpartition}]
Set $\eta^{\ref{boxdensitycontrolledpartition}}_{1,\alpha, n}=2^{-1}\eta^{\ref{conesdensitycontrolledpartition}}_{1,\alpha/16n, n}$ and $\eta^{\ref{boxdensitycontrolledpartition}}_{2,\alpha, n}=2^{-1}\eta^{\ref{conesdensitycontrolledpartition}}_{2,\alpha/16n, n}$. By taking an affine transformation, without loss of generality, we can assume that  $Q=[-(1+\frac{\alpha}{4n}), (1+\frac{\alpha}{4n})]^n$.

Let $Q'=[-1,1]^n$. For $1 \leq i \leq n$ let $F_i$ be the face of the box $Q$ in direction $e_i$ and for $-n \leq i \leq -1$  let $F_i$ be the face of the box $Q$ in direction $-e_{-i}$. So $F_{-n}, \dots, F_{-1}, F_1, \dots F_n$ are the faces of $Q$.

Let $X'=X \cap Q' $. Note that
\begin{align*}
    \begin{split}
        |X'|&\geq |X|+|Q'|-|Q|\geq \alpha |Q|+|Q'|-|Q|=  \left( 1- (1-\alpha)\left(1+\frac{\alpha}{4n}\right)^n\right)|Q'|\geq \frac{\alpha}{2}|Q'|.
    \end{split}
\end{align*}
We apply \Cref{conesdensitycontrolledpartition} with parameter $\alpha/16n$ to the set $X'$. Thus, we construct a subset $Y \subset X'$ with $|Y| \geq \eta^{\ref{conesdensitycontrolledpartition}}_{1,\alpha/16n, n}|Q'|$ such that for all $x \in Y$ and all $j \in [-n,n] \setminus \{0\}$, we have
$|(x+C_{j, \alpha/16n}) \cap X'| \geq \eta^{\ref{conesdensitycontrolledpartition}}_{2,\alpha/16n, n} |Q'|.$ As $|Q'|=(1+\frac{\alpha}{4n})^{-n}|Q|\geq \frac{1}{2}|Q|$ this implies that $$|Y|\geq 2^{-1}\eta^{\ref{conesdensitycontrolledpartition}}_{1,\alpha/16n, n}|Q| \quad \text{ and }\quad 
|(x+C_{j, \alpha/16n}) \cap X| \geq 2^{-1}\eta^{\ref{conesdensitycontrolledpartition}}_{2,\alpha/16n, n} |Q|.$$
For $x \in Q'$ and $j \in [-n,n] \setminus \{0\}$ it  is easy to check that
$ (x+C_{j, \alpha/16n}) \cap Q \subset \co(\{x\} \cup F_j)$
 and as $X \subset Q$, this implies that
$ (x+C_{j, \alpha/16n}) \cap X\subset \co(\{x\} \cup F_j) \cap X.$ We conclude that for $x \in Y$ and  $j \in [-n,n] \setminus \{0\}$,  we have
$$|\co(\{x\} \cup F_j) \cap X| \geq 2^{-1}\eta^{\ref{conesdensitycontrolledpartition}}_{2,\alpha/16n, n} |\co(X)|. $$
\end{proof}

\subsubsection{Proof of \Cref{conesdensitycontrolledpartition}}

\begin{proof}[Proof of \Cref{conesdensitycontrolledpartition}]
Set $\eta^{\ref{conesdensitycontrolledpartition}}_{1,\alpha, n}=\eta^{\ref{conesdensitycontrolledpartition}}_{2,\alpha, n}=\alpha^{4n^2+2n-1}/2^{16n^3+14n^2+11n-5}$.
Given a number $i$, let $\overline{i} \in [n]$ be the unique number  such that $\overline{i}=i \text{ mod } n$. Set $\widehat{i}=\overline{i}$ if $i<n$ and $\widehat{i}=-\overline{i}$ if $i \geq n$. Our strategy is to apply repeatedly analogues of \Cref{stepdensitycontrolledpartition} where instead of focusing on the pair of coordinates $(n,1)$, we focus on the pair of coordinates $(\overline{i}, \overline{i+1})$.

By \Cref{bmn}, there exists a box $Q' \subset Q$, $$Q'=[0,b_1] \times \hdots \times [0,b_n]+y\quad\text{ where $y \in \mathbb{R}^n$},\quad\text{ such that $b_n=2$ and $b_n \frac{(\alpha/4)^2}8 \leq b_1, \dots,  b_{n-1} \leq b_n \frac{(\alpha/4)^2}4 $,}$$
and such that $X \cap Q'$ has density $|X\cap Q'|\geq \alpha |Q'|/4$. In particular $|Q'| \geq |Q|\alpha^{2n-2}/2^{7n-7}$

By applying \Cref{stepdensitycontrolledpartition}, we construct inductively the sequence of pairs $(X^0, Q^0), \dots, (X^{2n}, Q^{2n})$ with the following properties. Define $Q^0=Q'$ and $X^0=X \cap Q'$. Assume that at step $i-1$ we have constructed a box $$Q^{i-1}=[0,a^{i-1}_1]\times \dots \times [0,a^{i-1}_n]+x^{i-1}\quad\text{ with $a^{i-1}_{\overline{i-1}} \frac{(\alpha/4^i)^2}8 \leq  a^{i-1}_{j} \leq a^{i-1}_{\overline{i-1}} \frac{(\alpha/4^i)^2}4 $ for $j \in [n] \setminus \overline{i-1}$},$$ and a subset $X^{i-1} \subset Q^{i-1}$ with density $|X_{i-1}| \geq \alpha|Q^{i-1}|/4^i$. We apply \Cref{stepdensitycontrolledpartition} to the pair $(X^{i-1}, Q^{i-1})$ focusing on the pair of coordinates $(\overline{i-1}, \overline{i})$ to produce the pair $(X^i,Q^i)$. More precisely, we obtain a box $Q^i \subset Q^{i-1}$ with $$Q^i=[0,a^{i}_1]\times \dots \times [0,a^{i}_n]+x^{i}\quad\text{ where $a^{i}_{\overline{i}} \frac{(\alpha/4^{i+1})^2}8 \leq a^i_j \leq a^{i}_{\overline{i}} \frac{(\alpha/4^{i+1})^2}4 $ for $j \in [n] \setminus \overline{i}$.}$$
We also obtain the set $X^i=X^{i-1} \cap Q^i$. These have the properties $|Q^i| \geq (\alpha/4^i)^{2n}|Q^{i-1}|/2^{7n}$, $|X^i| \geq \alpha |Q^i|/4^{i+1}$, and
$$|(x+C_{\widehat{i-1}, \alpha}) \cap X^{i-1}|\geq |(x+C_{\widehat{i-1}, \alpha/4^i}) \cap X^{i-1}| \geq \alpha |Q^{i-1}|/4^i\qquad \text{ for all $x \in X^i$.}$$
Because $X^{2n}\subset \dots \subset X^0 \subset X$, $Q^{2n} \subset \dots \subset Q_0$ and $\{\widehat{0}, \widehat{1}, \dots, \widehat{2n-1}\}=[-n,n] \setminus \{0\}$, it follows that for $x \in X^{2n}$ and $j \in [-n,n] \setminus \{0\}$ we have $|(x+C_{j, \alpha}) \cap X| \geq \alpha |Q^{2n}|/4^{2n}$. Moreover, $|X^{2n}| \geq \alpha |Q^{2n}|/4^{2n+1} $. 

It remains to evaluate $|Q^{2n}|$. Recall that $|Q^0| \geq |Q|\alpha^{2n-2}/2^{7n-7}$.  Using the recurrence relation $$|Q^i| \geq (\alpha/4^i)^{2n}|Q^{i-1}|/2^{7n} \geq \alpha^{2n}|Q^{i-1}|/2^{8n^2+7n},$$ we get $|Q^{2n}| \geq |Q^0| \alpha^{4n^2}/2^{16n^3+14n^2}$. These imply $|Q^{2n}|\geq \alpha^{4n^2+2n-2}|Q|/2^{16n^3+14n^2+7n-7}$.

Thus, we conclude that the subset $X^{2n} \subset X$ satisfies $$|X^{2n}| \geq \alpha |Q^{2n}|/4^{2n+1}\geq  \alpha^{4n^2+2n-1}|Q|/2^{16n^3+14n^2+11n-5},$$ and for all $x \in X^{2n}$ and $j \in [-n,n] \setminus \{0\}$ we have $$|(x+C_{j, \alpha}) \cap X| \geq \alpha |Q^{2n}|/4^{2n} \geq \alpha^{4n^2+2n-1}|Q|/2^{16n^3+14n^2+11n-7}.$$
\end{proof}

\subsubsection{Proof of \Cref{stepdensitycontrolledpartition}}

\begin{proof}[Proof of \Cref{stepdensitycontrolledpartition}]
\begin{clm}\label{xts}
 Given a box $R=[0,c_1]\times \dots \times [0,c_n]+z$ where $z \in \mathbb{R}^n$ and given a parameter $s \leq \min(c_1, \dots, c_n)$, we can construct a partition into boxes with disjoint interiors $R= \sqcup_{i=1}^k R^i$ where  $R^i=[0,c^i_1]\times \dots \times [0,c^i_n]+z^i$ such that $z^i \in \mathbb{R}^n$ and $c^i_1=c_1$ and $s/2\leq c^i_2, \dots, c^i_n \leq s$.
\end{clm}

\begin{proof}[Proof of \Cref{xts}]
For $2 \leq j \leq n$ partition the interval $[0,c_j]$ into consecutive intervals $[0,c_j]=I_1^j \sqcup \dots \sqcup I^j_{\lfloor 2c_j/s \rfloor}$ where $$I^j_k=[(k-1)s/2, ks/2]\quad\text{ for }1 \leq k \leq \lfloor 2c_j/s \rfloor-1,\qquad I^j_{\lfloor 2c_j/s \rfloor}=[(\lfloor 2c_j/s \rfloor-1)s/2, c_j].$$ Let $[0,c_1]=I_1^1$. Finally, for $k_1=1, 1 \leq k_2 \leq \lfloor 2c_2/s \rfloor, \cdots, 1 \leq k_n \leq \lfloor 2c_n/s \rfloor $ construct the box $R^{k_1,k_2, \dots, k_n}= I^1_{k_1} \times \dots \times I^n_{k_n}+z$. 

Note that $R= \sqcup_{k_1, \dots, k_n}R^{k_1,k_2, \dots, k_n}$ is a partition into boxes with disjoint interiors. Moreover, the box $R^{k_1,k_2, \dots, k_n}=I^1_{k_1} \times \dots \times I^n_{k_n}+z$ satisfies $|I^1_{k_1}|=c_1$, and for $2 \leq j \leq n$ we have $|I^j_{k_j}|=s/2 $ if $1 \leq k_j< \lfloor 2c_j/s \rfloor $, and $|I^j_{\lfloor 2c_j/s \rfloor}|= c_j- (\lfloor 2c_j/s \rfloor-1)s/2 \in [s/2,s]$. We conclude that this partition into boxes has the desired properties.
\end{proof}

\begin{clm}\label{bmn}
Given a box $R=[0,c_1]\times \dots \times [0,c_n]+z$ where $z \in \mathbb{R}^n$ and $c_1 \leq 2 \min(c_2, \dots, c_n)$ and given a subset $Z \subset R$ with density $|Z| \geq \alpha |R|/4$, there exist a sub-box of $R$, $Q=[0,b_1] \times \hdots \times [0,b_n]+y$ where $y \in \mathbb{R}^n$ such that $b_1=c_1$ and $b_1 (\alpha/4)^2/8 \leq b_2, \dots,  b_{n} \leq b_1 (\alpha/4)^2/4 $, and such that $Y:=Z \cap Q$ has density $|Y|\geq \alpha |Q|/4$.
\end{clm}

\begin{proof}[Proof of \Cref{bmn}]
Set $s=c_1(\alpha/4)^2/4$ and note that, because $c_1 \leq 2 \min(c_2, \dots, c_n)$, we have $s \leq \min(c_1, \dots, c_n)$.

By \Cref{xts}, we can construct a partition into boxes with disjoint interiors $R= \sqcup_{i=1}^k R^i$ where  $R^i=[0,c^i_1]\times \dots \times [0,c^i_n]+z^i$ such that $z^i \in \mathbb{R}^n$ and $c^i_1=c_1$ and $  c_1(\alpha/4)^2/8 \leq c^i_2, \dots, c^i_n \leq c_1(\alpha/4)^2/4$. A simple averaging argument shows that there exists an index $i$ such that $\frac{|Z \cap R^i|}{|R_i|}\geq \frac{|Z \cap R|}{|R|}= \alpha/4$.

Set $Q=R^i$, $Y=Z \cap Q$, $b_1=c_1^i, \dots, b_n=c_n^i$ and $y=z^i$. It is easy to check that $Q=[0,b_1] \times \hdots \times [0,b_n]+y$ where $b_1=c_1$ and $b_1 (\alpha/4)^2/8 \leq b_2, \dots,  b_{n} \leq b_1 (\alpha/4)^2/4 $ and $|Y|\geq \alpha |Q|/4$.
\end{proof}

Fix $0\leq r \leq a_n(1-\alpha/4)$ with the following property. Consider the partition of $P$ into three box with disjoint interiors $P=R_1 \sqcup R_2 \sqcup R_3$ where $R_1=[0,a_1] \times \dots \times [0,r]+x$, $R_2=[0,a_1] \times \dots \times [r,r+a_n\alpha/4]+x$ and $R_3=[0,a_1] \times \dots \times [r+a_n\alpha/4, a_n]+x$. Let $Z_1= X\cap R_1$, $Z_2=X \cap R_2$ and $Z_3=X \cap R_3$. We choose $r$ such that $|Z_1|=|Z_3|$.

\begin{clm}\label{fsr}
$a_1 \leq 2\min(a_2, \hdots, a_{n-1},r)$ and  $|Z_1|=|Z_3| \geq \alpha |P|/4 \geq \alpha|R_1|/4$. 
\end{clm}

\begin{proof}[Proof of \Cref{fsr}]
By construction $|Z_2| \leq |R_2| \leq \alpha |P|/4$. By hypothesis, $|Z_1|+|Z_2|+|Z_2|=|X| \geq \alpha |P|$. Therefore,  $|Z_1|+|Z_3| \geq 3\alpha |P|/4$. As $|Z_1|=|Z_3|$, we conclude $|Z_1|=|Z_3| \geq \alpha |P|/4 \geq \alpha|R_1|/4$.

From hypothesis $a_n \alpha^2/8 \leq a_1, \dots,  a_{n-1} \leq a_n \alpha^2/4 $, it immediately follows that $a_1 \leq 2\min(a_2, \hdots, a_{n-1})$. For $a_1 \leq 2r$, note that $$r=a_n \frac{|R_1|}{|P|}\geq a_n\frac{|Z_1|}{|P|}\geq a_n\frac{\alpha}4> a_n\frac{\alpha^2}8.$$
\end{proof}

\begin{obs}\label{opd}
It follows immediately from the definition of $C_{n ,\alpha}$ and the hypothesis $a_1, \dots,  a_{n-1} \leq a_n \alpha^2/4 $  that for $y \in R_1$ we have $y+C_{n, \alpha} \supset R_3$. In particular, by \Cref{fsr}, for $y \in Z_1$ we have $|(y+C_{n, \alpha})\cap X|\geq \alpha |P|/4 $.
\end{obs}

To finish the proof, let recall $R_1=[0,a_1] \times \dots \times [0,r]+x$ and $Z_1= R_1 \cap X$. By \Cref{fsr}, we have $a_1 \leq 2 \min(a_2, \dots, a_{n_1},r)$ and $|Z_1| \geq \alpha |R_1|/4$. By \Cref{bmn}  there exist a sub-box of $R_1$, $$Q=[0,b_1] \times \hdots \times [0,b_n]+y\quad \text{ where $y \in \mathbb{R}^n$, such that $b_1=a_1$ and $b_1 (\alpha/4)^2/8 \leq b_2, \dots,  b_{n} \leq b_1 (\alpha/4)^2/4 $},$$ and such that $Y:=Z_1 \cap Q=X\cap Q$ has density $|Y|\geq \alpha |Q|/4$. As $Y \subset Z_1$, by \Cref{opd} for all $y \in Y$ 
$|(y+C_{n,\alpha}) \cap X| \geq \alpha |P|/4$. 

It remains to check that $|Q| \geq \alpha^{2n}/4^{3n-2}2^{n-1}|P|$. This easily follows from the aforementioned fact that $b_1=a_1$, $a_1 (\alpha/4)^2/4 \leq b_2, \dots, b_n$ and the fact that $a_1 \geq 2^{-1} \max(a_2, \dots, a_{n-2}) $ and $a_1 \geq a_n\alpha^2/8$ (the last fact follows from the hypothesis $a_n \alpha^2/8 \leq a_1, \dots,  a_{n-1} \leq a_n \alpha^2/4 $).
\end{proof}

\subsubsection{Proof of \Cref{subsetdoubling}}
\begin{proof}[Proof of \Cref{subsetdoubling}]
Let the set $H_1,\dots,H_m\subset\mathbb{R}^n$ be a finite collection of halfspaces so that $C=\bigcap H_i$. Consider the sequence $A_i:=A\cap \bigcap_{j=1}^i H_i$. Construct another set of halfspaces $H'_i$ and subsets $B_i\subset B$ defined by finding a halfspace $H_i'$ parallel to $H_i$ so that $|H_i'\cap B_{i-1}|=|A_i|$. Note that because $H_i$ is parallel to $H_i'$ we find that
$tA_i+(1-t)B_i$ is disjoint from $t(A_{i-1}\setminus A_i)+(1-t) (B_{i-1}\setminus B_i)$. By the Brunn-Minkowksi inequality we find
\begin{align*}
    |tA_{i-1}+(1-t)B_{i-1}|-|A_{i-1}|&\geq |tA_i+(1-t)B_i|-|A_{i}|+|t(A_{i-1}\setminus A_i)+(1-t) (B_{i-1}\setminus B_i)|-|A_{i-1}\setminus A_i|\\
    &\geq |tA_i+(1-t)B_i|-|A_{i}|.
\end{align*}
The lemma follows by induction.
\end{proof}

\section{Proof of the Linear Theorem for few vertices (\Cref{main_thm_2})}

\begin{proof}[Proof of \Cref{main_thm_2}]
Let $c^{\ref{main_thm_2}}=c^{\ref{LinearThm}}$ and $k_n^{\ref{main_thm_2}}=\binom{v}{n}$.
Let $\lambda_n$ be anything, and choose $\eta_n$ so that $(1+\eta_n)^{-n}=\frac12$.
By \Cref{boundedsandwichreduction}, we may assume that $A,B$ form a simple $\lambda_n$-bounded $\eta_n$- sandwich.

Consider the origin $o\in A\cap B$ and a triangulation $\mathcal{T}'$ of the boundary of $\co(A)$. Note that $|\mathcal{T}'|\leq \binom{v}{n}$. Now consider the triangulation $\mathcal{T}$ of $\partial\co(A)$ obtained by adding $o$ to each of the simplices of $\mathcal{T}'$, i.e.
$$\mathcal{T}:=\{\co(S'\cup\{o\}):S'\in\mathcal{T}'\}.$$
We will consider $|\co(A)\setminus A|$ inside each of the simplices in $\mathcal{T}$. Write $A_{S'}$ for $A\cap S'$ and note that $\co(A_{S'})=S'$. Using \Cref{subsetdoubling}, find $B_{S'}\subset B$ so that $|A_{S'}|=|B_{S'}|$ and
$|tA_{S'}+(1-t)B_{S'}|\leq |tA+(1-t)B|\leq  \delta |A|.$
Distinguish two cases; either $\delta|A|< \min\{t,1-t\}^n|A_{S'}|$ or $\delta |A|\geq \min\{t,1-t\}^n|A_{S'}|$. In the former case, we can apply \Cref{LinearThm} to find
$|\co(A_{S'})\setminus A_{S'}|\leq \min\{t,1-t\}^{-c^{\ref{LinearThm}}n^{8}} \delta|A|.$
In the latter case we use that $A$ is a $\eta_n$ sandwich so that
$\frac{1}{1+\eta_n}\co(A_{S'})\subset A_{S'}$, which implies
\begin{align*}|\co(A_{S'})\setminus A_{S'}|&\leq \left|\co(A_{S'})\setminus \frac{1}{1+\eta_n}\co(A_{S'})\right|\leq \left(1-(1+\eta_n)^{-n}\right)|\co(A_{S'})|\\
&\leq \frac12 |\co(A_{S'})|\leq |A_{S'}|\leq \min\{t,1-t\}^{-n}\delta |A|.
\end{align*}
Hence, in both cases we find
$|\co(A_{S'})\setminus A_{S'}|\leq \min\{t,1-t\}^{-c^{\ref{LinearThm}}n^{8}} \delta|A|.$
We conclude by adding up the contributions over all ${S'}\in\mathcal{T}$,
$$|\co(A)\setminus A|=\sum_{{S'}\in\mathcal{T}} |\co(A_{S'})\setminus A_{S'}|\leq |\mathcal{T}|\min\{t,1-t\}^{-c^{\ref{LinearThm}}n^{8}} \delta|A|\leq \binom{v}{n} \min\{t,1-t\}^{-c^{\ref{LinearThm}}n^{8}} \delta|A|.$$
This concludes the proof of the theorem.
\end{proof}

\section{Intermediate results for the Linear Theorem: from few vertices (\Cref{main_thm_2}) to the general case (\Cref{LinearThmGeneral})}

\subsection{Outline of the proof of \Cref{LinearThmGeneral}}

The proof of \Cref{LinearThmGeneral} follows the following steps.

\begin{enumerate}
    \item By \Cref{boundedsandwichreduction}, we may assume $B(o,0.01)\subset K\subset A,B\subset 1.01 K\subset B(0,100)$ for some convex set $K$.
    \item Find a collection of disjoint convex regions $X_i$ with small diameter $\xi$ so that $\sum_i|\co(A\cap X_i) \setminus (A\cap X_i)| \geq \Omega_n(|\co(A) \setminus A|)$ (\Cref{DensityDiamLemma}) along the following steps.
    \begin{itemize}
        \item Cone off at the origin to find simplices $T_i$ so that $\bigcup T_i=\co(A)$ and $\co(T_i\cap A)=T_i$. Note that $|T_i\cap A|\geq 0.99 |T_i|$.
        \item In each of the $T_i$ run the following process: find a point in $A$ centrally in $T_i$ and cone off to the vertices of $T_i$ to find smaller simplices with the same properties as $T_i$ and not much lower density. (\Cref{findpointcopy})
        \item Iterate until either the diameter is small (in which case we have found our region $X_i$) or the density has dropped (say between 0.9 and 0.95). (\Cref{LowDensityReduction})
        \item For simplices with moderate density, it is easy to find a set of small diameter and a positive proportion of the missing region. (\Cref{LowDensityCaseInductionStep})
    \end{itemize}
    \item Consider the tube $U=S\times \mathbb{R}^+$ where $S \subset \mathbb{R}^{n-1}$ is a regular simplex with side length $\epsilon$ centered at the origin. We choose $\epsilon\ll 0.01$, so that $S$ falls well within $B(o,0.01)$. On the other hand, we choose $\epsilon\gg \xi$, so that a random rotation of $U$ contains a region $X_i$ completely with probability $\Omega_n(1)$.
    \item Note that $|\co(A\cap U)\setminus (A\cap U)|\geq \sum_{X_i\subset U} |\co(A\cap X_i) \setminus (A\cap X_i)|$, so that it remains to show $|\co(A\cap U)\setminus (A\cap U)|\leq O_{n,t}(|tA+(1-t)B|-|A|)$. (\Cref{CoAintubes})
    \item To this ends find a homothetic tube $V=x+(1+\beta)U$ so that $|A\cap U|=|B\cap V|$ and $|t(A\cap U)+(1-t)(B\cap V)|-|A\cap U|\leq |tA+(1-t)B|-|A|$, which moreover is very similar to $U$, so that $x$ and $\beta$ are very small. (\Cref{SmallTubeLemma})
    \item Slightly rescaling $A\cap U$ and $B\cap V$ (i.e., taking homotheties with factors very close to 1), we find $A',B'$ in a tube over the same simplex $W$ so that $2|\co(A')\setminus A'|\geq |\co(A\cap U)\setminus (A\cap U)|$ and a $t'\in(\frac12 t, 2t)$ so that  $|t'A'+(1-t')B'|-t'|A'|-(1-t')|B'| \leq  |tA+(1-t)B| -|A|.$ (\Cref{ChangeoftLemma}) 
    \item Partition the tube $W$ into parallel simplicial tubes $U_i$ according to the convex hull of $A'$, so that $\bigcup_i \co(A'\cap U_i)=\co(A')$ and each of the $\co(A'\cap U_i)$ has exactly $2n$ vertices. Note that we have $$\sum_i |t'(A'\cap U_i)+(1-t')(B'\cap U_i)|-t'|A'\cap U_i|-(1-t')|B'\cap U_i|\leq |t'A'+(1-t')B'|-t'|A'|-(1-t')|B'|.$$
    \item Consider the fibres in the direction of the tube and note that $A'$ and $B'$ start with a long interval in each of these fibres by the virtue of $S$ lying well within $B(o,0.01)$. We might not have $|A'\cap U_i|=|B'\cap U_i|$, but because of the long interval at the beginning of each fibre, we may extend $A'$ and $B'$ without affecting $|t'(A'\cap U_i)+(1-t')(B'\cap U_i)|-t'|A'\cap U_i|-(1-t')|B'\cap U_i|$. (\Cref{extendingTubes})
    \item Because $K\subset A,B\subset 1.01 K$, we never have to extend much to get $|A'\cap U_i|=|B'\cap U_i|$. Hence, we can apply \Cref{main_thm_2} to find that $$|\co(A'\cap U_i)\setminus (A'\cap U_i)|\leq O_{n,t}(|t'(A'\cap U_i)+(1-t')(B'\cap U_i)|-t'|A'\cap U_i|-(1-t')|B'\cap U_i|),$$ so that summing over $U_i$ gives the result.
\end{enumerate}

\subsection{Setup}
\begin{defn}\label{tubedef}
A \emph{vertical tube} with diameter $\ell$ is a set $U$ of the form $U=S \times \mathbb{R}_+$ where $S \subset \mathbb{R}^{n-1}$ is a regular simplex with side length $\ell$ centered at the origin. A \emph{tube} with diameter $\ell$ is a rotated vertical tube with diameter $\ell$. 
\end{defn}

\begin{defn}\label{tubularsandwichdef}
A pair $X,Y\subset \mathbb{R}^n$ is a \emph{tubular $\lambda$-bounded $\eta$-sandwich} if it is the intersection between the same tube and a $\lambda$-bounded $\eta$-sandwich.
\end{defn}

\subsection{Propositions}

\begin{prop}\label{DensityDiamLemma}
 For $n\in\mathbb{N}$ and $\xi>0$, there are constants $\theta>0$ and $\eta>0$ such that the following hold. Assume $A \subset \mathbb{R}^n$ is a measurable  $\eta$-sandwich, then there exist disjoint measurable subsets $A_1, \dots, A_k$ of $A$ with the following properties:
 \begin{enumerate}
    \item $A_i=A\cap \co(A_i)$,
     \item $\co(A_1), \dots, \co(A_k)$ are disjoint,
     \item $\sum_i|\co(A_i) \setminus A_i| \geq \theta |\co(A) \setminus A|$, 
     \item for every $i$, $\text{diam}(\co(A_i))< \xi \text{diam} (\co(A))$. 
 \end{enumerate}
\end{prop}

\begin{prop}\label{CoAintubes}
For $n\in\mathbb{N}$, $\lambda>1$, $\varepsilon>0$ sufficiently small in terms of $n$ and $\lambda$, and $0<t\leq \frac12$, there exists $\eta>0$ such that the following hold. Assume  $U \subset \mathbb{R}^n$ is a tube with diamater $\varepsilon$ and assume $A,B\subset\mathbb{R}^n$ form a $\lambda$-bounded, $\eta$-sandwich of measurable sets of volume $1$, then
$$|\co(A\cap U)\setminus (A\cap U)|\leq k^{\ref{CoAintubes}}_n t^{-c^{\ref{CoAintubes}}n^8}(|tA+(1-t)B|-|A|).$$
\end{prop}

\begin{prop}\label{LinearFunctionsProp}For $n\in\mathbb{N}, \lambda>1$, $\varepsilon>0$ sufficiently small in terms of $n$ and $\lambda$, and $t\in(0,1)$, there exists $\eta_{n,\lambda,\varepsilon,t}^{\ref{LinearFunctionsProp}},k_n^{\ref{LinearFunctionsProp}},c^{\ref{LinearFunctionsProp}}>0$ such that the following holds. Assume $W \subset \mathbb{R}^n$ is a tube or diameter $\epsilon$ and assume the subsets $A',B' \subset W$ form a tubular $\lambda$-bounded $\eta$-sandwich. Then $$|\co(A')\setminus A'| \leq k^{\ref{LinearFunctionsProp}}_n(t')^{-c^{\ref{LinearFunctionsProp}}n^8}(|t'A'+(1-t')B'|-t'|A'|-(1-t')|B'|).$$
\end{prop}

\subsection{Lemmas}

\begin{lem}\label{findpointcopy}
For all $n \in \mathbb{N}$ there exists $\alpha>0$ depending $n$ such that the following holds. If $A$ is a subset of $\mathbb{R}^n$ with $\co(A)$ a simplex with vertices $V(\co(A))=\{x_0, \dots, x_n\}$ and if $|\co(A)| \leq (1+\alpha) |A|$, then there exists a point $x \in A$ such that for all $0 \leq i \leq n$
$$\frac{|\co\{x_0, \dots, x_{i-1}, x, x_{i+1}, \dots, x_n\}|}{|\co\{x_0, \dots, x_{i-1}, x_i, x_{i+1}, \dots, x_n\}|} \geq \frac{1}{n+2}\quad
\text{ and }\quad
d(x,x_i) \leq \frac{n+1}{n+2}\max_{j,k}d(x_j,x_k).$$
\end{lem}

\begin{lem}\label{LowDensityReduction}
For all $n\in\mathbb{N}$ and $\alpha>0$ sufficiently small in terms of $n$, there exists an $\eta>0$ so that the following holds.
Given an $\eta$-sandwich $A$ there exists an essential partition into convex sets $\co(A)=P_1 \sqcup \dots \sqcup P_k$ so that denoting $A_i=P_i\cap A$ we have:
\begin{enumerate}
\item $P_i=\co(A_i)$,
\item $|\co(A_i)\cap\co(A_j)|=0$ if $i\neq j$,
\item $\co(A_i)$ is a simplex,
\item Each $A_i$ fall into two categories, i.e., $[k]=I\sqcup J$ so that: 
\begin{itemize}
\item for all $i\in I$, we have $\frac{|\co(A_i)|}{|A_i|}\in (1+\alpha, 1+n!\alpha)$, 
\item $\sum_{I\in J} |\co(A_i)\setminus A_i|\leq \alpha |\co(A)\setminus A|$.
\end{itemize}
\end{enumerate}
The same conclusion holds if we replace the hypothesis that $A$ is an $\eta$-sandwich with the hypothesis that $\co(A)$ is a simplex with $|\co(A)|\leq (1+\alpha)|A|$.
\end{lem}

\begin{lem}\label{LowDensityCaseInductionStep}
For all $n\in\mathbb{N}$, $\alpha>0$ sufficiently small in terms of $n$, there exist $\zeta,\beta>0$ so that the following holds.
Given $A\subset\mathbb{R}^n$ so that $\co(A)$ is a simplex and $|\co(A)|= (1+\alpha)|A|$, there exists a subset $A'\subset A$ and $\psi\geq \zeta$, so that 
\begin{enumerate}
    \item $diam(A')\leq (1-\psi) diam(A)$,
    \item $|\co(A')|\geq (1-\psi)^{3n^2}|co(A)|$,  
    \item $\co(A')$ is a simplex,
    \item $|\co(A')\setminus A'|\geq \beta|A'|$.
\end{enumerate}  
\end{lem}

\begin{lem}\label{SmallTubeLemma}
For $n\in\mathbb{N}$, $\lambda>1$, for all $\varepsilon>0$ sufficiently small in terms of $\lambda$ and $n$, so that for all $0<t\leq \frac12$, $\alpha>0$ there exists $\eta>0$ such that the following hold. Assume  $U \subset \mathbb{R}^n$ is a tube with diamater $2\varepsilon$ and assume $A,B\subset\mathbb{R}^n$ form a $\lambda$-bounded, $\eta$-sandwich measurable sets of volume $1$, then
there exists a homothetic tube $V=x+(1+\beta)U$ with $||x||_2 \leq \alpha$ and $|\beta|\leq \alpha$ such that $A^\circ=A\cap U$ and $B^\circ=B\cap V$ satisfy $|A^\circ|=|B^\circ|$ and $$|tA^\circ+(1-t)B^\circ| -|A^\circ|\leq |tA+(1-t)B|-|A|.$$
\end{lem}

\begin{lem}\label{ChangeoftLemma}
For all $\eta'>0$, $n\in\mathbb{N}$, $t\in(0,1/2]$, $\lambda>1$ and sufficiently small $\epsilon$ in terms of $n$ and $\lambda$, there exist $\eta,\alpha>0$, so that the following holds.
Let $t'=\frac{t}{1+\beta-\beta t}$. Assume  $U \subset \mathbb{R}^n$ is a tube of diameter $\epsilon$, $V$ is a homothetic tube $V=x+(1+\beta)U$ with $|\beta|,||x||_2\leq \alpha$ and assume  $A,B$ is a $\lambda$-bounded $\eta$-sandwich with $|A|=|B|=1$ and $A^\circ :=A\cap U$ and $B^\circ := V\cap B$ have equal size. Set $W=\frac{t}{t'}U$, $W'=\frac{1-t}{1-t'}V$, $A'=\frac{t}{t'}A^\circ$ and $B'=\frac{1-t}{1-t'}B^\circ$. Then the tube $W'=W+\frac{1-t}{1-t'}x$ has diameter less than $2\epsilon$, $A' \subset W$ and $B' \subset W'$. Moreover, we have $$ |t'A'+(1-t')B'|-t'|A'|-(1-t')|B'| \leq  |tA^\circ+(1-t)B^\circ| -|A^\circ|.$$
Furthermore, $A',B'-\frac{1-t}{1-t'}x$ form a tubular $2\lambda$-bounded $\eta'$-sandwich.
Finally, if $A^\circ,B^\circ$ are $\alpha$-almost convex, then $|A^\circ \triangle B^\circ|\leq |A'\triangle (B'-\frac{1-t}{1-t'}x)|+O_{n,\epsilon}(\alpha|A^\circ|)$

\end{lem}

\begin{lem}\label{extendingTubes}
For $n\in\mathbb{N}, \lambda>1$, $\varepsilon>0$ sufficiently small in terms of $n$ and $\lambda$, and $t\in(0,1)$, there exists $\eta>0$ such that the following holds. Given a convex $K\subset\mathbb{R}^n$ which is $\lambda$-bounded and a tube $U=S\times \mathbb{R}_+$ of radius $\varepsilon$. Then for all $x\in S\times \mathbb{R}_-$, we have $(1-t)x+t(U\cap ((1+\eta)K\setminus K))\subset K\cup (S\times \mathbb{R}_-)$.
\end{lem}

\subsection{Proof of Propositions}

\subsubsection{Proof of \Cref{DensityDiamLemma}}
\begin{proof}[Proof of \Cref{DensityDiamLemma}]
Choose parameters according to the following hierarchy
$$n,\xi\gg \alpha\gg \zeta\gg \beta\gg \theta\gg \eta.$$
We create a sequence of families $\mathcal{F}_i$ of sets which all have properties 1 and 2. They have property 3 with ever decreasing $\theta$ while the diameter of the sets decreases until property 4 is satisfied. All sets in the families will have a simplex as their convex hull. Given a family $\mathcal{F}$ of subsets of $A$ satisfying properties 1 and 2, write $D(\mathcal{F}):=\max_{X\in\mathcal{F}}\frac{diam(X)}{diam(A)}$ and $P(\mathcal{F}):=\frac{\sum_{X\in \mathcal{F}}|\co(X) \setminus X|}{|\co(A) \setminus A|}$, so that we want to find a family with $D(\mathcal{F})<\xi$ and $P(\mathcal{F})\geq \theta$.

First apply \Cref{LowDensityReduction} with parameter $\alpha$ to find a partition and let $\mathcal{F}_0$ be the set of parts so that $\frac{|\co(A_i)|}{|A_i|}\in (1+\alpha, 1+n!\alpha)$. Note that \Cref{LowDensityReduction} says that $P(\mathcal{F}_0)\geq 1-\alpha $.

Consider the following process of obtaining a new family $\mathcal{F}_{2i+1}$ from $\mathcal{F}_{2i}$ and $\mathcal{F}_{2i+2}$ from $\mathcal{F}_{2i+1}$. All elements $X$ in $\mathcal{F}_{i}$ with $diam(X)\leq \xi diam (A)$ remain fixed and stay in all consequent families, we will call them \emph{finished}. The idea will be to apply \Cref{LowDensityCaseInductionStep} to all other elements in $\mathcal{F}_{2i}$ to reduce the diameter. Then we apply \Cref{LowDensityReduction} to some of the elements in $\mathcal{F}_{2i+1}$ so that the density of the sets in $\mathcal{F}_i$ in their convex hull always stays below $1/(1+\alpha)$.

 All non-finished elements $X\subset A$ in $\mathcal{F}_{2i}$ have $\frac{|\co(X)\setminus X|}{|\co(X)|}\geq \alpha$. If we consider the last subset $Y\in \bigcup_{j<2i} \mathcal{F}_{j}$, so that $X\subset Y$ and $\frac{|\co(Y)\setminus Y|}{|\co(Y)|}\leq n!\alpha$, then $X$ derives from $Y$ by a sequence of applications of \Cref{LowDensityCaseInductionStep}. Hence, as $diam(X)>\xi diam(Y)$, we know that $$\frac{|\co(X)\setminus X|}{|\co(X)|}\leq \frac{|\co(Y)\setminus Y|}{|\co(Y)|}\cdot\frac{|\co(Y)|}{|\co(X)|}<n!\alpha \xi^{-3n^2}.
 $$
 Choosing $\alpha$ sufficiently small in terms of $\xi$ and $n$, we can apply \Cref{LowDensityCaseInductionStep} to all non-finished elements of $\mathcal{F}_{2i}$ and construct $\mathcal{F}_{2i+1}$ with $D(\mathcal{F}_{2i+1})\leq (1-\zeta)D(\mathcal{F}_{2i})$ and $P(\mathcal{F}_{2i+1})\geq \beta P(\mathcal{F}_{2i})$.
 
To construct $\mathcal{F}_{2i+2}$ from $\mathcal{F}_{2i+1}$, apply \Cref{LowDensityReduction} to all non-finished elements with $|\co(A)\setminus A|<\alpha |A|$. Note that $D(\mathcal{F}_{2i+2})\leq D(\mathcal{F}_{2i+1})$ and $P(\mathcal{F}_{2i+2})\geq (1-\alpha)P(\mathcal{F}_{2i+1})$.

This construction implies that there exists a $i\leq \frac{\log(\xi)}{\log(1-\zeta)}+1$, so that $D(\mathcal{F}_{2i+1})\leq \xi$. For this $i$, we thus find $P(\mathcal{F}_{2i+1})\geq \beta^{i+1}(1-\alpha)^i\geq \theta$, where the last inequality follows from choosing $\theta$ sufficiently small in terms of $n,\xi,\alpha,\zeta$, and $\beta$. Hence, this family satisfies all the conclusions and thus confirms the proposition.

\end{proof}

\subsubsection{Proof of \Cref{CoAintubes}}

\begin{proof}[Proof of \Cref{CoAintubes}]
Let $\varepsilon$ sufficiently small in terms of $n$ and $\lambda$ to apply \Cref{SmallTubeLemma},\Cref{ChangeoftLemma}, and \Cref{LinearFunctionsProp}. Let $\eta':= \eta_{n,\lambda,\varepsilon,t}^{\ref{LinearFunctionsProp}}$ so that we can apply \Cref{LinearFunctionsProp}. Choose $\alpha$ sufficiently small to apply \Cref{ChangeoftLemma}. Let $\eta$ sufficiently small so that we can apply \Cref{SmallTubeLemma}, and \Cref{ChangeoftLemma}.
Let $A^\circ:=U\cap A$.
First use \Cref{SmallTubeLemma} to find $V=x+(1+\beta)U$  with $||x||_2\leq \alpha$ and $|\beta|\leq\alpha$. Let $B^\circ=B\cap V$, so that $|tA^\circ+(1-t)B^\circ| -|A^\circ|\leq |tA+(1-t)B|-|A|.$ 

As in \Cref{ChangeoftLemma}, let $t'=\frac{t}{1+\beta-\beta t}$. Set $W=\frac{t}{t'}U$, $W'=\frac{1-t}{1-t'}V$, $A'=\frac{t}{t'}A^\circ$ and $B'=\frac{1-t}{1-t'}B^\circ$. By \Cref{ChangeoftLemma}, the tube $W'=W+\frac{1-t}{1-t'}x$ has diameter less than $2\epsilon$, $A' \subset W$ and $B' \subset W'$, and we have $$ |t'A'+(1-t')B'|-t'|A'|-(1-t')|B'| \leq  |tA^\circ+(1-t)B^\circ| -|A^\circ|\leq |tA+(1-t)B|-|A|.$$
Furthermore, $A',B'-\frac{1-t}{1-t'}x$ form a tubular $2\lambda$-bounded $\eta'$-sandwich

 By \Cref{LinearFunctionsProp} (applied to $A'$ and $B'-\frac{1-t}{1-t'}x$), we find
\begin{align*}
|\co(A')\setminus A'| &\leq k^{\ref{LinearFunctionsProp}}_n(t')^{-c^{\ref{LinearFunctionsProp}}n^8}(|t'A'+(1-t')B'|-t'|A'|-(1-t')|B'|)\\
&\leq k^{\ref{LinearFunctionsProp}}_n(t/2)^{-c^{\ref{LinearFunctionsProp}}n^8}(|tA+(1-t)B|-|A|).
\end{align*}
By definition of $A'$, we have:
$$|\co(A\cap U)\setminus (A\cap U)|=|\co(A^\circ)\setminus A^{\circ}|=\left(\frac{t'}{t}\right)^n|\co(A')\setminus A'|\leq 2^n |\co(A')\setminus A'|.$$
Hence, we find some constant $k_n^{\ref{CoAintubes}}:=k^{\ref{LinearFunctionsProp}}_n2^n2^{c^{\ref{LinearFunctionsProp}}n^8}$, so that
$$|\co(A\cap U)\setminus (A\cap U)|=k^{\ref{CoAintubes}}_nt^{-c^{\ref{LinearFunctionsProp}}n^8}(|tA+(1-t)B|-|A|).$$
\end{proof}

\subsubsection{Proof of \Cref{LinearFunctionsProp}}

\begin{proof}[Proof of \Cref{LinearFunctionsProp}]
Let $k^{\ref{LinearFunctionsProp}}_n=k^{\ref{main_thm_2}}_n(2n)$ and $c^{\ref{LinearFunctionsProp}}=c^{\ref{main_thm_2}}$. Let $\eta=\eta_{n,\lambda,\varepsilon,t}^{\ref{LinearFunctionsProp}}$ sufficiently small to apply \Cref{extendingTubes}.

As the problem is rotation invariant, we can assume wlog that $W=W'=S\times\mathbb{R}_+$ is a vertical tube.

We can assume wlog that $\co(A')$ has $2n$ vertices ($co(A')$ is the intersection of $T$ with a halfspace). Indeed,  partition $S$ into simplices $S=S_1 \cup \dots \cup S_m$  by projecting a triangulation of the upper boundary  of $\co(A')$. Consider the corresponding partition  $T=T_1 \cup \dots \cup T_m$, where $T_i=S_i \times \mathbb{R}_+$ and set $A_i=A' \cap T_i$ and $B_i=B' \cap T_i$. Note that $A'=\cup_i A_i$, $B'=\cup_i B_i$, and $\co(A')=\cup_i\co(A_i)$ are partitions (but $\co(B')\supset \cup_i\co(B_i)$ is essentially a disjoint union). Unlike many other partitions in this paper, we generally do not have $|A_i|=|B_i|$. Moreover, as $t'A_i+(1-t')B_i$ are essentially disjoint, we find $$ \sum_i|t'A_i+(1-t')B_i|-t'|A_i|-(1-t')|B_i|\leq |t'A'+(1-t')B'|-t'|A'|-(1-t')|B'|.$$
Hence, it is enough to show that $$ |\co(A_i)\setminus A_i| \leq |t'A_i+(1-t')B_i|-t'|A_i|-(1-t')|B_i|.$$
As $A',B'$ form a tubular $\lambda$-bounded $\eta$-sandwich, we deduce that $A_i,B_i$ form a $2\lambda$-bounded $\eta$-sandwich. Moreover, by construction $\co(A_i)$ has $2n$ vertices ($co(A_i)$ is the intersection of $T_i=S_i \times \mathbb{R}_+$ with a half-space). By taking an affine transformation, we can also assume $S_i \subset \mathbb{R}^{n-1}$ is a regular simplex centered at the origin. This concludes the reduction to the case when $\co(A')$ has $2n$ vertices.

Now assume wlog $\co(A')$ has $2n$ vertices and say $|A'|\leq |B'|$ (the other case is identical).  Let  $\alpha=(|B'|-|A'|)/|S|$ and note that $0 \leq \alpha =O_{n,\lambda}(\eta)$. Let $\beta>0$ large. Construct the sets $A''=A'\sqcup S \times [-\alpha-\beta,0)$ and $B''= B'\sqcup S \times [-\beta,0)$ such that $|A''|=|B''|$. 

By \Cref{extendingTubes}, we find that
$$t'A''+(1-t')B''=S \times [-t'\alpha-\beta,0) \sqcup (t'A'+(1-t')B').$$
Recalling the definition of $A''$ and $B''$ and the fact that $|A''|=|B''|$, we deduce 
$$|t'A''+(1-t')B'|-|A''|=|t'A'+(1-t')B'|-t'|A'|-(1-t')|B'|$$
and
$|\co(A'')\setminus A''|=|\co(A')\setminus A'|.$
 Therefore, it is enough to show that
$$|\co(A'')\setminus A''| \leq k^{\ref{LinearFunctionsProp}}_n(t')^{-c^{\ref{LinearFunctionsProp}}n^8}(|t'A''+(1-t')B''|-|A''|).$$
Moreover, by making $\beta$ arbitrarily large, we can make the ratio $\frac{|t'A''+(1-t')B''|-|A''|}{|A''|}$ arbitrarily close to $0$ (as the numerator is constant but the denominator can be arbitrarily large). As $\co(A'')=\co(A')$ has $2n$ vertices, this is a simple application of \Cref{main_thm_2}.
\end{proof}

\subsection{Proof of Lemmas}

\subsubsection{Proof of \Cref{findpointcopy}}

\begin{proof}[Proof of \Cref{findpointcopy}]
By taking an affine transformation, we can assume without loss of generality that $S=\co(A)=\co\{x_0, \dots, x_n\}$ is a regular simplex with volume $1$ centered at the origin $o$.  Clearly, for all $0 \leq i \leq n$
$$\frac{|\co\{x_0, \dots, x_{i-1}, o, x_{i+1}, \dots, x_n\}|}{|\co\{x_0, \dots, x_{i-1}, x_i, x_{i+1}, \dots, x_n\}|} = \frac{1}{n+1}.$$
By continuity, there exists $\varepsilon>0$ sufficiently small in terms of $n$ such that for all $x \in \varepsilon S$ and $0 \leq i \leq n$
$$\frac{|\co\{x_0, \dots, x_{i-1}, o, x_{i+1}, \dots, x_n\}|}{|\co\{x_0, \dots, x_{i-1}, x_i, x_{i+1}, \dots, x_n\}|} \geq \frac{1}{n+2}.$$
As $|\varepsilon S| =\varepsilon^n |S|$, provided $\alpha\leq \varepsilon^n/10 $ we deduce that there exists $x \in \varepsilon S \cap A$ and it follows that $x$ has the desired properties. 

For the last part, fix $0 \leq k \leq n$. Let $q_k$ be the intersection of the ray $x_kx$ with the face opposite $x_k$. 
On the one hand, a simple computation gives
$\frac{|xq_k|}{|x_kq_k|}= \frac{|\co\{x_0,\dots,x_{i-1},x,x_{i+1},\dots,x_n\}|}{|\co\{x_0,\dots,x_n\}|}.$
From the first part, we know that
$\frac{|\co\{x_0,\dots,x_{i-1},x,x_{i+1},\dots,x_n\}|}{|\co\{x_0,\dots,x_n\}|} \geq \frac{1}{n+2}.$
Combining the last two inequalities, we get
$\frac{|xq_k|}{|x_kq_k|} \geq \frac{1}{2(n+1)} \text{, i.e., } \frac{|xx_k|}{|x_kq_k|} \leq \frac{n+1}{n+2}.$
On the other hand, because the diameter of a simplex is realized between two vertices, we have
$|x_kq_k| \leq \max_{i,j}(x_ix_j).$
Combining the last two inequalities, we find
$d(x,x_i)\leq \frac{n+1}{n+2}\max_{i,j} d(x_i,x_j),$
which concludes the proposition.
\end{proof}

\subsubsection{Proof of \Cref{LowDensityReduction}}

\begin{proof}[Proof of \Cref{LowDensityReduction}]
First we argue that we can assume without loss of generality that $P:=\co(A)$ is a simplex. 

Consider the essential partition $P=S^1 \sqcup \dots \sqcup S^r$ into simplices with a vertex at the origin $o$, obtained by partitioning the boundary $\partial P$ into $(n-1)$-dimensional simplices and conning off at $o$. Then the sets $A^j=A\cap S^j$ are $\eta$-sandwiches. 

Note that given $A$ is a finite union of boxes, we have $V(S^j) \subset A_j$ i.e., $\co(A^j)=S^j$. In particular, $|\co(A) \setminus A|=\sum_j |\co(A^j)\setminus A^j|$.

Now assume that for each $j$ we can find an essential partition into convex sets $\co(A^j)=P^j_1 \sqcup \dots \sqcup P^j_{k_j}$ and a corresponding partition of indices $[k_j]=I^j\sqcup J^j$ with the desired properties. Then it is easy to check that the  essential partition into convex sets of $\co(A)= \sqcup_{j \leq r, i \leq k_j}P^j_i$ and the corresponding partition of indices $I=\sqcup_j I^j$ and $J= \sqcup_j J^j$  have the desired properties.

Thus from now on we can assume $\co(A)$ to be a simplex. Moreover, from now on we only retain the weaker hypothesis that $\frac{|\co(A)|}{|A|} \leq (1+\eta)^n \leq 1+ \alpha$.

Given a simplex $\co(A)$, consider the following iterative process. First set $\mathcal{T}_0=\{\co(A)\}$ and $\mathcal{S}_0=\emptyset$, and note that $|\co(A)|\leq ( 1+ \alpha)|A|$ by hypothesis.
At a given stage $i$ with $\mathcal{T}_i,\mathcal{S}_i$, look at every element $S'\in \mathcal{T}_i$ and distinguish two cases: either $|S'|\leq (1+\alpha)  |S'\cap A|$ or $ |S'| \geq (1+\alpha) |S'\cap A| $.

For each simplex $S'=\co\{x_0,\dots,x_n\}\in\mathcal{T}_i$ with $|S'|\leq (1+\alpha)  |S'\cap A|$ we construct the $n+1$ simplices $f_0(S'),\dots, f_n(S')$ as follows. We apply \Cref{findpointcopy}  to find a central point $x\in S'\cap A$ and we construct the simplex $f_j(S')=\co\{x_0,\dots,x_{j-1},x,x_{j+1},\dots,x_n\}$. 

Now let
$$\mathcal{T}_{i+1}:=\bigcup_{S'\in \mathcal{T}_i:\  |S'|\leq (1+\alpha)  |S'\cap A|} \left\{f_0(S'),\dots,f_n(S')\right\}\qquad \text{and} \qquad \mathcal{S}_{i+1}:=\mathcal{S}_i \cup \left\{S'\in \mathcal{T}_i: |S'|\geq (1+\alpha)  |S'\cap A|\right\}.$$
Using the fact that $A$ is closed, it follows  by induction that for $i \in \mathbb{N}$ and  $S' \in \mathcal{T}_i \sqcup \mathcal{S}_i$ we have $\co(A\cap S')=S'$. Moreover,  $\mathcal{T}_i \sqcup \mathcal{S}_i$ forms an essential partition of $\co(A)$.

\begin{clm}\label{clm_size_a}
    For all $S' \in \mathcal{T}_i$ and $j \in [0,n]$, we have $|f_j(S')| \geq |S'|/n+2$.
\end{clm}
\begin{proof}[Proof of Claim]
By \Cref{findpointcopy} (and our choice of $x \in S' \cap A$), we have $|f_j(S')|\geq |f_j(S')\cap A|\geq |S'|/n+2$.
\end{proof}

\begin{clm}\label{clm_size_b}
For all $S'\in \mathcal{T}_i$, we have $|S'|\leq (1+(n+3)\alpha)  |S'\cap A|$ and for all $S'\in \mathcal{S}_i$  we have $  (1+\alpha)  |S'\cap A| \leq |S'|\leq (1+(n+3)\alpha)  |S'\cap A|$
\end{clm}
\begin{proof}[Proof of Claim]
Every simplex $S'\in\mathcal{S}_i\cup\mathcal{T}_i$ is $f_j(S'')$ for some $S''$ with $|S'|\leq (1+\alpha)  |S'\cap A|$. By \Cref{findpointcopy} and the fact that $\alpha>0$ is small, we have $|S''|\leq (1+(n+3)\alpha)  |S''\cap A|$. In addition, by construction of $\mathcal{S}_i$,  for all $S'\in \mathcal{S}_i$  we have $  (1+\alpha)  |S'\cap A| \leq |S'|$.
\end{proof}

Before we conclude we need two more claims. Given a simplex $S'$, let radius $rad(S')$ be the maximal length among its edges. 

\begin{clm}\label{copyvolumebigradius}
For all $r>0$ the following holds
$$\sum_{S'\in \mathcal{T}_i: rad(T)>r} |S'| \to 0 \text{ as } i \to \infty.$$
\end{clm}
\begin{proof}[Proof of Claim]
We make the convention $\bigcup \mathcal{T}_i:= \bigcup_{S' \in \mathcal{T}_i} S'$ Let $k=\lceil\log_{(1-\rho)}(r)\rceil$. We can distinguish two types of elements in $\mathcal{T}_i$. Either $S'\subset \bigcup \mathcal{T}_{i+k(n+1)} $ or not. Collect the former in $\mathcal{T}_i'$ and the latter in $\mathcal{T}_i''$.

For $S'\in\mathcal{T}''_i$, at least some simplex originating from $S'$ is in $\mathcal{S}_{i+k(n+1)}$. By the first claim, that simplex will have size at least $\frac{1}{(n+2)^{k(n+1)}}|S'|$. Hence,
$$\sum_{S'\in \mathcal{T}_{i+k(n+1)}:\ S'\subset\bigcup\mathcal{T}''_i } |S'|\leq \left(1-\frac{1}{(n+2)^{k(n+1)}}\right)\left|\bigcup\mathcal{T}''_i \right|.$$
For $S'\in\mathcal{T}_i'$, we will find an element $S''\in\mathcal{T}_{i+k(n+1)}$ with $S''\subset S'$ and $rad(S'')<r$. Let $S'=S^0$, and consider $S^{j+1}:=f_n(f_{n-1}(\dots f_0(S^j)\dots)$. Crucially, $rad(S^{j+1})\leq ((n+1)/(n+2))rad(S^j)$. Indeed, none of the edges of $S^j$ remain and all of the edges added have length at most $(1-\rho)rad(S^j)$ by \Cref{findpointcopy}. Hence, $rad(S^k)\leq r\cdot rad(S^0)\leq r$. Note that again by the first claim $|S^k|\geq \frac{1}{(n+2)^{k(n+1)}}|S'|$.

Combining these two cases gives 
$$\sum_{S'\in \mathcal{T}_{i+k(n+1)}:\ rad(S')>r} |S'| \leq \left(1-\frac{1}{(n+2)^{k(n+1)}}\right)\sum_{S'\in \mathcal{T}_{i}:\ rad(S')>r} |S'|.$$
The conclusion follows.
\end{proof}

\begin{clm}\label{neighbourhod}
    Assuming $A$ is a finite union of boxes, for every $\varepsilon>0$ there exists $r>0$ depending on $A$ and $\varepsilon$  such that for all $i \in \mathbb{N}$ the following holds $$\sum_{\substack{S'\in \mathcal{T}_i: \\ rad(T)\leq r \text{ and } S' \not \subset A}} |S'| \leq \varepsilon.$$
\end{clm}
\begin{proof}[Proof of Claim]
    Note that $\bigcup_{\substack{S'\in \mathcal{T}_i: \\ rad(T)\leq r \text{ and } S' \not \subset A}} S' \subset B(o, r)+ \partial A$
    and hence
    $$\sum_{\substack{S'\in \mathcal{T}_i: \\ rad(T)\leq r \text{ and } S' \not \subset A}} |S'| \leq |B(o, r)+ \partial A|.$$
Moreover, we have
$\lim_{r \rightarrow 0}\frac{|B(o, r)+ \partial A|}{r}=|\partial A|.$ Therefore, for $r>0$ sufficiently small in terms of $A$ and $\varepsilon$
$$\sum_{\substack{S'\in \mathcal{T}_i: \\ rad(T)\leq r \text{ and } S' \not \subset A}} |S'| \leq |B(o, r)+ \partial A| \leq r|\partial A| \leq \varepsilon.$$
\end{proof}

Returning to the proof of the Lemma, we can combine \Cref{copyvolumebigradius} and \Cref{neighbourhod} to obtain that for some $r$ sufficiently small and some $i$ sufficiently large (both depending on $A, \alpha$), we have
$$\sum_{\substack{S'\in \mathcal{T}_i: \\ S' \not \subset A}} |S'| \leq \sum_{\substack{S'\in \mathcal{T}_i: \\ rad(T)\leq r \text{ and } S' \not \subset A}} |S'| + \sum_{\substack{S'\in \mathcal{T}_i: \\ rad(T)\geq r }} |S'| \leq \alpha |\co(A) \setminus A|.$$
We assumed here that $A$ is not convex as otherwise we are immediately done. We further get that
$$\sum_{S'\in \mathcal{T}_i} |S' \setminus A| \leq \alpha |\co(A) \setminus A|.$$
On the other hand, by \Cref{clm_size_b}, for all $S' \in \mathcal{S}_i$, we have
$\frac{|S'|}{|S' \cap A|} \in (1+\alpha, 1+n!\alpha).$

By construction, $\mathcal{T}_i \cup \mathcal{S}_i$ gives an essential partition into convex sets $\co(A)=P_1 \sqcup \dots \sqcup P_k$ that satisfy properties $1,2$ and $3$. Moreover, by the last two centered equations, the partition of the indices $[k]=I \sqcup J$ where $I=\{j \in [k] \colon P_j \in \mathcal{S}_i\}$ and $J=\{j \in [k] \colon P_j \in \mathcal{S}_i\}$ satisfies properties $4$ and $5$. The conclusion follows.
\end{proof}

\subsubsection{Proof of \Cref{LowDensityCaseInductionStep}}

\begin{proof}[Proof of \Cref{LowDensityCaseInductionStep}]
Choose parameters according to the following hierarchy:
$$n\gg \gamma\gg \alpha_0\geq \alpha\gg \zeta\gg \eta\gg \beta.$$
Here $\alpha_0\geq\alpha$ reflects that $\alpha$ is assumed to be sufficiently small in terms of $n$.
We proceed by contradiction, so that every subset $X\subset A$ which satisfies conclusions 1, 2, and 3  has $|\co(X)\setminus X|\leq \beta |X|$.

Let $v_0,\dots, v_n$ denote the vertices of $T=\co(A)$ and assume $T$ is centred, i.e., that its barycenter is the origin. Consider the slightly shrunk simplex $(1-\gamma)T$ with vertices $(1-\gamma)v_0,\dots, (1-\gamma)v_n$. Our first step will be to find points in $A\cap (1-\gamma)T$ close to each of these vertices. Indeed, consider the simplex $S_i:=\gamma T+(1-2\gamma)v_i$ which is a translate of $\gamma T$ whose $i$-th vertex coincides with $(1-\gamma)v_i$. Note that since $|\co(A)\setminus A|=\alpha |A|$ and $|\gamma T|=\gamma^n|T|\geq\gamma^n |A|$, there must be a point $u_i\in S_i\cap A$ for all $i$.

We iterate a similar construction to show that in fact we can find points considerably closer to the vertices $v_i$. Indeed consider the simplex $T_i:=\eta T+(1-\zeta-\eta)v_i$ which is a translate of $\eta T$ whose $i$-th vertex lies on $(1-\zeta)v_i$. Note that for $\eta$ sufficiently small in terms of $\zeta$, we find that $T_i\subset \co(u_0,\dots,u_{i-1},v_i,u_{i+1},\dots,u_n)$ for any points $u_j\in S_j$. Hence, if we consider $$A_i:=A\cap\co(u_0,\dots,u_{i-1},v_i,u_{i+1},\dots,u_n),$$ we find that $\co(A_i)=\co(u_0,\dots,u_{i-1},v_i,u_{i+1},\dots,u_n)$ is a simplex,  $diam(A_i)\leq (1-\gamma/n)diam(A)$, and $|\co(A_i)|\geq (1-\gamma)^{3n}|\co(A)|\geq (1-\gamma/n)^{3n^2}|\co(A)|$, so by our contradiction assumption, we have that $|\co(A_i)\setminus A_i|\leq \beta |A_i|$. Since $$T_i\subset \co(u_0,\dots,u_{i-1},v_i,u_{i+1},\dots,u_n)\quad\text{ and }\quad |T_i|=\eta^n |T|>\beta |A_i|$$ (assuming that $\beta$ is sufficiently small in terms of $\eta$), there exists $p_i\in T_i\cap A_i\subset A$. Let $A':=A\cap \co(p_0,\dots, p_n)$ so that $A'\subset (1-\zeta)T$ and $diam(A')\leq (1-\zeta) diam(A)$. Clearly, $\co(A')=\co(p_0,\dots, p_n)$ is a simplex. In the limit $\zeta,\eta\to0$, we have $|\co(A')|=(1-o(1))|\co(A)|$. Hence, for $\zeta,\eta$, and $\beta$ sufficiently small in terms of $\alpha$ and $n$ we find:
$$|\co(A)\setminus A|\leq |\co(A)\setminus \co(A')|+|\co(A')\setminus A'|<\alpha |A|,$$
a contradiction. This proves the result.
\end{proof}

\subsubsection{Proof of \Cref{SmallTubeLemma}}

\begin{proof}[Proof of \Cref{SmallTubeLemma}]
As $A,B$ form a $\lambda$-bounded $\eta$-sandwich, consider the convex $K\subset\mathbb{R}^n$, so that $$B(o, (3\lambda n)^{-1})\subset K\subset A,B\subset (1+\eta)K\subset B(o, 3\lambda n),$$ where the balls follow from \Cref{BoundedBallsObs} combined with $1+\eta\leq 3/2$.

We construct $V$ as follows. Let $H_1,\dots,H_n$ be the defining hyperplanes of $U$, and let $H_i^+$ be the corresponding halfspaces so that $U=\bigcap_i H_i^{+}$. We construct parallel hyperplanes $G_i$ as follows: Given $G_1,\dots, G_{i-1}$, find $G_i$ so that $|A\cap\bigcap_{j=1}^{i} H_j^+|=|B\cap \bigcap_{j=1}^{i} G_j^+|$. Note that this is possible since, by construction $|A\cap\bigcap_{j=1}^{i-1} H_j^+|=|B\cap\bigcap_{j=1}^{i-1} G_j^+|$ and $|B\cap\bigcap_{j=1}^{i} G_j^+|$ changes continuously with $G_j$.

We will show by induction that the distance between $G_i$ and $H_i$ tends to zero as $\eta$ tends to zero. By induction $|(H_j^+\triangle G_j^+)\cap B|\to 0$ as $\eta \to 0$ for $j<i$. This implies that also $$\bigg|\Big|(B\cap H_i^+)\cap \bigcap_{j=1}^{i-1} H_j^+\Big|-\Big|(B\cap H_i^+)\cap\bigcap_{j=1}^{i-1} G_j^+\Big|\bigg|\leq \bigg|(B\cap H_i^+)\cap \bigcup_{j=1}^{i-1} (H_j^+\triangle G_j^+)\bigg|\to 0\qquad \text{as $\eta \to 0$.}$$
In the same vein, we have that because $A,B$ is a $\lambda$-bounded $\eta$ sandwich 
$$\biggl||A\cap\bigcap_{j=1}^{i} H_j^+|-|B\cap\bigcap_{j=1}^{i} H_j^+|\biggr|\leq \biggl|(A\triangle B)\cap \bigcap_{j=1}^{i} H_j^+\biggr|\leq |(1+\eta)K\setminus K|\to 0\quad \text{ as }\eta \to 0. $$
Hence, using that $\left|(B\cap G_i^+) \cap\bigcap_{j=1}^{i-1} G_j^+\right|=|A\cap \bigcap_{j=1}^{i} H_j^+|$, we find
$$\bigg|\Big|(B\cap G_i^+) \cap\bigcap_{j=1}^{i-1} G_j^+\Big|-\Big|(B\cap H_i^+)\cap\bigcap_{j=1}^{i-1} G_j^+\Big|\bigg|=\bigg|\Big|A\cap \bigcap_{j=1}^{i} H_j^+\Big|-\Big|(B\cap H_i^+)\cap\bigcap_{j=1}^{i-1} G_j^+\Big|\bigg|\to 0 \quad \text{ as } \eta \to 0.$$
On the other hand, we find
$$\bigg|\Big|(B\cap G_i^+) \cap\bigcap_{j=1}^{i-1} G_j^+\Big|-\Big|(B\cap H_i^+)\cap\bigcap_{j=1}^{i-1} G_j^+\Big|\bigg|=\Big|B\cap (G_i^+\triangle H_i^+) \cap\bigcap_{j=1}^{i-1} G_j^+\Big|\geq \Big|(G_i^+\triangle H_i^+)\cap  B(o,(3\lambda n)^{-1})\cap\bigcap_{j=1}^{i-1} G_j^+\Big|.$$
By induction, we have $\Big|(G_i^+\triangle H_i^+)\cap  B(o,(3\lambda n)^{-1})\cap\bigcap_{j=1}^{i-1} G_j^+\Big|=\Big|(G_i^+\triangle H_i^+)\cap  B(o,(3\lambda n)^{-1})\cap\bigcap_{j=1}^{i-1} H_j^+\Big|+o(1)$, where $o(1)\to 0$ as $\eta \to 0$.
As $\epsilon$ is sufficiently small in terms of $\lambda$ and $n$, we find that 
$$\bigg|(G_i^+\triangle H_i^+)\cap  B(o,(3\lambda n)^{-1})\cap\bigcap_{j=1}^{i-1} H_j^+\bigg|=\Omega_{n}(\big|(G_i^+\triangle H_i^+)\cap  B(o,(3\lambda n)^{-1})\big|).$$
We conclude that $\big|(G_i^+\triangle H_i^+)\cap  B(o,(3\lambda n)^{-1})\big|\to 0$ as $\eta\to 0$, so that indeed the distance between $G_i$ and $H_i$ tends to zero as $\eta \to 0$.

Let $V:=\bigcap_{i=1}^n G_i^+$. Since the $G_i$'s are parallel to the $H_i$'s, we find that $V$ is homothetic to $U$, i.e., there are $x\in\mathbb{R}^n$ and $\beta\in\mathbb{R}$ so that $V=x+(1+\beta)U$. As the defining hyperplanes tend to each other as $\eta \to 0$, we find that for sufficiently small $\eta$, $||x||_2\leq \alpha$ and $|\beta|\leq \alpha$.

For the last part of the theorem, we proceed by induction to show that $A_i:=A\cap \bigcap_{j=1}^{i} H_j^+$ and $B_i:=B\cap \bigcap_{j=1}^{i} H_j^+$ satisfy $$|tA_i+(1-t)B_i|-|A_i|\leq |tA+(1-t)B|-|A|.$$ For $i=0$ this is vacuously true. Note that $A_{i}=A_{i-1}\cap H_i^+$ and $B_{i}=B_{i-1}\cap G_i^+$. By construction $|A_i|=|B_i|$ and $|A_{i-1}|=|B_{i-1}|$, so that we also have $A_i':= A_{i-1}\setminus A_i= A_{i-1} \cap H_i^-$ and the analogously defined $B_i'$ have the same volume. We find that $$tA_i+(1-t)B_i\subset tH_i^+ +(1-t)G_i^+\quad\text{ and }\quad tA'_i+(1-t)B'_i\subset tH_i^- +(1-t)G_i^-,$$ which are two halfspaces separated by the hyperplane $tH_i+(1-t)G_i$, so that $tA_i+(1-t)B_i$ and $tA'_i+(1-t)B'_i$ are disjoint subsets of $t A_{i-1} +(1-t)B_{i-1}$ and thus 
$$|tA_i+(1-t)B_i|+|tA'_i+(1-t)B'_i|\leq |t A_{i-1} +(1-t)B_{i-1}|.$$ 
By the Brunn-Minkowski inequality we know that $|tA'_i+(1-t)B'_i|\geq |A_i'|$, so we can conclude
\begin{align*}
|tA_i+(1-t)B_i|-|A_i|\leq |tA_i+(1-t)B_i|-|A_i|+|tA'_i+(1-t)B'_i|-|A_i'|&\leq |t A_{i-1} +(1-t)B_{i-1}|-|A_{i-1}|\\
&\leq |tA +(1-t)B|-|A|,
\end{align*}
where the last inequality follows from induction. Finally, $A_n=A^\circ$ and $B_n=B^\circ$, so the theorem follows.
\end{proof}

\subsubsection{Proof of \Cref{ChangeoftLemma}}

\begin{proof}[Proof of \Cref{ChangeoftLemma}] 
The first two conclusions are straightforward to checks so we turn our attention to the third conclusion. Note that by construction $t'A'+(1-t')B'=tA^\circ+(1-t)B^\circ$, so it is enough to show that
$$t'|A'|+(1-t')|B'|\geq |A^\circ|.$$
As $|A'|=(t/t')^n|A^\circ|$ and $|B'|=((1-t)/(1-t'))^n|B^\circ|$ and $|A^\circ|=|B^\circ|$, this is equivalent to 
$$t'(t/t')^n+(1-t')((1-t)/(1-t'))^n \geq 1,$$
which is true by the convexity of the function $\tau\mapsto \tau^n$. 
We turn our attention to the last conclusion.

Recall that $K\subset A,B\subset (1+\eta) K$ for some convex set $K$. Let $K'=\min\left\{\frac{t}{t'},\frac{1-t}{1-t'}\right\}K$, so that $K'\cap W\subset A'$ and $K'\cap W'\subset B'$. Similarly, let $K''=\max\left\{\frac{t}{t'},\frac{1-t}{1-t'}\right\}(1+\eta)K$, so that $K''\cap W\supset A'$ and $K''\cap W'\supset B'$. Note that for $\alpha$ sufficiently small in terms of $\eta$, we have $t$ and $t'$ sufficiently similar that $K''\subset(1+2\eta)K'$.

To deal with the translation, note that $K'$ contains a smaller homothetic copy centred at $\frac{1-t}{1-t'}x$. Indeed, let $p$ be the intersection of the ray $\mathbb{R}^+x$ with $\partial K'$. Since $K$ is $\lambda$-bounded, $K'$ is $1.5\lambda$-bounded, and so $||p||_2=\Omega_n(1/\lambda)$. Consider the homothety $H$ centred at $p$ with ratio $\frac{||p||_2-\frac{1-t}{1-t'}||x||_2}{||p||_2}=1+O_n(\alpha \lambda^{-1})$. Note that $H(o)=\frac{1-t}{1-t'}x$. Let $K''':=H(K')$. As $p'\in K'$, we find $K'''\subset K'$. Hence, $K'''-\frac{1-t}{1-t'}x\subset A',B'-\frac{1-t}{1-t'}x$.

Similarly, let $q$ be the point in the intersection between $\mathbb{R}^-x$ and $\partial (1+2\eta)K'$, so that $||q||_2=\Omega_n(1/\lambda)$. Consider the homothety $H'$ centred at $q$ with ratio $\frac{||q||_2+\frac{1-t}{1-t'}||x||_2}{||q||_2}=1+O_n(\alpha \lambda^{-1})$, so that $H'(o)=\frac{1-t}{1-t'}x$ and $H((1+2\eta)K')\supset (1+2\eta)K'$ is centred at $\frac{1-t}{1-t'}x$. Hence, we find $A',B'-\frac{1-t}{1-t'}x\subset \frac{||q||_2+\frac{1-t}{1-t'}||x||_2}{||q||_2} (1+2\eta)K'$.
Combining these two, and taking $\eta$ and $\alpha$ sufficiently small in terms of $\lambda$, $n$, and $\eta'$, we find that $A',B'-\frac{1-t}{1-t'}x$ form a tubular $2\lambda$-bounded $\eta'$-sandwich.

For the final conclusion on the symmetric difference, first note that for translating does not change the symmetric difference too much:
\begin{align*}
    |A^\circ \triangle (B^\circ+x)|&\leq |\co(A^\circ)\setminus A^\circ|+|\co(B^\circ+x)\setminus (B^\circ+x)|+|\co(A^\circ)\triangle\co(B^\circ+x)|\\
    &\leq 2\alpha|A^\circ|+|\co(A^\circ)\triangle(\co(B^\circ)+x)|\\
    &\leq 2\alpha|A^\circ|+ O_n(||x||_2)+|\co(A^\circ)\triangle\co(B^\circ)|\\
    &\leq (2\alpha+ O_{n,\epsilon}(||x||_2))|A^\circ|+|\co(A^\circ)\triangle\co(B^\circ)|\\
    &\leq O_{n,\epsilon}(\alpha|A^\circ|)+(|\co(A^\circ)\setminus A^\circ|+|\co(B^\circ)\setminus B^\circ|+|A^\circ\triangle B^\circ|)\\
    &\leq O_{n,\epsilon}(\alpha|A^\circ|)+|A^\circ\triangle B^\circ|.
\end{align*}
Hence, we may assume $x=0$. Similarly we find that the homotheties do not affect the symmetric difference too much. For simplicity assume $\beta>0$ (the case $\beta\leq 0$ follows analogously), so that $|A'\triangle B'|\leq |(1+\beta)B^\circ \triangle A^\circ|$. As above, it suffices to show that $|(1+\beta)\co(B^\circ) \triangle \co(A^\circ)|$ is small. Because the origin is in both $\co(A^\circ)$ and $\co(B^\circ)$ we find that $|\co(A^\circ)\setminus (1+\beta)\co(B^\circ)|\leq |\co(A^\circ)\setminus \co(B^\circ)|$. For the other term, note that
$$|(1+\beta)\co(B^\circ)\setminus \co(A^\circ)|\leq |(1+\beta)\co(B^\circ)\setminus \co(B^\circ)|+|\co(B^\circ)\setminus \co(A^\circ)|\leq O_n(\beta)|A^\circ|+|\co(B^\circ)\setminus \co(A^\circ)|.$$
Combining these two bounds, we find
$$|(1+\beta)\co(B^\circ) \triangle \co(A^\circ)|=|\co(A^\circ)\setminus (1+\beta)\co(B^\circ)|+|(1+\beta)\co(B^\circ)\setminus \co(A^\circ)|\leq O_n(\alpha)|A^\circ|+|\co(B^\circ)\triangle \co(A^\circ)|.$$
We conclude as above that
$|(1+\beta)B^\circ \triangle A^\circ|\leq O_n(\alpha)|A^\circ|+|\co(B^\circ)\triangle \co(A^\circ)|,$
which finishes the proof.
\end{proof}

\subsubsection{Proof of \Cref{extendingTubes}}

\begin{proof}[Proof of \Cref{extendingTubes}]
Note that by monotonicity it suffices to show that for points $p\in (\partial (1+\eta)K)\cap U$ and $x\in S\times\{0\}$, (where $S$ is the regular simplex from \Cref{tubedef}), we have $(1-t)x+tp\in K\cup (S\times \mathbb{R}_-)$.

Consider the angle between $px$ and $op$. Since $|ox|\leq \epsilon$ and $|op|\geq (2\lambda n)^{-1}$ by \Cref{BoundedBallsObs}, we find that for sufficiently small $\epsilon$, the angle between $px$ and $op$ is less than any $\alpha_{n,\lambda}$.

Let $q$ be the intersection between the line $xp$ and $\partial K$ that lies between $x$ and $p$. Similarly, let $p':=\frac{1}{1+\eta}p\in op\cap \partial K$. By the above, we have $\angle qpp'\leq \alpha_{n,\lambda}$.

Now we will show that $\angle op'q$ is bounded away from zero in terms of $n$ and $\lambda$. Let $p'_n$ and $q_n$ denote the $n$-th coordinates of $p'$ and $q$ respectively. Since both points lie in $U\setminus B(0,(2\lambda n)^{-1})$, we find that $p'_n,q_n\geq (2\lambda n)^{-1}-\epsilon>(3\lambda n)^{-1}$, so that the line extending $p'q$ does not intersect $B(o,(3\lambda n)^{-1})$ inside $U$. Since, that line intersects $K$ exactly in the segment $p'q$, it also does not intersect $B(o,(3\lambda n)^{-1})\subset K$ outside of $K$. Now note that for every line $\ell$ through any point $r\in B(o,2\lambda n)\setminus B(o,(3\lambda n)^{-1})$ that does not intersect $B(o,(3\lambda n)^{-1})$, we find that $\angle \ell,or\geq \arcsin((3\lambda n)^{-2})\geq (3\lambda n)^{-2}$. Hence, $\angle op'q\geq (3\lambda n)^{-2}$, and thus $$\angle p'qp=\angle op'q-\angle qpp'\geq (3\lambda n)^{-2}-\alpha_{n,\lambda}\geq (4\lambda n)^{-2}.$$
This implies that $|pq|=O_{n,\lambda}(|p'p|)= O_{n,\lambda}(\eta|op'|)$. Note that as $p'\in\partial K$ and $K\subset B(o,2\lambda n)$, we have $|op'|\leq2\lambda n$, so that $|pq|=O_{n,\lambda}(|p'p|)= O_{n,\lambda}(\eta)$. On the other hand, we have by the triangle inequality that $|xp|\geq |op|-|ox|\geq 2\lambda n-\epsilon\geq \lambda n$. Hence, $\frac{|pq|}{|qx|}=O_{n,\lambda}(\eta)$, so that choosing $\eta$ sufficiently small in terms of $n,\lambda$, and $t$, we find that $(1-t)x+tp\in xp\subset K$. The lemma follows.
\end{proof}

\section{Proof of general Linear Theorem (\Cref{LinearThmGeneral})}

\begin{proof}[Proof of \Cref{LinearThmGeneral}]
 Let $\epsilon=\epsilon_n$ sufficiently small that we can apply \Cref{CoAintubes} with $\lambda=2n^3$. Choose $\xi$ sufficiently small in terms of $n$ and $\epsilon$. Finally, choose $\theta=\theta_n$ and $\eta=\eta_{n,t}$ sufficiently small so that we can apply \Cref{CoAintubes} and \Cref{DensityDiamLemma}. Note that the choice for $\theta$ does not depend on $t$.

By \Cref{boundedsandwichreduction}, we may assume that $A,B$ form a $2n^3$-bounded $\eta$-sandwich, by picking $d_{n,t}$ sufficiently small. After re-normalization so that $|A|=|B|=1$, this implies by \Cref{BoundedBallsObs}, that $B(o,(4n^4)^{-1})\subset A,B\subset B(o,4n^4)$.

we will show the bound for $|\co(A)\setminus A|$. The bound on $|\co(B)\setminus B|$ follows analogously.

Apply \Cref{DensityDiamLemma} with parameter $\xi$ to $A$ to find $A_1,\dots, A_k$. Note that as $A\subset B(o,4n^4)$, we have $diam(\co(A))\leq 8n^4$, so that $diam(\co(A_i))\leq 8n^4 \xi$.

Consider a uniformly random direction $v\in S^{n-1}$ and consider the tube $U\subset\mathbb{R}^n$ of diameter $\epsilon$ in that direction. Note that a simplex of diameter $\epsilon$ contains a ball of diameter $\epsilon'$, where $\epsilon'$ depends only on $n$ as $\epsilon$ depends only on $n$. Hence, $U$ contains the circular $B(o,\epsilon')+\mathbb{R}v$. Fix one point $a_i$ in each of the $\co(A_i)$ and note that if $a_i\in B(o,\epsilon'-8n^4 \xi)+\mathbb{R}v$ then $A_i\subset U$. Choosing $\xi\leq \epsilon'/(16n^4)$, we find that the probability
$\PP(A_i\subset U)\geq\PP(a_i\in B(o,\epsilon'-8n^4 \xi)+\mathbb{R}v)$ is lower bounded in terms of $n$ for any $a_i\in A\subset B(o,4n^4)$. Say $\rho=\rho_n>0$ is so that $\PP(A_i\subset U)\geq \rho$.

Note that if $A_i\subset U$, then $\co(A\cap U)\setminus (A\cap U)\supset \co(A_i)\setminus A_i$. Hence, as the $\co(A_i)$ are disjoint, we find
$$\EE\Big[|\co(U\cap A)\setminus (U\cap A)|\Big]\geq \sum_i|\co(A_i) \setminus A_i|\cdot\PP(A_i\subset U)\geq \rho \sum_i|\co(A_i) \setminus A_i|\geq \rho \theta |\co(A)\setminus A|.$$
Combining this with the bound we have by \Cref{CoAintubes}, we find
$$|\co(A)\setminus A|\leq \frac{1}{\rho \theta} \EE\Big[|\co(U\cap A)\setminus (U\cap A)|\Big]\leq \frac{k_n^{\ref{CoAintubes}}}{\rho \theta} t^{-c^{\ref{CoAintubes}}n^8}\delta|A|.$$
Letting $k_n^{\ref{LinearThmGeneral}}=\frac{k_n^{\ref{CoAintubes}}}{\rho \theta}$ gives the theorem.
\end{proof}

\section{Proof of Symmetric difference vs Common convex hull result (\Cref{thm_int_convex})}
\subsection{Propositions}
The proof of these results separates into two parts, viz if $X,Y$ very similar, i.e., $|X\triangle Y|$ small, or if $X,Y$ very dissimilar, i.e., $|X\cap Y|$ small. The former is more difficult.

\begin{prop}\label{prop_int_convex}
There exists an absolute constants $c_n, d_n>0$ such that the following holds. Assume $\delta \in [0, d_n]$ and let $A, B \subset \mathbb{R}^n$ be convex sets
such that $|A\cap B| \geq (1-\delta)\max( |A|, |B|). $
Then $$|\co(A \cup B)| \leq (1+c_n\delta) \min(|A|,|B|).$$
\end{prop}

\begin{prop}\label{thm_small_intersection} There exist constants $c_n$ so that, given convex sets $X,Y\subset\mathbb{R}^n$, we have 
$$|\co(X\cup Y)|\leq c_n\frac{|X|\cdot|Y|}{|X\cap Y|}.$$
\end{prop}

\subsection{Proofs of Propositions}

\subsubsection{Proof of \Cref{prop_int_convex}}
For this proof we use a particular case of the main result from \cite{van2021sharp}, which asserts the following.

\begin{thm}\label{homobm}
For all $n\in\mathbb{N}$, there are computable constants $c_n^{\ref{homobm}},d_n^{\ref{homobm}}>0$, such that the following holds. Let $X\subset\mathbb{R}^n$ so that $\left|\frac{X+X}{2}\right|\leq (1+d_n^{\ref{homobm}})|X|$, then $$|\co(X)\setminus X|\leq c_n^{\ref{homobm}}\left|\frac{X+X}{2}\setminus X\right|.$$
\end{thm}
With this theorem in hand, the proposition follows quickly.

\begin{proof}[Proof of \Cref{prop_int_convex}]
Let $c_{n}^{\ref{homobm}}$ and $d_n^{\ref{homobm}}(t)$ be the output of \Cref{homobm}. Choose $c_n=10c_{n}^{\ref{homobm}}$. Choose $d_n\leq 10^{-1}d_n^{\ref{homobm}}(\frac12)$

By standard approximation arguments we can assume that $\co(A)$ and $\co(B)$ have a finite number of vertices. Construct the finite set 
$F=V(\co(A))\cup V(\co(B))$
and construct the measurable sets
$X= (A\cap B)\cup F.$
By construction $\co(A\cup B)=\co(X)$ and $|X| \leq \min(|A|,|B|).$ Therefore, it is enough to prove
$$|\co(X)\setminus X| \leq c_n \delta |X|= 10c_n^{\ref{homobm}} \delta |X|.$$
Given that $10\delta\leq d_n^{\ref{homobm}}(\frac12)$, by \Cref{homobm} applied to the set $X$, it suffices to prove that
$ \left|\frac12(X+X)\right| \leq (1+10\delta)|X|.$
Recall that by construction
\begin{equation*}
    \begin{split}
        \frac12(X+X) &=\frac12\bigg([(A\cap B) \cup V(\co(A)) \cup V(\co(B))] +[(A\cap B)\cup V(\co(A)) \cup V(\co(B))]\bigg)\\
        &\subset \frac12\bigg([(A\cap B) \cup V(\co(A))] +[A\cap B]\bigg) \bigcup \frac12\bigg([(A\cap B) \cup V(\co(B))] +[A\cap B]\bigg)\\
        &  \qquad \qquad \qquad \ \ \bigcup \frac12\bigg([V(\co(A)) \cup V(\co(B))] +[V(\co(A)) \cup V(\co(B))]\bigg)\\
        &\subset \frac12\bigg(A+A\bigg) \bigcup \frac12\bigg(B+B\bigg)\bigcup \frac12\bigg([V(\co(A)) \cup V(\co(B))] +[V(\co(A)) \cup V(\co(B))]\bigg)\\
        &\subset A\cup B\cup \frac12\big(V(\co(A))+V(\co(B))\big).
    \end{split}
\end{equation*}
Therefore, by hypothesis and as $|\frac12(V(\co(A))+V(\co(B)))|=0$, we get
$$\left|\frac12(X+X)\right| \leq |A\cup B| \leq 2\max(|A|,|B|)-|A\cap B| \leq (1+\delta)\max(|A|,|B|). $$
Also by hypothesis, we get
$ |X|=|A\cap B|\geq (1-\delta)\max(|A|,|B|).$
Combining the last two inequalities, we conclude
\begin{equation*}
    \begin{split}
\left|\frac12(X+X)\right|\leq  (1+\delta)\max(|A|,|B|) \leq (1+\delta)(1-\delta)^{-1}|X| \leq (1+10 \delta)|X|.
    \end{split}
\end{equation*}
The last inequality follows by noting $\delta \leq \frac12$ and $(1+x)(1-x)^{-1}\leq1+10x$ for $x\leq1/2$.

\end{proof}

\subsubsection{Proof of \Cref{thm_small_intersection}}

We first prove the result for axis aligned boxes and then reduce the general case to that case.

\begin{lem}\label{smallintersectionboxes}
For axis aligned $R,T\subset \mathbb{R}^n$, we have
$$|\co(R\cup T)|\leq 2^n \frac{|R|\cdot |T|}{|R\cap T|}$$
\end{lem}
\begin{proof}
Linearly transforming and translating if needed we may assume that $R$ is a translated unit cube and $T=[0,t_1]\times \dots\times [0,t_n]$.
Given this setup we find that 
$|R\cap T|\leq \prod_{i=1}^n \min\{t_i,1\}.$
On the other hand, 
$$|\co(R\cup T)|\leq |R+T|=\prod_{i=1}^n (t_i+1)\leq 2^n \prod_{i=1}^n \max\{t_i,1\}.$$
Combining these bounds with $|R|=1$ and $|T|=\prod_{i=1}^nt_i$, we get
$$|\co(R\cup T)|\leq 2^n \prod_{i=1}^n \max\{t_i,1\}=2^n \prod_{i=1}^n \frac{t_i}{\min\{t_i,1\}}\leq 2^n \frac{|R|\cdot |T|}{|R\cap T|},$$
which concludes the lemma.
\end{proof}
With this lemma in hand, the proof of \Cref{thm_small_intersection} is just a quick reduction.

\begin{proof}[Proof of \Cref{thm_small_intersection}]
Let $E$ and $F$ the John ellipsoids of $X$ and $Y$, respectively, so that 
$$X\subset E,\quad Y\subset F,\quad |E|\leq O_n(|X|),\quad\text{and}\quad |F|\leq O_n(|Y|).$$ Note that $|E\cap F|\geq |X\cap Y|$. Affinely transforming if needed we may assume $E$ is a ball. Rotating if necessary, we may assume the axes of symmetry of $F$ are the basis vectors $e_1,\dots, e_n$. Let $R$ and $T$ be the smallest axis aligned boxes containing $E$ and $F$ respectively, so that $|R|=O_n(|E|)=O_n(|X|)$ and $|T|=O_n(|F|)=O_n(|Y|)$. Now use $\co(X\cup Y)\subset \co(R\cup T)$ and $X\cap Y\subset R\cap T$, together with \Cref{smallintersectionboxes}, to find
$$|\co(X\cup Y)|\leq |\co(R\cup T)|\leq 2^n \frac{|R|\cdot |T|}{|R\cap T|}\leq O_n\left(\frac{|X|\cdot |Y|}{|X\cap Y|}\right).$$
\end{proof}

\subsection{Proof of Theorem \Cref{thm_int_convex}}

\begin{proof}[Proof of \Cref{thm_int_convex}]
Let $d_n^{\ref{prop_int_convex}},c_n^{\ref{prop_int_convex}},c_n^{\ref{thm_small_intersection}}$, be the constants from \Cref{prop_int_convex} and \Cref{thm_small_intersection}.
Let $c_n:= \max\{c_n^{\ref{prop_int_convex}},c_n^{\ref{thm_small_intersection}}/d_n^{\ref{prop_int_convex}}\}$.
Let $\delta=1-\frac{|X\cap Y|}{\max(|X|,|Y|)}$. Consider two cases; either $\delta\leq d_n^{\ref{prop_int_convex}} $ or $\delta> d_n^{\ref{prop_int_convex}} $.
In the former case, by \Cref{prop_int_convex} we have
$$\frac{|\co(X\cup Y)|}{\min(|X|,|Y|)}-1\leq c^{\ref{prop_int_convex}}_n\delta=c_n\left(1-\frac{|X\cap Y|}{\max(|X|,|Y|)}\right) \leq c_n\frac{|X\triangle Y|}{|X\cap Y|}.$$
In the latter case, we use 
$$d_n^{\ref{prop_int_convex}}< \delta\leq 1-\frac{|X\cap Y|}{\max(|X|,|Y|)}\leq \frac{|X\triangle Y|}{\max\{|X|,|Y|\}},$$
to find that $\max\{|X|,|Y|\}< |X\triangle Y|/d_n^{\ref{prop_int_convex}}$. Hence, by \Cref{thm_small_intersection} we find
$$\frac{|\co(X\cup Y)|}{\min\{|X|,|Y|\}}\leq c_n^{\ref{thm_small_intersection}}\frac{\max\{|X|,|Y|\}}{|X\cap Y|}\leq c_n\frac{|X\triangle Y|}{|X\cap Y|},$$
which concludes the proof.
\end{proof}

\section{Putting it all together: Proof of \Cref{main_thm_1}}

\begin{proof}[Proof of \Cref{main_thm_1}]
Choose $d_{n,t}^{\ref{main_thm_1}}$ sufficiently small in terms of $n$ and $t$ to make various statements throughout the proof true, e.g. to allow the applications of \Cref{main_thm_6}, \Cref{LinearThmGeneral}, and \Cref{prop_int_convex}.

By \Cref{main_thm_6}, we find that (after a translation)
$$|A\triangle B|\leq c_n^{\ref{main_thm_6}}t^{-1/2}\delta^{\frac12}|A|,\quad \text{ so that }\quad|A\cap B|\geq (1-c_n^{\ref{main_thm_6}}t^{-1/2}\delta^{\frac12})|A|.$$
By \Cref{LinearThmGeneral}, we find that 
$$|\co(A)\setminus A|+|\co(B)\setminus B|\leq t^{-c^{\ref{LinearThmGeneral}}n^8}\delta |A|.$$
Combining these two bounds and using that $d_{n,t}^{\ref{main_thm_1}}$ is small, we find
$|\co(A)\triangle \co(B)|\leq 2c_n^{\ref{main_thm_6}}t^{-1/2}\delta^{\frac12}|A|.$
Hence, using \Cref{prop_int_convex} and the fact that $|A|\leq \min\{|\co(A)|,|\co(B)|\}$, we have
\begin{align*}
|\co(A\cup B)|&=|\co(\co(A)\cup\co(B))|\leq \left(1+c_n^{\ref{prop_int_convex}}2c_n^{\ref{main_thm_6}}t^{-1/2}\delta^{\frac12}\right)\max\{|\co(A)|,|\co(B)|\}\\
&\leq \left(1+c_n^{\ref{prop_int_convex}}2c_n^{\ref{main_thm_6}}t^{-1/2}\delta^{\frac12}\right)\left(1+t^{-\ref{LinearThmGeneral}n^8}\delta\right)|A|\leq \left(1+c_n^{\ref{main_thm_1}}t^{-1/2}\delta^{1/2}\right)|A|.
\end{align*}
This concludes the proof of the theorem.
\end{proof}

\section{Open Problems}
It is natural to ask for the $t$-dependence in \Cref{LinearThmGeneral}. To this end we make the following conjecture.

\begin{conj}
\label{main_conj_2}
For all $n \in \mathbb{N}$ and $t \in (0,1/2]$, there are computable constants $c_n^{\ref{main_conj_2}},d^{\ref{main_conj_2}}_{n,t}>0$ such that the following holds. Assume $\delta \in [0,d^{\ref{main_conj_2}}_{n,t}]$, and assume $A,B\subset \mathbb{R}^n$ are measurable sets of equal volume, so that 
$$ |tA+(1-t)B| =(1+\delta)|A|.$$
Then
$$|\co(A)\setminus A|\le c_n^{\ref{main_conj_2}}t^{-1}\delta|A|\quad\text{ 
and }\quad |\co(B)\setminus B|\le c_n^{\ref{main_conj_2}}t^{-n+1}\delta|A|$$
\end{conj}

If true, in \Cref{main_conj_2}, the exponents of $\delta$ and $t$, prioritised in this order, are optimal. Indeed, to bound $|\co(A) \setminus A|$, take $B=[-1,1]^n$ and $A=[-1,1]^n \cup \{p\}$, where $p=(1+h,0,\dots,0)$. To bound $|\co(B)\setminus B|$, take $A=[-1,1]^n$ and $B=[-1,1]^n \cup \{p\}$, where $p=(1+h,0,\dots,0)$. In both cases $h\ll_{n,t}1$.

It is worth noting that \Cref{main_conj_2} was established in dimension two \cite{planarBM}. 

For doubling further from minimal, we recall the following conjecture from \cite{van2023locality}. 

\begin{conj}
There is an absolute constant $\Delta>0$ so that if $\delta<\Delta$ and $A\subset\mathbb{R}^n$ with $\left|\frac{A+A}{2}\right|\leq (1+\delta)|A|$ then there is some convex $K\subset\mathbb{R}^n$ with $|K\triangle A|\leq O_{\delta}(1)|A|$.
\end{conj}

Finally, we recall the following conjecture suggested by \cite{boroczky2022quantitative}.

\begin{conj}\label{PrekopaConj}
Let $t\in(0,1)$ $f,g,h\colon \mathbb{R}^n\to \mathbb{R}_{\geq 0}$ be measurable functions with the property that $h(t x+(1-t)y)\geq f(x)^t g(y)^{1-t}$ for all $x,y\in\mathbb{R}^n$ and $\int f=\int g=1$. If $\int h\leq 1+\delta$ with $\delta$ sufficiently small in terms of $t$ and $n$, then there exists a log-concave function $\ell\colon \mathbb{R}^n\to \mathbb{R}_{\geq 0}$ so that $\int |h-\ell|+|f-\ell|+|g-\ell|\leq O_{t,n}({\delta}^{1/2})$.
\end{conj}

\bibliographystyle{alpha}
\bibliography{references}

\end{document}